\documentclass[11pt]{article}
\usepackage{amsmath}
\usepackage{amssymb, yhmath}

\usepackage{enumerate}
\usepackage{listings}

\usepackage[pdftex]{graphicx}
\usepackage{epstopdf}

\usepackage[ruled]{algorithm2e}

\DeclareMathOperator{\sgn}{sgn}
\DeclareMathOperator{\Diag}{Diag}

\newcommand{\Thr}{\mathbb{T}}

\usepackage{color}
\definecolor{dkgreen}{rgb}{0,0.6,0}
\definecolor{gray}{rgb}{0.5,0.5,0.5}

\usepackage{authblk}

\date{}
\newtheorem{theorem}{Theorem}[section]

\newtheorem{lemma}[theorem]{Lemma}

\newtheorem{proposition}[theorem]{Proposition}

\newenvironment{proof}[1][Proof]{\begin{trivlist}\item[\hskip \labelsep {\bfseries #1.}]}{$\Box$\end{trivlist}}

\newcommand{\range}{\operatorname{range}}

\newcommand{\M}{Y}
\numberwithin{equation}{section}

\title{Compression Approaches for the Regularized Solutions of Linear Systems from Large-Scale 
Inverse Problems}

\author[1]{Sergey Voronin}
\author[2]{Dylan Mikesell}
\author[3]{Guust Nolet}
\affil[1]{Department of Applied Mathematics, University of Colorado, Boulder, CO 80309, USA}
\affil[2]{Department of Earth, Atmospheric and Planetary Sciences, Massachusetts Institute of Technology, Cambridge, MA 02139, USA}
\affil[3]{G\'{e}oazur, Universit\'{e} de Nice, 06560 Sophia Antipolis, France}

\date{\today}

\begin{document}
\bibliographystyle{plain}

\lstset{language=Matlab,
   keywords={break,case,catch,continue,else,elseif,end,for,function,
      global,if,otherwise,persistent,return,switch,try,while},
   basicstyle=\ttfamily,
   keywordstyle=\color{blue},
   commentstyle=\color{red},
   stringstyle=\color{dkgreen},
   numbers=left,
   numberstyle=\tiny\color{gray},
   stepnumber=1,
   numbersep=10pt,
   backgroundcolor=\color{white},
   tabsize=4,
   showspaces=false,
   showstringspaces=false}

\maketitle
\begin{abstract}
We introduce and compare new compression approaches to obtain regularized solutions of  
large linear systems which are commonly encountered in large scale inverse problems. 
We first describe how to approximate matrix vector operations with a large matrix through a 
sparser matrix with fewer nonzero elements, by borrowing from ideas used in wavelet image compression. 
Next, we describe and 
compare approaches based on the use of the low rank SVD, which can result in 
further size reductions. We describe how to obtain the approximate low rank SVD 
of the original matrix using the sparser wavelet compressed matrix. 
Some analytical results concerning the various methods are presented and the results of the proposed techniques 
are illustrated using both synthetic data and a very large linear system from a seismic 
tomography application, where we obtain significant compression 
gains with our methods, while still resolving the main features of the solutions.
\end{abstract}

\section{Introduction}
\label{sec:Introduction}
This paper describes practical approaches to obtain approximate but accurate regularized 
solutions to large linear systems arising from large scale inverse problems, without the need to load 
into memory the often very large original matrix used in the corresponding 
optimization problems.
Typically, such as in the case of the seismic tomography application which we mention here for illustration  
\cite{Simons.Loris.ea2011} 
(involving the reconstruction of seismic wave velocities in the Earth's interior with respect 
to a given spherically symmetric model), the physics calls for a solution of a  
linear system $Ax = \bar{b}$ with matrix $A \in \mathbb{R}^{m \times n}$ (often with $m \neq n$).
In practice, instead of the true right hand side $\bar{b}$, 
we are given the noisy right hand side $b = \bar{b} + \nu$, 
with $\nu$ being an unknown noise vector. 
The matrix $A$ can be very large 
and is likely to be ill-conditioned and exhibit fast nonlinear decay of singular values \cite{nolet08}.  
In order to obtain a solution given matrix $A$ and right hand side $b$, one often uses a 
derivative of Tikhonov regularization involving a regularization parameter 
$\lambda>0$ \cite{Tikhonov63}. In its classical form, this is simply the minimization problem:
\begin{equation}
\label{eq:tikhonov_min_eq}
\bar{x} = \arg\min_x \left\{ ||Ax - b||_2^2 + \lambda ||x||_2^2 \right\},
\end{equation}
which replaces the constrained system $Ax = b$ by the $\ell_2$ minimization of the model 
residual norm $\|Ax - b\|_2$, with a constraint on the $\ell_2$ norm of the model, 
controlled by the parameter $\lambda$. For large $\lambda$, $\bar{x}$ tends to be close to zero.
Regularization is necessary to counter the effects of ill-conditioning: the presence of small 
singular values in the matrix, which if left unaccounted for, blows up the norm of the solution 
and makes it very sensitive to data errors \cite{Calvetti2000423}. 
The latter part of this property is worth repeating as it is central to the ideas in this paper: 
small errors in the operator $A$ and the right hand side $b$ do not induce big 
changes in the regularized solution. The regularization in \eqref{eq:tikhonov_min_eq}  
is referred to as $\ell_2$ regularization, because it involves the minimization of the $\ell_2$ 
model norm. Other types of regularization are possible: for example, sparsity constrained 
regularization is also frequently used, including in geophysical applications 
\cite{Charlety2013}. In this paper, 
we discuss the application of our methods to $\ell_2$ regularization, 
as it is the most commonly used regularization.
However, the techniques apply also to other types of regularization and 
optimization techniques. The quadratic functional in \eqref{eq:tikhonov_min_eq} can be differentiated 
to yield the linear system for the regularized solution:
\begin{equation}
\label{eq:tikhonov_min_eq_system}
(A^T A + \lambda I) \bar{x} = A^T b.
\end{equation}
If the matrix $A$ is not too large, then there is no problem in solving this linear system with 
an iterative algorithm. 
A conjugate gradient or the LSQR algorithm \cite{Paige:1982:LAS:355984.355989} 
can be efficiently used for this purpose. Typically, 
we may wish to incorporate additional terms into the regularization, 
such as Laplacian smoothing \cite{nolet08}. 
In that case we solve instead:
\begin{equation}
\bar{x} = \arg\min_x \left\{ ||Ax - b||_2^2 + \lambda_1 ||x||_2^2 + \lambda_2 ||L x||_2^2 \right\},
\end{equation}
which can be solved through the linear system:
\begin{equation}
(A^T A + \lambda_1 I + \lambda_2 L^T L) \bar{x} = A^T b,
\end{equation}
or through the augmented least squares problem and its corresponding normal equations:
\begin{equation*}
\bar{x} = \arg\min_x \left\|
\begin{bmatrix}
A \\
\sqrt{\lambda_1} I \\
\sqrt{\lambda_2} L \\
\end{bmatrix} x - 
\begin{bmatrix}
b \\
0 \\
0 \\
\end{bmatrix}
\right\|_2^2 \implies 
\begin{bmatrix}
A \\
\sqrt{\lambda_1} I \\
\sqrt{\lambda_2} L \\
\end{bmatrix}^T
\begin{bmatrix}
A \\
\sqrt{\lambda_1} I \\
\sqrt{\lambda_2} L \\
\end{bmatrix}
\bar{x} = 
\begin{bmatrix}
A \\
\sqrt{\lambda_1} I \\
\sqrt{\lambda_2} L \\
\end{bmatrix}^T
\begin{bmatrix}
b \\
0 \\
0 \\
\end{bmatrix}.
\end{equation*}
As long as $A$ and $L$ can be applied to vectors, the solution can be obtained by a number 
of iterative algorithms. The problem occurs when $A$ is too large to load into memory. 
In the seismic tomography application we refer to 
\cite{debayle2004, Simons.Loris.ea2011, vanheijst99}, the matrix 
is several terabytes in size, so it may not be possible to load into memory in full, 
even on relatively large memory computer clusters. Thus, we must find ways to condense the matrix 
size using acceptable approximations which do not significantly alter the final regularized 
solutions. 

Many attempts at approximating matrices have been documented 
\cite{MarkovskyLowRank,ImprovingCURMatrixDecomps}. 
However, few attempts have been made to apply the approximations to regularization. 
One of the main papers which precedes ours is \cite{Lampe20122845}, where Krylov subspace 
approximations for Tikhonov regularization are discussed. In this paper, 
we discuss two different techniques: wavelet 
based approximations and low rank SVD (singular value decomposition). Our SVD techniques are 
especially effective when the matrix exhibits fast nonlinear decay of singular values. 
From our experiments, Krylov subspace dimensionality reduction techniques, 
while interesting and promising, 
tend to do worse when the decay of singular values of the matrix is fast. 
This is in contrast to the techniques 
we describe, which in such cases, do not significantly degrade the solution quality  
and lower the hardware requirements to obtain a solution. Even if $A$ is small enough 
that it can be loaded into memory, there may still be interest in the techniques we 
describe for gains of speed or to be able to solve several problems at once on one machine. 

\section{Organization of the Paper}
\label{sec:Organization}
We now briefly describe the organization of this paper. We assume that the reader is interested 
in obtaining regularized solutions to a system $Ax = b$, where $A \in \mathbb{R}^{m \times n}$ 
is as previously described: 
very large (perhaps more than a TB), with rapidly decaying singular values, 
and stored on the disk. In Section \ref{sec:Notation}, we describe notation and 
preliminary concepts including 
the various norms we use, the singular value decomposition, and a few lemmas that we 
use for our later derivations. In Section 
\ref{sec:WaveletCompression}, we describe  
how to do approximate matrix-vector operations with the matrix $A$, 
using a smaller matrix $M$ derived from $A$, via a wavelet thresholding based algorithm. The 
matrix $M$ is obtained from $A$ entirely on the disk. The big $A$ matrix is never required to be 
loaded into RAM. We assume that on output of this procedure, the matrix $M$, which is still 
large, but significantly smaller than $A$ (in memory size), can be loaded into RAM at least 
for a limited number of operations.
After $M$ is obtained, two options are available to the user: the regularization can be 
performed directly via $M$, or greater compression may be sought. In many cases, we assume 
that the latter will be true: the user would like to obtain a matrix small enough to use on 
their local machine. In Section \ref{sec:SVDCompression}, we describe how to 
compute and use the low rank SVD, which is known to provide an optimal 
(in terms of error in the Frobenius and spectral norms) 
rank $k$ approximation of the matrix. We mention how to compute such an approximation 
with a randomized algorithm, which uses a limited number of matrix vector operations with 
$M$ (or with $A$, if that is feasible). We introduce several different strategies 
which can be used. We show that several strategies are mathematically equivalent, but one may be 
preferred over others depending on the setup of the problem. 
Both in Section \ref{sec:WaveletCompression} 
and Section \ref{sec:SVDCompression}, we mention 
block matrix techniques, which are very useful for very large problems, where operating 
with the full matrices $A$ or $M$ is not possible. The outlined strategies make feasible 
to compute approximate regularized solutions to the original $Ax = b$ system, using matrices 
many times smaller than $A$, either with fewer nonzeros, in the case of the wavelet compressed $M$, or with much smaller dimensions, 
in the case of the low rank SVD. 
For some approaches, the matrices may be small enough 
to load on modern laptop computers, even if the original $A$ was more than a TB in size. 
In Section \ref{sec:Numerics}, we present numerical experiments to illustrate the techniques 
for the compression approaches outlined in Sections \ref{sec:WaveletCompression} 
and \ref{sec:SVDCompression}. 
We present results for both synthetic data, exhibiting different rates of 
decay of singular values and different wavelet compressibility characteristics, 
and for real data from a large scale seismic tomography application.

\section{Notation and Preliminaries}
\label{sec:Notation}

We refer to $x \in \mathbb{R}^n$ and $A \in \mathbb{R}^{m \times n}$, as respectively, 
a real valued vector of $n$ elements and a real 
valued matrix of $m$ rows and $n$ columns. Most of the techniques we describe 
apply to complex valued matrices also. For vectors, we define 
the vector norm as the usual Euclidean norm:
\begin{equation*}
\|x\|_2 = \left( \displaystyle\sum_{i=1}^n x_i^2 \right)^{\frac{1}{2}},
\end{equation*}
and we use the notation $\|x\|$ to mean $\|x\|_2$. For matrices, we define the spectral norm as:
\begin{equation*}
\|A\|_2 = \sigma_{\max}(A)
\end{equation*}
where $\sigma_{\max}(A)$ denotes the largest singular value of 
matrix $A$. The Frobenius norm is defined as:
\begin{equation*}
\|A\|_{F} = \left( \displaystyle\sum_{i,j} A_{i,j}^2 \right)^{\frac{1}{2}}.
\end{equation*}
By $A^{-1}$ we denote the inverse matrix, which is applicable only for square dimensions 
(i.e. $m = n$). The following result, which can be directly verified by means of block matrix inversion, 
is known as the Woodbury inverse formula \cite{Woodbury50} 
and will be useful in our analysis in Section \ref{sec:SVDCompression}: 
\begin{lemma}
Take $D \in \mathbb{R}^{n \times n}$, $P \in \mathbb{R}^{n \times k}$, $T \in \mathbb{R}^{k \times k}$, and 
$R \in \mathbb{R}^{k \times n}$. Assume that $D$ and $T$ are invertible. 
Then $D+PTR$ is invertible if and only if $T^{-1}+RD^{-1}P$ is, and the following identity holds:
\begin{equation}
\label{eq:woodbury_inverse_formula}
(D + P T R)^{-1} = D^{-1} - D^{-1} P \left( T^{-1} + R D^{-1} P \right)^{-1} R D^{-1}.
\end{equation}
\end{lemma}

Every matrix $A$ admits a 
\emph{singular value decomposition} (SVD) \cite{trefethen97} of the form
\begin{equation}
\label{eq:svdofA}
\begin{array}{ccccccccccc}
A &=& U & \Sigma & V^{T},\\
m\times n && m\times p & p\times p & p\times n
\end{array}
\end{equation}
where $p = \min(m,n)$ and $U$ and $V$ are orthonormal matrices and $\Sigma$ is a diagonal matrix.
The columns $(u_{j})_{j=1}^{p}$ and $(v_{j})_{j=1}^{p}$ of $U$ and $V$ are called
the left and right singular vectors of $A$, respectively, and the diagonal entries
$(\sigma_{j})_{j=1}^{p}$ of $\Sigma$ are the singular values of $A$. The singular
values of $A$ are ordered so that $\sigma_{1} \geq \sigma_{2} \geq \cdots \geq \sigma_{p} \geq 0$.
$U$ and $V$ have orthonormal columns ($U^T U =  V^T V = I_p$). 
\[
U = \bigl[u_{1}\ u_{2}\ \cdots\ u_{p}\bigr],
\qquad
V = \bigl[v_{1}\ v_{2}\ \cdots\ v_{p}\bigr],
\qquad\mbox{and}\qquad
\Sigma = \left[\begin{array}{cccc}
\sigma_{1} & 0 & 0 & \cdots \\
0 & \sigma_{2} & 0 & \cdots \\
0 & 0 & \sigma_{3} & \cdots \\
\vdots & \vdots & \vdots & \ddots 
\end{array}\right],
\]
so that 
\begin{equation*}
A = \sum_{j=1}^{p}\sigma_{j}\,u_{j}\,v_{j}^{T}.
\end{equation*}
In finite precision, the numerical rank of the matrix will be $r$ and it is possible 
(in fact, likely for a large matrix) that $r < p$. That is, $\sigma_{j}$ appears as $0$ 
to the machine for $j\geq r$. Thus, in such scenario we write:
\begin{equation*}
A = \sum_{j=1}^{r}\sigma_{j}\,u_{j}\,v_{j}^{T}.
\end{equation*}
where the precise value of $r$ is typically unknown. It is always the case that $r \leq p$.

For a matrix which is not well conditioned and has fast decay of singular values, 
many nonzero singular values $\sigma_j$ for $j<r$ 
will be very small relative to the largest singular value $\sigma_1$ and the
drop off in value starting from $\sigma_1$ will be rapid and nonlinear.
In these cases, the low rank SVD approximation $A_k$ provides a good approximation to the 
matrix for relatively small $k$ relative to $p$. We define $A_k$ 
by taking into account only the first $k < p$ singular values and vectors:
that is, with $U_k\in\mathbb{R}^{m \times k}$ consisting of the first $k$ columns of $U$, 
$\Sigma_k=\Diag(\sigma_1,\ldots,\sigma_k)\in\mathbb{R}^{k \times k}$ consisting of $k$ rows 
and columns of $\Sigma$, and
$V_k\in\mathbb{R}^{n \times k}$ consisting of the first $k$ columns of $V$:
\begin{equation}
\label{eq:svdofAtrunc}
A_{k} = \sum_{j=1}^{k}\sigma_{j}\,u_{j}\,v_{j}^{T} = U_{k}\,\Sigma_{k}\,V_{k}^{T},
\end{equation}
$$
U_{k} = \bigl[u_{1}\ u_{2}\ \cdots\ u_{k}\bigr],
\qquad
V_{k} = \bigl[v_{1}\ v_{2}\ \cdots\ v_{k}\bigr],
\qquad\mbox{and}\qquad
\Sigma_k = \left[\begin{array}{ccccc}
\sigma_{1} & 0 & 0 & \cdots & 0 \\
0 & \sigma_{2} & 0 & \cdots & 0 \\
0 & 0 & \sigma_{3} & \cdots & 0 \\
\vdots & \vdots & \vdots && 0 \\
0 & 0 & 0 & \cdots & \sigma_{k}
\end{array}\right].
$$
By the Eckart-Young theorem \cite[Theorem 5.8]{trefethen97}, it is known that $A_k$ is the optimal rank $k$ approximation 
to $A$ in both the spectral and Frobenius norms and that:
\[
\|A - A_{k}\|_2 = \sigma_{k+1},
\]
when the error is measured in the $\ell^{2}$ operator norm, and
$$
||A - A_{k}||_F = \left(\sum_{j=k+1}^{p}\sigma_{j}^{2}\right)^{1/2}
$$
in the Frobenius norm. When $k \ll p$, the matrices $U_k$, $\Sigma_k$, and $V_k$ 
are significantly smaller than the corresponding full SVD matrices $U$, $\Sigma$, and $V$. 
The choice of $k$ is up to the user, 
but greater $k$ requires greater computation time
and storage requirements. Notice that $A$ and $A_k$ are related via the expansion: 
\begin{equation*}
A = \displaystyle\sum_{i=1}^{k} \sigma_i u_i v_i^T + \displaystyle\sum_{i=k+1}^{r} \sigma_i u_i v_i^T 
\end{equation*}
where the first sum on the right corresponds to $A_k$ and the second sum corresponds 
to $\hat{A_k}$, consisting of the remaining singular vectors 
(in matrices $\hat{U}_k$, $\hat{V}_k$) which are not used in the truncated SVD expansion. These 
remaining singular vectors are orthogonal to the vectors in matrices $U_k$ and $V_k$ which go into 
the construction of $A_k$. We have the following relations for $k<r$:
\newpage
 
\begin{eqnarray*}
U &=& [U_k, \hat{U}_k] \quad \mbox{;} \quad V = [V_k, \hat{V}_k] ;
\\
A &=& \displaystyle\sum_{i=1}^{k} \sigma_i u_i v_i^T + \displaystyle\sum_{i=k+1}^{r} \sigma_i u_i v_i^T = U_k \Sigma_k V^T_k + \hat{U}_k \hat{\Sigma}_k \hat{V}_k^T = A_k + \hat{U}_k \hat{\Sigma}_k \hat{V}_k^T = A_k + \hat{A_k},
\\
A^T &=& \displaystyle\sum_{i=1}^{k} \sigma_i v_i u_i^T + \displaystyle\sum_{i=k+1}^{r} \sigma_i v_i u_i^T = V_k \Sigma_k U^T_k + \hat{V}_k \hat{\Sigma}_k \hat{U}_k^T = A_k^T + \hat{V}_k \hat{\Sigma}_k \hat{U}_k^T = A_k^T + \hat{A_k}^T,
\\
A^T A &=& \displaystyle\sum_{i=1}^{k} \sigma_i^2 v_i v_i^T + \displaystyle\sum_{i=k+1}^{r} \sigma_i^2 v_i v_i^T = V_k \Sigma^2_k V^T_k + \hat{V}_k \hat{\Sigma}^2_k \hat{V}^T_k = A^T_k A_k + \hat{V}_k \hat{\Sigma}^2_k \hat{V}^T_k = A^T_k A_k + \hat{A_k}^T \hat{A_k},
\end{eqnarray*}
where 
\begin{equation*}
A_k =  \displaystyle\sum_{i=1}^{k} \sigma_i u_i v_i^T = U_k \Sigma_k V^T_k \quad \mbox{and} \quad A^T_k A_k =  \displaystyle\sum_{i=1}^{k} \sigma_i^2 v_i v_i^T = V_k \Sigma^2_k V^T_k,
\end{equation*}
and $U_k^T \hat{U}_k = V_k^T \hat{V}_k = 0$ and $U_k^T U_k = \hat{U}_k^T \hat{U}_k = V_k^T V_k = \hat{V}_k^T \hat{V}_k = I$. Additionally, we have the following properties which we will exploit in 
Section \ref{sec:SVDCompression}:
\begin{lemma}
For vectors $v \in \mathbb{R}^k$ and $w \in \mathbb{R}^m$, $||U_k v||_2 = ||v||_2$ and $||U_k^T w||_2 \leq ||w||_2$. The same also holds for vectors $\bar{v} \in \mathbb{R}^k$ and $\bar{w} \in \mathbb{R}^n$ and matrices $V_k$ and $V_k^T$.
\end{lemma}
\begin{proof}
Note that 
\begin{equation*}
U U^T = I = [U_k, \hat{U}_k] \begin{bmatrix} U_k^T \\ \hat{U}_k^T \end{bmatrix} = U_k U_k^T + \hat{U}_k \hat{U}_k^T \implies U_k U_k^T = I - \hat{U}_k \hat{U}_k^T.
\end{equation*}
Thus:
\begin{eqnarray*}
&& ||U_k v||_2^2 = \langle U_k v, U_k v \rangle = \langle v, U_k^T U_k v \rangle = \langle v, v \rangle = ||v||_2^2,
 \\
&& ||U_k^T w||_2^2 = \langle U_k^T w, U_k^T w \rangle = \langle w, U_k U_k^T w \rangle = 
\langle w, (I - \hat{U}_k \hat{U}_k^T ) w \rangle = \langle w, w \rangle - \langle w, \hat{U}_k \hat{U}_k^T w \rangle \leq ||w||_2^2.
\end{eqnarray*}
The computations with $V_k$ and $V_k^T$ take similar form.
\end{proof}

\section{Approximate Matrix-Vector Operations with Wavelet Compression}
\label{sec:WaveletCompression}
Most iterative algorithms applicable to our discussion can be successfully implemented if we can perform the two key operations with the matrix $A$:
\begin{equation}
\label{eq:mat_vec_ops}
A x \quad \mbox{and} \quad A^T y,
\end{equation} 
where $A \in \mathbb{R}^{m \times n}$, $x \in \mathbb{R}^n$ and $y \in \mathbb{R}^m$. We now  
discuss a technique to perform these operations approximately, using a smaller matrix derived 
from $A$ by means of wavelet compression \cite{DaubechiesWaveletsI,har-etal:wavelets}. 
Wavelets provide a multi-resolution approach to signal analysis, capturing  
the fine and coarse scale parts of a signal, and wavelet 
transforms can be performed efficiently \cite{Akansu:1992:MSD:573878,swe:spie95}. 
In our application, the matrix rows have features which are well represented by wavelets.
To motivate this approach, consider wavelet compression applied to a geophysical model 
(or any typical vectorized image). We compare the original model $x$ 
(in row vector form) to the inverse transform of the thresholded wavelet transformed 
model based on the relation:
\begin{equation}
\label{eq:wavelet_approx_for_x}
x \approx \left(W^{-1} \left( \Thr(W x^T) \right)\right)^T,
\end{equation}
where $W$ and $W^{-1}$ represent the forward and inverse wavelet 
transforms \cite{YvesMeyerWaveletsandAlgs} and the thresholding operation $\Thr(\cdot)$ 
retains a certain percentage of the largest coefficients (by absolute value) of its input vector. 
The transpose operations assure that we are applying the transforms to column vectors, in view 
of their representation as matrices $W$ and $W^{-1}$.
Relation \eqref{eq:wavelet_approx_for_x} holds when the row vector $x$ is wavelet compressible. 
This is not necessarily the case for arbitrary $x$, yet does hold in many situations. 
For example, in the case of the 
application we allude to in this paper, the vectors are geophysical kernels representing a 
sensitivity of the observable (usually a phase or a delay) with respect to the intrinsic 
velocity as a function of space \cite{GJI:GJI426}. 
These kernels arise from integral equations and are generally smooth, 
and have been observed by us to be compressible by imposing a threshold on the wavelet coefficients.
Many different kinds of thresholding functions exist. For our purposes, we simply use the hard 
thresholding function:
\begin{equation}
\label{eq:hard_thresholding}
   H_\alpha(x) = 
   \begin{cases}
      x &\mbox{    if    } |x|>\alpha,
   \\ 
      0 &\mbox{    if    } |x|\leq\alpha. 
   \end{cases}
\end{equation}
With the right choice of wavelet transform, only a small fraction 
of the coefficients in the wavelet transformed representation $W x$ need to be retained for a good 
reconstruction. That is, the threshold $\alpha$ can be taken to be quite large relative 
to the magnitudes of the elements of the vector $W x$. 
In Figure \ref{fig:fig_wavelet_compression}, below, a smooth CDF $9-7$ transform was used 
\cite{CDFWavelets}. We compare the original row vectorized image $x$ to the reconstructed image 
$\left(W^{-1} \Thr(W x^T)\right)^T$ using a $2D$ CDF $9-7$ transform over the image.
We observe that as the amount of retained nonzero wavelet coefficients decreases, 
the reconstruction quality worsens, but the main features of the image are still retained. 
In the rightmost plot of Figure \ref{fig:fig_wavelet_compression}, we 
define $E = 100 \frac{\| x - \left(W^{-1} \Thr(W x^T)\right)^T \|}{\|x\|}$ as the percent 
error and $N = 100 \frac{nnz(\Thr(W x^T))}{nnz(W x^T)}$ as the percent coefficients retained. 
Clearly, the reconstruction error can be controlled by keeping a certain (typically small) number of 
nonzero coefficients. Notice also that at about $7\%$ coefficients retained, we have a 
substantial $30\%$ error $E$. Yet, the image looks quite recognizable to the eye, with a bit 
of smoothing compared to the original.
\begin{figure*}[!ht]
\centerline{
\includegraphics[scale=0.18]{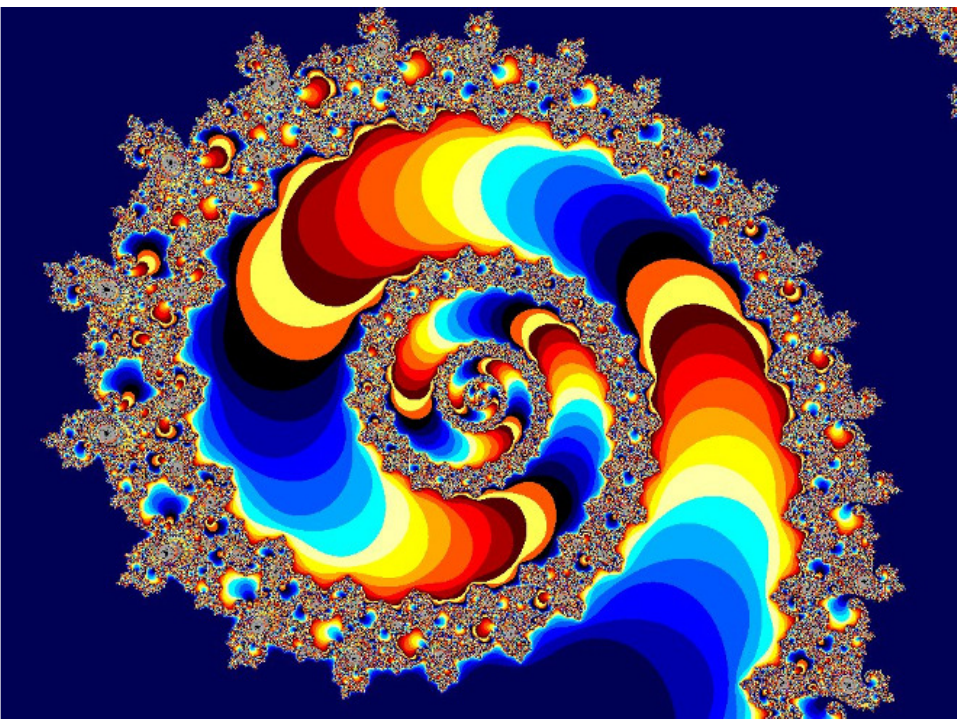}
\quad
\includegraphics[scale=0.18]{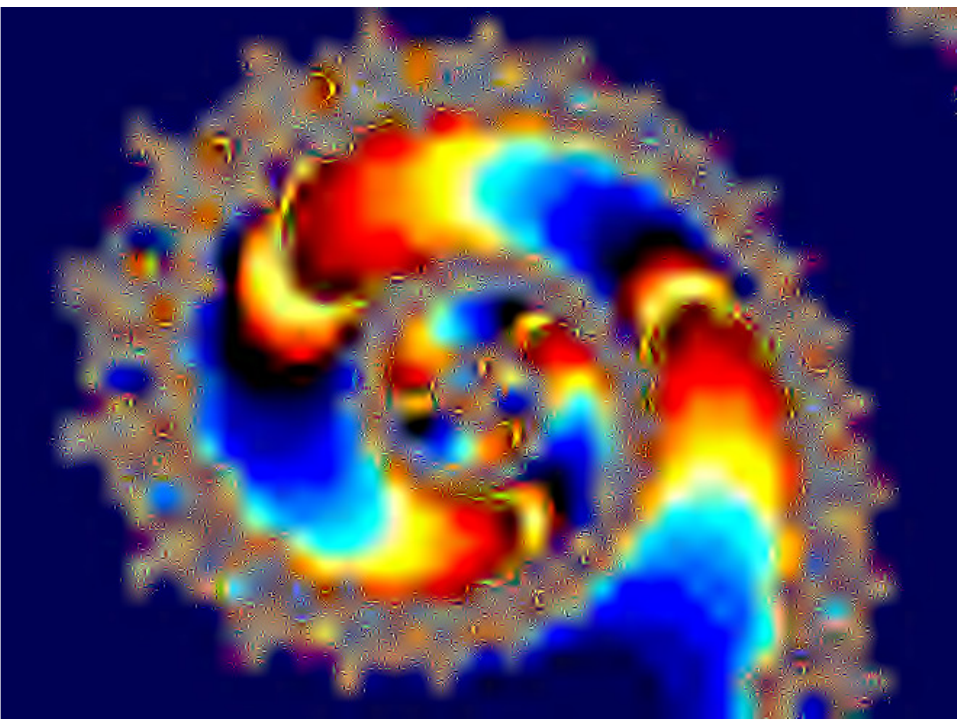}
\includegraphics[scale=0.18]{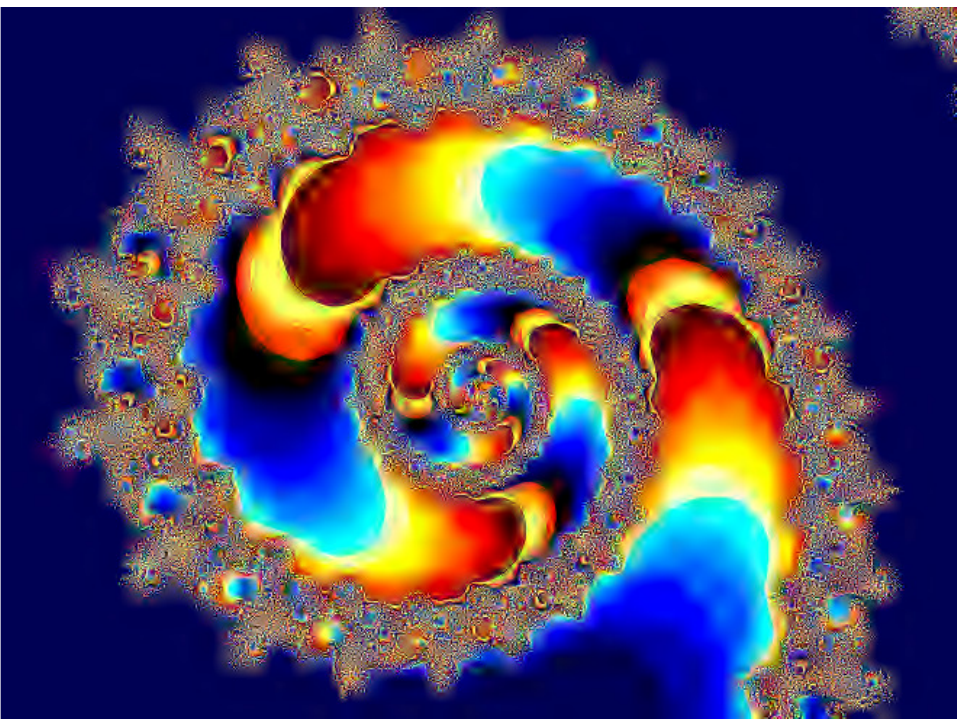}
\includegraphics[scale=0.16]{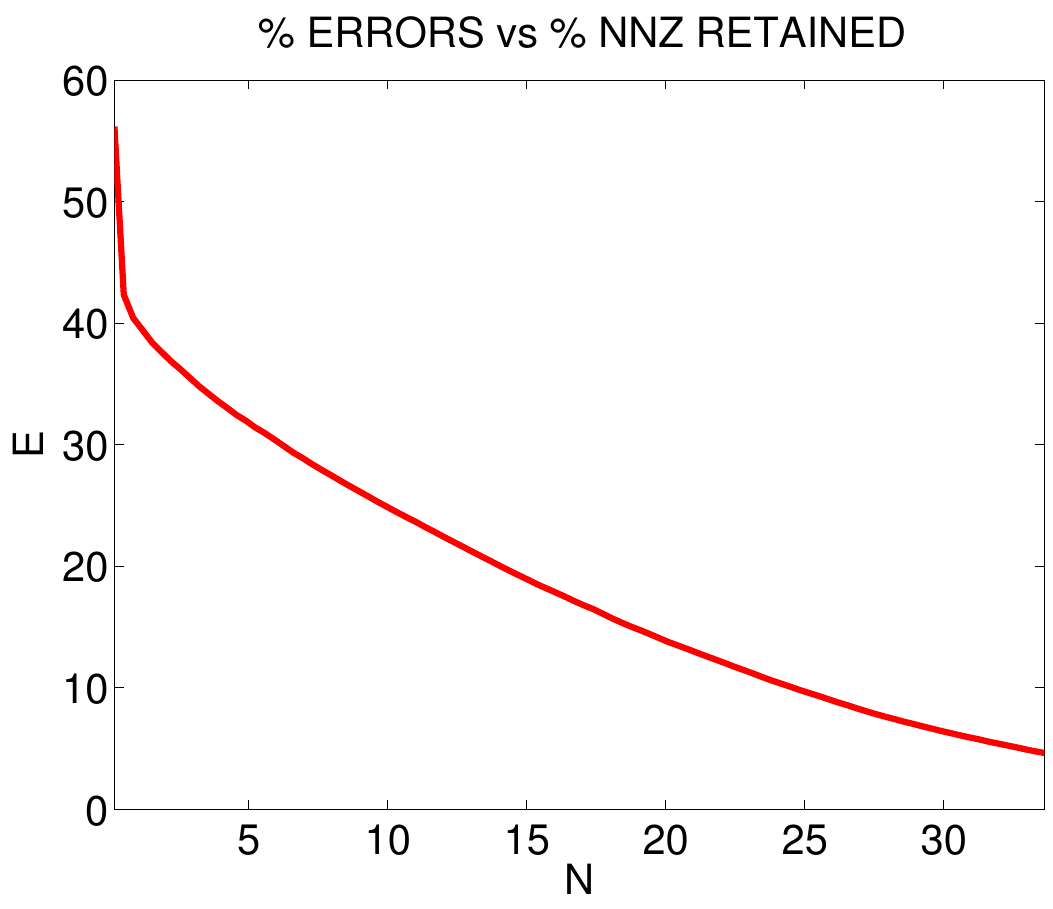}
}
\caption{A fractal image $x$ (left) and reconstructions 
$\left(W^{-1} \left( \Thr(W x^T) \right)\right)^T$ with $1.4\%$ and $6.8\%$ 
of retained wavelet coefficients. Plot of percent error norm vs percent nonzeros retained.}
\label{fig:fig_wavelet_compression}
\end{figure*}

Assuming the rows of our matrix $A$ are wavelet compressible (that is for some relatively 
small threshold, satisfy approximately the relation \eqref{eq:wavelet_approx_for_x}), 
we would like to apply the 
same principle to approximate matrix vector operations \eqref{eq:mat_vec_ops} with the 
big original matrix $A$ through a smaller matrix $M$ so that only the smaller matrix $M$ needs to be 
loaded into memory. The matrix $M$ will have the same dimensions as $A$ but fewer nonzeros, so 
it takes less space on disk and in memory. One forms this matrix by transforming and thresholding 
the individual rows of $A$, an operation which can be done entirely on the disk, without loading 
any parts of $A$ into RAM. The transform $W$ used for each row can vary from application to 
application, depending on the structure of the rows of $A$. In our seismic 
tomography application for which we give examples in Section \ref{sec:Numerics}, we simply used the $1D$ 
CDF $9-7$ transform for each row disregarding their inherent multi-dimensional structure. We 
believe that even better results can be obtained by tailoring $W$ to the structure of the matrix data.  

Each row of $M$ is obtained by applying the wavelet transform and thresholding to the 
corresponding row of $A$:
\begin{equation*}
A = \begin{bmatrix}
r_1 \\
r_2 \\
\vdots \\
r_m \\
\end{bmatrix} 
\ \rightarrow\  
M = 
\begin{bmatrix}
\Thr( W r_1^T )^T \\
\Thr( W r_2^T )^T \\
\vdots \\
\Thr( W r_m^T )^T \\
\end{bmatrix} = \Thr( A W^T ) \approx A W^T 
\end{equation*}
We can then approximate the operations \eqref{eq:mat_vec_ops}. Using the relations:
\begin{equation*}
Mx \approx A W^T x \quad \mbox{and} \quad M^T y \approx (A W^T)^T y = W A^T y, 
\end{equation*}
we obtain the approximation formulas:
\begin{equation}
\label{eq:wavelet_approx_matvec_ops}
Ax \approx M W^{-T} x \quad \mbox{and} \quad A^T y \approx W^{-1} M^T y.
\end{equation}
This means that the operations \eqref{eq:mat_vec_ops} can be performed approximately 
via \eqref{eq:wavelet_approx_matvec_ops}, using the smaller matrix $M$ and the inverse 
and inverse-transpose wavelet transforms. In practice, 
only $M$ needs to be loaded in memory as the wavelet transforms would be implemented as 
routines. The inverse-transpose transform 
is equivalent to the forward transform when $W$ is orthogonal and $W^{-1} = W^T$. For the non-orthogonal 
case, such as for example the CDF $9-7$ transform, the inverse-transpose transform can be  
approximated by applying the forward transform with the inverse filters. 
The success of this approximation method 
depends on the size ratio between $M$ and $A$ and the percent error in the approximate operations. 
This depends on the data, the transform that is used, and the threshold used in the thresholding 
function. Typically, we identify the threshold $\alpha$ in \eqref{eq:hard_thresholding} as follows. 
The input is sorted by putting 
the entries with largest absolute magnitude in front. Then a threshold is identified by putting the 
marker at some point of the nonzero entries (for example at the largest $15\%$ mark of the total 
nonzeros). Then all the entries with absolute magnitude less than the identified threshold 
are zeroed out. The percent error in the approximate
operations then depends on the percent error in the reconstruction of each row. That is, if 
for an arbitrary row $r$, $\left(W^{-1} \Thr( W r^T )\right)^T$ is not close to $r$, 
then the approximate operations 
using $M$ formed with this threshold will probably not be accurate. A less aggressive threshold 
then needs to be used.
Later we give examples for synthetic data and our seismic tomography application. 
For our application, we have observed that one can expect $M$ 
to be at least $3$ times smaller in memory requirements than 
$A$ without incurring significant errors in the operations $A x$, $A^T y$, and $A^T A x$. 

If $A$ is very large, the matrix $M$ may still be too big to load directly into memory. In that 
case, we may consider splitting the matrix in parts along its rows, with the matrix vector 
operations applied blockwise:
\begin{equation*}
A = \begin{bmatrix}
A_1 \\
A_2 \\
\vdots \\
A_p \\
\end{bmatrix}\ \implies \ 
Ax = 
\begin{bmatrix}
A_1 x\\
A_2 x\\
\vdots \\
A_p x\\
\end{bmatrix} \quad \mbox{and} \quad
A^T y = 
\begin{bmatrix}
A_1\\
A_2\\
\vdots \\
A_p\\
\end{bmatrix}^T 
\begin{bmatrix}
y_1\\
y_2\\
\vdots \\
y_p\\
\end{bmatrix} = \displaystyle\sum_{j=1}^p A_j^T y_j.
\end{equation*} 
Next, we can apply the wavelet compressed technique to the block matrices. We can proceed to form the 
matrices $M_1 = \Thr(A_1 W_1^T), \dots, M_p = \Thr(A_p W_p^T)$, which 
are smaller wavelet thresholded versions of the original blocks $A_1, \dots, A_p$. 
We can then perform approximate operations using these new sparser blocks:
\begin{eqnarray}
\label{eq:block_wavelet_operations}
&& A = 
\begin{bmatrix}
A_1 \\
A_2 \\
\vdots \\
A_p \\
\end{bmatrix}
\rightarrow 
M = 
\begin{bmatrix}
\Thr(A_1 W_1^T) \\
\Thr(A_2 W_2^T) \\
\vdots \\
\Thr(A_p W_p^T) \\
\end{bmatrix}
\implies
Ax \approx
\begin{bmatrix}
M_1 W_1^{-T} x \\
M_2 W_2^{-T} x \\
\vdots \\
M_p W_p^{-T} x \\
\end{bmatrix} \\
\nonumber
&& \mbox{and} \quad 
A^T y \approx \displaystyle\sum_{j=1}^p W_j^{-1} M_j^T y_j.
\end{eqnarray}
In the above formulas, we have used different transform matrices 
$W_1, \dots, W_p$ for the different 
blocks. This may provide an advantage when the data in the matrix can be grouped. For example, 
some groups may have mostly smooth and others may have mostly sharp features. In such a case, it 
may be advantageous to use different transforms (ex, smooth CDF wavelet or sharper Haar wavelet) 
on the different blocks. If this is not the case, the same transform can be used 
for each block so that $W_1 = \dots = W_p = W$. 

Let us now discuss the application of these ideas to \eqref{eq:tikhonov_min_eq_system}.
Plugging in the approximated matrix-vector operations we obtain:
\begin{equation*}
(W^{-1} M^T M W^{-T} + \lambda I) \tilde{x}_{w} = W^{-1} M^T b.
\end{equation*}
where $\tilde{x}_{w}$ will be the approximation to $\bar{x}$ in \eqref{eq:tikhonov_min_eq_system}.
If $A$ is so large that after forming $M$ we still cannot load $M$ into memory, then  
$M$ would be split into blocks $M_1, \dots, M_p$. No matter how large $A$ is, we 
can always choose $p$ large enough so that the individual blocks $M_j$ are manageable in size and 
can be loaded into RAM. In that case, we can still 
do operations in blocked form via \eqref{eq:block_wavelet_operations} by loading as many parts of 
$M$ as we can into memory, performing part of the operation and then replacing 
the in-memory blocks with the remaining blocks of $M$ to perform the rest. 
As long as fast disks (such as SSDs) are available, this is viable in practice, but may be very slow if 
many operations are needed. In the case that $M$ is too large to be loaded in full, the techniques 
discussed in the following section can be used to obtain further size reductions.

\section{Low Rank SVD Approximation}
\label{sec:SVDCompression}
The wavelet approximation techniques for matrix-vector operations discussed in the previous 
section enable us to approximate the operations \eqref{eq:mat_vec_ops} through a matrix 
several times smaller than $A$. However, in practice, the matrix $M$ can still be quite 
big if $A$ is particularly large. It is plausible that we can do some operations with $A$ 
through $M$ but only for a relatively short amount of time (perhaps through the 
blocked form \ref{eq:block_wavelet_operations}). Assuming that we can indeed do a limited number 
of matrix vector multiplications with $A$ through $M$, we now discuss other techniques for 
compression based on the low rank singular value decomposition (SVD). 
Once such a decomposition is obtained through a limited amount of matrix vector multiplications 
with $A$ (approximated through $M$), we can
obtain approximate forms of regularization algorithms which require the use of 
significantly smaller matrices.

\vspace{5.mm}
\subsection{Computation with Randomized Algorithm}
We now discuss how a rank $k$ low rank SVD approximation can be computed. 
One direct way is to compute it from the 
full SVD of the matrix. Given the full SVD $A = U \Sigma V^T$ one can take the first $k$ 
columns of $U$ and $V$ to be the matrices $U_k$ and $V_k$ and the first $k$ diagonal 
elements of $\Sigma$ to form $\Sigma_k$. For large matrices, this is not practical 
since the computation of the full SVD is prohibitively expensive (the cost for an 
$m \times n$ matrix is on the order of $\mathcal{O}(m n \min(m,n))$ operations \cite{trefethen97}).
The algorithm which we use is an adaptation of the method proposed in 
\cite{Halko:2011:FSR:2078879.2078881}. The cost of the proposed randomized algorithm 
for the rank $k$ SVD approximation is substantially lower (the cost is $\mathcal{O}(m n k)$ operations).

The randomized algorithm finding a rank $k$ approximation of 
$A \in \mathbb{R}^{m \times n}$ proposed in 
\cite{Halko:2011:FSR:2078879.2078881}
consists of several simple steps. The main idea is to obtain a good estimate for the range of 
$A$ by forming products of $A$ with a sample of random vectors, then using the orthogonal basis of this sample matrix to project the original matrix into a smaller, lower dimensional one, 
of which we extract the full SVD and use these components to construct the low rank SVD 
of the original big matrix $A$. The steps are as follows:
\begin{itemize}
\item Take $k$ samples of the range of matrix $A$ by multiplying $A$ with random Gaussian vectors 
to form sample matrix $Y$ of size $m \times k$. We then have $\range Y \approx \range A$.
\item Obtain an orthogonal matrix $Q$ from $Y$ 
(by e.g. performing QR factorization on $Y$ to get $Y=QR$, 
where $Q^TQ=I$ and $R$ is upper triangular). 
Then $\range Q \approx \range A \implies Q Q^T A \approx A$.
\item Project the original matrix into a lower dimensional one: $B = Q^T A$ where $B$ is $k \times n$, substantially smaller than $A$ which is $m \times n$.
\item Take the SVD of the smaller matrix $B = \tilde{U}_k \Sigma_k V^T_k$.
\item Take as low rank SVD of $A$ the product $U_k \Sigma_k V^T_k$ with $U_k = Q \tilde{U}_k$ (since 
$Q Q^T A \approx A$).
\end{itemize}
Various interpretations of these steps from \cite{Halko:2011:FSR:2078879.2078881},
including description of developed open source 
software can be found in \cite{2015arXiv150205366V}. We 
describe here the details of one particular approach mentioned in \cite{2015arXiv150205366V}, 
and formulate it in a way which can be used for very large matrices. 
In the approach we use, we construct a smaller matrix $B B^T$ and work with this matrix 
instead of $B$, because the matrix $B$ of size $k \times n$, can still be quite large for large $n$.
We compute the SVD components $\tilde{U}_k$ and $\Sigma_k$ of $B$ using the eigendecomposition 
of the small $k \times k$ symmetric matrix $B B^T$ and obtain $V_k$ by applying $B^T$.
This way, we avoid building $B$ or taking the SVD of it directly. We use the following relations:
\begin{eqnarray*}
&& 
   B 
   = \tilde{U}_k \Sigma_k V^T_k 
   = \displaystyle\sum_{i=1}^k \sigma_i \tilde{u}_i v_i^T 
 \quad \mbox{;} \quad 
   B^T 
   = V_k \Sigma_k \tilde{U}_k^T 
 \quad \mbox{;} \quad
   B v_i 
   = \sigma_i \tilde{u}_i;
\\
&& 
   B B^T 
   = \left( \displaystyle\sum_{i=1}^k 
         \sigma_i \tilde{u}_i v^T_i 
     \right) 
     \left( \displaystyle\sum_{j=1}^k 
         \sigma_j \tilde{u}_j v^T_j 
     \right)^T 
   = \displaystyle\sum_{i,j=1}^k 
        \sigma_i \sigma_j \tilde{u}_i v_i^T v_j \tilde{u}_j^T 
   = \displaystyle\sum_{i=1}^k \sigma_i^2 \tilde{u}_i \tilde{u}_i^T 
   = \tilde{U}_k D_k \tilde{U}_k^T .
\end{eqnarray*}
This means the eigendecomposition of the $k \times k$ matrix $B B^T$ gives us the low rank 
SVD components $U_k = Q \tilde{U}_k$ and $\Sigma_k = \sqrt{D_k}$ element-wise.
To compute the right eigenvectors $v_i$, we can use the following relations:
\begin{equation*}
B^T \tilde{U}_k = V_k \Sigma_k \tilde{U}_k^T \tilde{U}_k = V_k \Sigma_k \implies B^T \tilde{U}_k \Sigma_k^{-1} = V_k ,
\end{equation*}
which implies: 
\begin{equation}
\label{eq:svd_computation_for_eigenvectorsV}
v_i =  V_k e_i = (B^T \tilde{U}_k \Sigma_k^{-1}) e_i = \tfrac{1}{\sigma_i} B^T \tilde{u}_i = \tfrac{1}{\sigma_i} A^T Q \tilde{u}_i,
\end{equation} 
assuming all the singular values in $\Sigma_k$ are above zero 
(which is the case for $k$ smaller than the numerical rank $r$).
In practice, a slight oversampling often improves the approximation. For an approximation 
of rank $k$, $k+p$ samples can be used with $p$ a small number like $10$. Other techniques 
like the power sampling scheme also improve the approximation and are described in 
more detail in \cite{2015arXiv150205366V}.

Notice that all matrix-vector operations involving $A$ and $A^T$ can be approximated via the 
wavelet compressed matrices $M$ and $M^T$. To build up $B B^T$ column by column we can use 
matrix-vector products with standard basis vectors $e_j$:
\begin{equation}
\label{eq:svd_mat_mult_approx}
B B^T e_j = Q^T A A^T Q e_j \approx Q^T M W^{-T} W^{-1} M^T Q e_j,
\end{equation} 
and for the right eigenvectors, we have from \eqref{eq:svd_computation_for_eigenvectorsV} that:
\begin{equation*}
v_i = \tfrac{1}{\sigma_i} A^T Q u_i \approx \tfrac{1}{\sigma_i} W^{-1} M^T Q u_i .
\end{equation*}
We now illustrate the main steps of the random algorithm to compute the 
low rank SVD, which we use in our computations for the numerical experiments. Below, 
we use Matlab like pseudocode.
\begin{itemize}
\item Take $l = k+p$ samples of matrix $A$ (where $p$ is a small oversampling number) 
with random Gaussian vectors and perform Gram-Schmidt 
orthogonalization to calculate the projection matrix $Q$.
\begin{lstlisting}
for j=1:l
    rj = randn(n,1);
    yj = A*rj;
    Y(:,j) = yj;
end

Q = Y;
for ind=1:2
    for j=1:l
        vj = Q(:,j);
        for i=1:(j-1)
            vi = Q(:,i);
            vj = vj - project_vec(vj,vi);
        end
        vj = vj/norm(vj);
        Q(:,j) = vj;
    end
end
\end{lstlisting}
where the projection of $v$ in direction of $u$ is defined as $\frac{(v \cdot u)}{||u||_2^2} u$.
For best results, the Gram-Schmidt orthogonalization should be performed twice to account for 
loss of orthogonality. Note that for matrix-vector multiplications with $A$ we use 
$A r_i \approx M W^{-T} r_i$.

\item Build the $l \times l$ matrix $B B^T = Q^T A A^T Q$ by computing $k$ matrix-vector 
products with standard basis vectors.

Once we have built $Q$ and its transpose, we can form the matrix $B B^T$ column by column:
\begin{lstlisting}
BBt = zeros(l,l);
for j=1:l
    ej = zeros(l,1);
    ej(j) = 1;
    colj = Qt*(A*(At*(Q*ej)));
    BBt(:,j) = colj;
end
\end{lstlisting}
Here, we would make use of \eqref{eq:svd_mat_mult_approx} for approximating 
$Q^T A A^T Q e_j$.

\item Compute the eigendecomposition of $B B^T$

This simply is the eigendecomposition of a small $k \times k$ matrix:
\begin{lstlisting}
[Uhat,D] = eig(BBt);
\end{lstlisting}

\item Compute the low rank SVD components of $A$ by using the eigendecomposition derived 
in the previous step and applying $B^T = A^T Q$ to eigenvectors.

Here we use the fact that the eigenvalues of $B B^T$ are the squares of the singular values 
of $B$ and the computation \eqref{eq:svd_computation_for_eigenvectorsV} for the eigenvectors $V$.
\begin{lstlisting}
Sigma = zeros(l,l);
for i=1:l
    Sigma(i,i) = sqrt(D(i,i));
end

U = Q * Uhat;

V = zeros(n,l);
for j=1:l
    vj = 1/Sigma(j,j) * (At * U(:,j));
    V(:,j) = vj;
end
\end{lstlisting}
Here, we could use $A^T u_j \approx W^{-1} M^T u_j$. 

\item Finally, we extract the most dominant $k$ components of $U$, $V$, and $\Sigma$ 
to form $U_k = U(:,1:K), V_k = V(:,1:k), \Sigma_k = \Sigma(1:k,1:k)$. 
Notice that in this and previous steps, we use either the first or the last $k,l$ singular 
vectors and values, depending on the order returned by the eig function, corresponding to 
biggest to smallest by absolute magnitude.  

\end{itemize}
We note that the implementation of the low rank SVD algorithm above is 
simple, as long as we can perform matrix-vector operations using the wavelet 
compressed matrix $M$ and compute 
the eigendecomposition of a small $k \times k$ matrix, which can be done with a large number 
of available numerical packages. The disadvantage of this version is that working with the matrix $B B^T$
essentially squares the condition number of $A$, such that small singular values near machine precision
may not be properly resolved. This is an issue if $A$ is expected to have very small singular
values amongst $\sigma_1, \dots, \sigma_k$. However, if we take  
$k$ to be small relative to $\min(m,n)$ as we do in our application, $\sigma_k$ is significantly 
larger in magnitude than machine precision. The implementation of the algorithm in the pseudocode above 
is not very efficient for the randomized algorithm proposed in \cite{Halko:2011:FSR:2078879.2078881}, 
but one that is practical to use for very large $A$ when 
the corresponding wavelet compressed matrix $M = \Thr(A W^T)$ is available. In particular, 
for a more efficient implementation, one may want to block as many operations as possible, replacing 
matrix-vector by matrix-matrix multiplications. If possible, one may want to explicitly 
compute the matrix $B$ and then use it to form $B B^T$. Likewise, $V_k$ can be calculated 
directly from the matrix product $B^T \tilde{U_k} \Sigma_k^{-1}$. A power iteration strategy 
can also be implemented to improve accuracy in cases where the tail singular values 
decay more slowly. We refer the reader to \cite{2015arXiv150205366V} for more details.

\subsection{Application to Regularization Schemes}
For purposes of iterative regularization algorithms, we can make use of the low rank SVD 
in several ways. If we obtain the low rank SVD of the whole matrix, we can directly use it to 
approximate matrix vector operations:
\begin{equation}
Ax \approx U_k \left( \Sigma_k (V^T_k x) \right) \quad \mbox{and} \quad A^T y \approx V_k \left( \Sigma_k (U^T_k y) \right),
\end{equation} 
and in some situations this is the most convenient and straightforward approach. The disadvantage 
of this approach is that one must keep the matrices $U_k,U_k^T,V_k,V_k^T$ in memory. Here and 
below we do not pay attention to storing the matrix $\Sigma_k$ which is a very small diagonal matrix in comparison to the former matrices.
If the matrix $A$ is large it may be difficult to compute the low rank SVD of the whole matrix $A$. 
Instead, if we block $A$ as previously discussed, we can compute the low rank SVD of certain blocks 
or of each block. In some applications, it may be possible to arrange the blocks of 
$A$ in a way that the first block of $A$ contains many linearly dependent rows. 
If that is the case, then it is worthwhile to use the low rank SVD for the first block 
since it could be approximated well with small $k$. We can then write down mixed relations as follows:
\begin{equation}
\label{eq:mixed_relation_matvec_ops1}
Ax \approx
\begin{bmatrix}
U_{k_1} \Sigma_{k_1} V^T_{k_1} x \\
M_2 W_2^{-T} x \\
\vdots \\
M_p W_p^{-T} x \\
\end{bmatrix}
\quad \mbox{and} \quad
A^T y \approx V_{k_1} \Sigma_{k_1} U^T_{k_1} y_1 + \displaystyle\sum_{j=2}^p W_j^{-1} M^T_j y_j,
\end{equation}
where in this example we have used the low rank SVD approximation for the first part 
of the matrix and the wavelet based approximation for the other parts. 

Additional information can be learned by plugging in the low rank SVD directly into the regularization 
system. Our general model problem and its corresponding linear system are:
\begin{equation}
\bar{x} = \arg\min_x \left( ||Ax - b||_2^2 + \lambda_1 ||x||_2^2 + \lambda_2 ||Lx||_2^2 \right) \implies (A^T A + \lambda_1 I + \lambda_2 L^T L) \bar{x} = A^T b.
\end{equation}
Replacing all instances of $A$ by the low rank SVD results in:
\begin{equation*}
(A_k^T A_k + \lambda_1 I + \lambda_2 L^T L) \tilde{x_1} = A_k^T b,
\end{equation*}
which when expanded gives:
\begin{equation}
\label{eq:tikhonov_approx_replaceallA_by_svd1}
(V_k \Sigma^2_k V^T_k + \lambda_1 I + \lambda_2 L^T L) \tilde{x_1} = V_k \Sigma_k U_k^T b.
\end{equation}
The advantage of \eqref{eq:tikhonov_approx_replaceallA_by_svd1} is that if the right 
hand side $V_k \Sigma_k U_k^T b$ is computed at the start of the iteration, only the matrices 
$V_k$ and $V_k^T$ must be kept in memory during the iteration. 
We may think of precomputing the right hand side $A^T b$ and 
approximating only the operator $A^T A$. Note that $A^T b$ can always be precomputed 
before the iteration as long as we can split up $A$ into blocks. In this case we get:
\begin{equation}
\label{eq:tikhonov_approx_replaceallA_by_svd2}
(V_k \Sigma^2_k V^T_k + \lambda_1 I + \lambda_2 L^T L) \hat{x_1} = A^T b.
\end{equation}
As we will show later, this can result in slightly better error upper bound when the 
singular value $\sigma_{k+1}$ is sufficiently small, though the norm of the 
solution for the same choice of $\lambda_1$ would be higher in this case. 
Another approach is to work with the lower dimensional projected system:
\begin{equation}
\label{eq:tikhonov_approx_utautb1}
(U^T_k A) x = U^T_k b ,
\end{equation} 
where $U^T_k A$ is $k \times n$ if $A \in \mathbb{R}^{m \times n}$. Note that we have the following 
simple result:
\begin{lemma}
Given the low rank SVD $A_k = U_k \Sigma_k V_k^T$ of $A$, we have that 
$U_k^T A = U_k^T A_k = \Sigma_k V_k^T$.
\end{lemma}
\begin{proof}
First, $U_k^T A_k = U_k^T (U_k \Sigma_k V_k^T) = \Sigma_k V_k^T$. Also:
\begin{equation*}
U_k^T A = U_k^T \left( U_k \Sigma_k V_k^T + \hat{U}_k \hat{\Sigma}_k \hat{V}_k^T \right)  = \Sigma_k V_k^T + 0 = \Sigma_k V_k^T.
\end{equation*}
\end{proof}
If we solve \eqref{eq:tikhonov_approx_utautb1} by means of Tikhonov regularization:
\begin{eqnarray}
\label{eq:tikhonov_approx_utautb_linear_system1}
&& \tilde{x_2} = \arg\min_x \left\{ ||(U^T_k A) x - U^T_k b||_2^2 + \lambda_1 ||x||_2^2 + \lambda_2 ||L x||_2^2 \right\} \\
& \implies\ & \left( (U_k^T A)^T (U_k^T A) + \lambda_1 I  + \lambda_2 L^T L \right) \tilde{x_2} = (U_k^T A)^T U_k^T b,
\end{eqnarray} 
we will obtain the same solution as \eqref{eq:tikhonov_approx_replaceallA_by_svd1}:

\begin{lemma}
The approximation scheme 
$(V_k \Sigma^2_k V^T_k + \lambda_1 I + \lambda_2 L^T L) \tilde{x} = V_k \Sigma_k U_k^T b$ has the same 
solution as the Tikhonov regularized solution \eqref{eq:tikhonov_approx_utautb_linear_system1} of 
the projected system $(U^T_k A) x = U^T_k b$. 
\end{lemma}
\begin{proof}
Since
\begin{eqnarray*}
&& A = U_k \Sigma_k V^T_k + \hat{U}_k \hat{\Sigma}_k \hat{V}^T_k \implies U^T_k A = \Sigma_k V^T_k + 0 = \Sigma_k V^T_k \\
&\implies\ & (U_k^T A)^T (U_k^T A) = (U_k^T A_k)^T (U_k^T A_k) = (\Sigma_k V_k^T)^T (\Sigma_k V_k^T) = V_k \Sigma_k^2 V_k^T    ,
\end{eqnarray*}
the linear system from \eqref{eq:tikhonov_approx_utautb_linear_system1} is equivalent to:
\begin{equation*}
\left( V_k \Sigma_k^2 V_k^T + \lambda_1 I + \lambda_2 L^T L \right) \tilde{x_2} = (U_k^T A)^T U_k^T b = A^T U_k U_k^T b.
\end{equation*} 
Next, for the right hand side we have:
\begin{equation*}
A^T U_k = V_k \Sigma_k U^T_k U_k + \hat{V_k} \hat{\Sigma_k} \hat{U^T_k} U_k = V_k \Sigma_k I + 0 = V_k \Sigma_k \implies A^T U_k U_k^T b = V_k \Sigma_k U_k^T b.
\end{equation*}
Hence the solution of \eqref{eq:tikhonov_approx_utautb_linear_system1} is 
equivalent to that of \eqref{eq:tikhonov_approx_replaceallA_by_svd1}:
\begin{equation*}
(V_k \Sigma^2_k V^T_k + \lambda_1 I + \lambda_2 L^T L) \tilde{x_2} = V_k \Sigma_k U_k^T b.
\end{equation*}
\end{proof}

The advantage of \eqref{eq:tikhonov_approx_utautb1} is that it may be convenient for larger systems 
where we can only perform the low rank SVD of its blocks. In that case, we may form the blocked system:
\begin{equation}
\label{eq:uta_utb_block_projected_system}
\begin{bmatrix}
U_{k_1}^T A_1 \\
U_{k_2}^T A_2 \\
\vdots \\
U_{k_p}^T A_p \\
\end{bmatrix} x = 
\begin{bmatrix}
U_{k_1}^T b_1 \\
U_{k_2}^T b_2 \\
\vdots \\
U_{k_p}^T b_p \\
\end{bmatrix} \quad \mbox{or} \quad  
\begin{bmatrix}
\Sigma_{k_1} V_{k_1}^T \\
\Sigma_{k_2} V_{k_2}^T \\
\vdots \\
\Sigma_{k_p} V_{k_p}^T \\
\end{bmatrix} x = 
\begin{bmatrix}
U_{k_1}^T b_1 \\
U_{k_2}^T b_2 \\
\vdots \\
U_{k_p}^T b_p \\
\end{bmatrix},
\end{equation}
and solve the optimization problem via the augmented normal equations:
\begin{equation*}
\begin{bmatrix}
U_{k_1}^T A_{1}\\
U_{k_2}^T A_{2}\\
\vdots \\
U_{k_p}^T A_{p}\\
\sqrt{\lambda_1} I \\
\sqrt{\lambda_2} L \\
\end{bmatrix}^T
\begin{bmatrix}
U_{k_1}^T A_{1}\\
U_{k_2}^T A_{2}\\
\vdots \\
U_{k_p}^T A_{p}\\
\sqrt{\lambda_1} I \\
\sqrt{\lambda_2} L \\
\end{bmatrix}
\tilde{x_2} = 
\begin{bmatrix}
U_{k_1}^T A_{1}\\
U_{k_2}^T A_{2}\\
\vdots \\
U_{k_p}^T A_{p}\\
\sqrt{\lambda_1} I \\
\sqrt{\lambda_2} L \\
\end{bmatrix}^T
\begin{bmatrix}
U_{k_1}^T b_1 \\
U_{k_2}^T b_2 \\
\vdots \\
U_{k_p}^T b_p \\
0 \\
0 \\
\end{bmatrix}
\quad\mbox{or}\quad
\begin{bmatrix}
\Sigma_{k_1} V_{k_1}^T\\
\Sigma_{k_2} V_{k_2}^T \\
\vdots \\
\Sigma_{k_p} V_{k_p}^T \\
\sqrt{\lambda_1} I \\
\sqrt{\lambda_2} L \\
\end{bmatrix}^T
\begin{bmatrix}
\Sigma_{k_1} V_{k_1}^T\\
\Sigma_{k_2} V_{k_2}^T \\
\vdots \\
\Sigma_{k_p} V_{k_p}^T \\
\sqrt{\lambda_1} I \\
\sqrt{\lambda_2} L \\
\end{bmatrix}
\tilde{x_2} = 
\begin{bmatrix}
\Sigma_{k_1} V_{k_1}^T\\
\Sigma_{k_2} V_{k_2}^T \\
\vdots \\
\Sigma_{k_p} V_{k_p}^T \\
\sqrt{\lambda_1} I \\
\sqrt{\lambda_2} L \\
\end{bmatrix}^T
\begin{bmatrix}
U_{k_1}^T b_1 \\
U_{k_2}^T b_2 \\
\vdots \\
U_{k_p}^T b_p \\
0 \\
0 \\
\end{bmatrix}.
\end{equation*}
The number of eigenvectors for each block can be adjusted based on their conditioning. 
If the same $k$ is used for all the blocks then some are bound to be projected less accurately than 
others. If the right hand side is precomputed, only the matrices $V_{k_j}^T$ and $\Sigma_{k_j}$ must be 
in memory for each block. If it is easier to compute the eigenvector matrix $U_k$, then the default 
system with $U_k^T A$ may be useful. 

A more aggressive approach is to use the right eigenvectors $V_k$ to project the system 
from both sides to form a matrix of size $k \times k$. Instead of solving the full system:
\begin{equation*}
(A^T A + \lambda_1 I + \lambda_2 L^T L) \bar{x} = A^T b ,
\end{equation*}
we project the matrix used to a smaller space by multiplying on left by $V_k^T$ and 
preconditioning on the right by $V_k$:
\begin{equation*}
V_k^T (A^T A + \lambda_1 I + \lambda_2 L^T L) (V_k \tilde{y_3}) = V_k^T A^T b \quad \mbox{;} \quad \tilde{x_3} = V_k \tilde{y_3}.
\end{equation*}
Expanding this and noting that $V_k^T V_k = I$, we have:
\begin{equation}
\label{eq:tikhonov_approx_ktimesk1}
\left( V_k^T A^T A V_k + \lambda_1 I + \lambda_2 V_k^T L^T L V_k \right) \tilde{y_3} = V_k^T A^T b \quad \mbox{;} \quad \tilde{x_3} = V_k \tilde{y_3}.
\end{equation}
The key observation is that the matrix used in the linear system is $V_k^T A^T A V_k$, 
which is just of size $k \times k$, much smaller than the 
$m \times n$ matrix $A$. We can further simplify \eqref{eq:tikhonov_approx_ktimesk1} 
using the following calculations:
\begin{lemma}
Given the low rank SVD $A_k = U_k \Sigma_k V_k^T$ of $A$, we have that 
$V_k^T A^T A V_k = V_k^T A_k^T A_k V_k = \Sigma_k^2$ and $V_k A^T b = V_k A_k^T b = \Sigma_k U_k^T b$.
\end{lemma}
\begin{proof}
\begin{eqnarray*}
&& V_k^T A^T = V_k^T (V_k \Sigma_k U_k^T + \hat{V}_k \hat{\Sigma}_k \hat{U}_k^T) =  V_k^T A_k^T = \Sigma_k U_k^T \implies A V_k = A_k V_k = (\Sigma_k U_k^T)^T = U_k \Sigma_k \\
&\implies \ & V_k^T A^T A V_k = \Sigma_k U_k^T U_k \Sigma_k = \Sigma_k^2 \\
&\implies & V_k A_k^T b = \Sigma_k U_k^T b.
\end{eqnarray*}
\end{proof} 
Thus, we can rewrite \eqref{eq:tikhonov_approx_ktimesk1} as:
\begin{equation}
\label{eq:tikhonov_approx_ktimesk2}
\left( \Sigma_k^2 + \lambda_1 I + \lambda_2 V_k^T L^T L V_k \right) \tilde{y_3} = \Sigma_k U_k^T b \quad \mbox{;} \quad \tilde{x_3} = V_k \tilde{y_3}.
\end{equation}
We will show later that when $\lambda_2 = 0$, $\tilde{x_3} = \tilde{x_1}$, an important result, 
since the system for $\tilde{y_3}$ can be solved on a small machine, as it involves 
just a $k \times k$ matrix. When 
$\lambda_2 \neq 0$, this is only an approximation. We can obtain the $k$ columns of $V_k^T L^T L V_k$ 
by evaluating matrix vector products:
\begin{equation*}
V_k^T L^T L V_k e_j \quad \mbox{for} \quad j=1,\dots,k.
\end{equation*}
This is feasible to do in practice, since $k$ is not very large. This method is useful when 
many solutions with different values of $\lambda_1$ and $\lambda_2$ are required, or when a 
rough guess to warm start a more accurate method is desired. 

\vspace{2.mm} 
Let us now summarize the different techniques we have described for approximate $\ell_2$ regularization 
using the low rank SVD and their computational requirements. 

\begin{itemize}
\item[ (1)] We can implement $(A^T A + \lambda_1 I + \lambda_2 L^T L) \bar{x} = A^T b$ as usual and 
replace the operations $A x$ and $A^T y$ with 
$U_k \Sigma_k V^T_k x$ and $V_k \Sigma_k U^T_k y$. This requires one to have the 
matrices $U_k,U_k^T,V_k,V_k^T$ in memory, which may not be very efficient. However, this direct 
approach may be useful for larger matrices split into blocks using relations such as 
\eqref{eq:mixed_relation_matvec_ops1}, where the low rank SVD is applied only to certain blocks 
and not to the whole matrix. In that case, only the SVD components for the relevant blocks need to 
be loaded.

\item[ (2)] We can plug in the low rank SVD into the regularization problem to get the system:
\begin{equation*}
(V_k \Sigma^2_k V^T_k + \lambda_1 I + \lambda_2 L^T L) \tilde{x_1} = V_k \Sigma_k U_k^T b.
\end{equation*}
Note that the right hand side $V_k \Sigma_k U_k^T b$ can be precomputed before the iteration 
so that only the matrices $V_k$ and $V_k^T$ need to be in memory during iteration. The result should 
be equivalent to the first case but this approach is more efficient.
Additionally, we can precompute accurately the right hand side $A^T b$ and use the system:
\begin{equation*}
(V_k \Sigma^2_k V^T_k + \lambda_1 I + \lambda_2 L^T L) \hat{x_1} = A^T b.
\end{equation*}
Here the only difference is in the right hand side. As we will see later this can sometimes lead to 
solutions with a lower upper error bound, but should be used with a larger threshold for $\lambda_1$.

\item[ (3)] We can utilize the lower dimensional projected system $U_k^T A x = U_k^T b$. 
The corresponding system for the regularized problem:
\begin{equation*}
\left( (U_k^T A)^T (U_k^T A) + \lambda_1 I + \lambda_2 L^T L\right) \tilde{x_2} = (U_k^T A) U_k^T b
\end{equation*}
is equivalent to the system for $\tilde{x_1}$. However, in certain cases, the matrix $U_k$ may be 
easier to compute than $V_k$ (depending on the dimensions of $A^T A$ and $A A^T$) in which case 
one may then compute $U_k^T A$ by means of matrix-vector products $A^T U_k e_j$ for $j=1,\dots,k$. 
The method may also be useful for large systems since we can make use of 
\eqref{eq:uta_utb_block_projected_system}.

\item[ (4)] We can use the $k \times k$ system:
\begin{equation*}
\left( \Sigma_k^2 + \lambda_1 I + \lambda_2 V_k^T L^T L V_k \right) \tilde{y_3} = \Sigma_k U_k^T b \quad \mbox{;} \quad \tilde{x}_3 = V_k \tilde{y}_3.
\end{equation*}
The solution of the linear system can be done on small memory computers since it involves 
the use of $k \times k$ matrices only and one multiplication with $V_k$ at the end. The last step 
can be performed on a larger machine loading only $V_k$ into memory; or on smaller 
machines in blocks. This scheme is useful when many runs with the system with different values 
of $\lambda_1$ and $\lambda_2$ are desired. The solution is equivalent to $\tilde{x_1}$ when 
$\lambda_2 = 0$ as shown later in this section. 
\end{itemize}

Note that up to now we have discussed the application of the compression techniques to $\ell_2$ norm 
minimization problems. However, the techniques are applicable to other types of regularization also. 
For example, for $\ell_1$ regularization, where we minimize $||x||_1$ instead of $||x||_2$, 
one typically uses a scheme similar to the iterative soft thresholding algorithm 
\cite{ingrid_thresholding1}:
\begin{equation*}
x^{n+1} = \mathbb{S}_{\tau} \left( x^n + A^T b - A^T A x^n \right),
\end{equation*}
where $\left(\mathbb{S}_{\tau}(x)\right)_k = \sgn(x_k) \max{\{0, |x_k| - \tau\}}$ is the 
componentwise soft thresholding function. The main computational requirement here is in 
the operation $A^T A x^n$, just as for $\ell_2$ regularization. Hence, many of the 
techniques we have described can be used for different types of regularization problems.

\subsection{Further Analysis and Error Bounds}
In this section, we give more analysis for the SVD based schemes we have discussed. 
To make the analysis easier, we assume that 
$\lambda_1 = \lambda$ and $\lambda_2 = 0$ so we can do our analysis without 
the smoothing operator $L$, which is not approximated.
Consider now the true solution:
\begin{equation}
\label{eq:tikhonov_true_solution}
\bar{x} = (A^T A + \lambda I)^{-1} A^T b \quad (\mbox{True Solution}).
\end{equation}
Notice that we can easily understand the significance of \eqref{eq:tikhonov_true_solution} by plugging 
in the (full rank) SVD $A = U \Sigma V^T$ into \eqref{eq:tikhonov_true_solution}. One then 
obtains the solution:
\begin{equation*}
\bar{x} = V D U^T b \quad \mbox{with} \quad D = \Diag\left( 
\frac{\sigma_1}{\sigma^2_1 + \lambda},
\frac{\sigma_2}{\sigma^2_2 + \lambda}, \ldots,
\frac{\sigma_\mathrm{r}}{\sigma^2_\mathrm{r} + \lambda},
0, \ldots,0  
\right).
\end{equation*}
We see that 
the regularization alleviates the effects of the singular vectors
corresponding to small singular values $\sigma_i$,
by replacing each $\sigma_i$ by $\frac{\sigma_i}{\sigma^2_i + \lambda}$,  which prevents the
singular vectors corresponding to singular values smaller than $\lambda$ from dominating the 
solution \cite{Tikhonov63}. Notice that while the application of Tikhonov 
minimization acts to filter the small singular values of $A$ on the solution, 
the use of the low rank SVD $A_k$ in place of $A$ removes many of the small values 
entirely: the filtering is now done on those singular values which are retained.

We now restate the approximate solutions 
$\tilde{x_1},\hat{x_1},\tilde{x_2},\tilde{x_3}$ that have been described 
in detail in the last section, but now with $\lambda_1 = \lambda$ and $\lambda_2 = 0$:
\begin{eqnarray}
\label{eq:tilde_x1}
 \tilde{x_1} &=& (A^T_k A_k + \lambda I)^{-1} A^T_k b ,\\
\label{eq:hat_x1}
 \hat{x_1} &=& (A^T_k A_k + \lambda I)^{-1} A^T b ,\\
\label{eq:tilde_x2}
 \tilde{x_2} &=& \left( (U_k^T A)^T (U_k^T A) + \lambda I \right)^{-1} (U_k^T A) U_k^T b ,\\
\label{eq:tilde_x3}
 \tilde{x_3} &=& V_k \left(\Sigma_k^2 + \lambda I\right)^{-1} \Sigma_k U_k^T b.
\end{eqnarray}
Recall here that $\tilde{x_1}$ and $\hat{x_1}$ correspond respectively, to 
\eqref{eq:tikhonov_approx_replaceallA_by_svd1} and \eqref{eq:tikhonov_approx_replaceallA_by_svd2}, 
$\tilde{x_2}$ corresponds to \eqref{eq:tikhonov_approx_utautb_linear_system1}, and 
$\tilde{x_3}$ corresponds to \eqref{eq:tikhonov_approx_ktimesk1}.
We have previously shown that $\tilde{x_2}$ and $\tilde{x_1}$ have the same solution. 
We will show in this section that $\tilde{x_3}$ also has the same solution as $\tilde{x_1}$.

Using the Woodbury inverse formula \eqref{eq:woodbury_inverse_formula},
we can derive expressions relating the terms 
$(A^T_k A_k + \lambda I)^{-1}$ and $(A^T A + \lambda I)^{-1}$ which appear 
in the solutions $\tilde{x_1},\hat{x_1},\tilde{x_2},\tilde{x_3}$ and in the 
true solution $\bar{x}$.
\begin{lemma}
\label{lemma_error_bnd1}
Let $k$ be in the range $1 \leq k \leq r-1$ and $\lambda > 0$. Then:
\begin{eqnarray}
\label{eq:lemma_error_bnd1_eqn1}
&& (A^T_k A_k + \lambda I)^{-1} = \lambda^{-1} I - V_k S_k V_k^T \\
\nonumber
&& \mbox{with} \quad S_k = \Diag\left( \frac{\sigma_s^2}{\lambda^2+\lambda\sigma_s^2} \right) \quad \mbox{for} \quad s = 1,\dots,k,
\end{eqnarray}
and:
\begin{eqnarray}
\label{eq:lemma_error_bnd1_eqn2}
&& (A^T A + \lambda I)^{-1} =  (A^T_k A_k + \lambda I)^{-1} - \hat{V}_k \hat{S}_k \hat{V}_k^T \\
\nonumber
&& \mbox{with} \quad \hat S_k = \Diag\left( \frac{\sigma_{s}^2}{\lambda^2+\lambda\sigma_{s}^2} \right) \quad \mbox{for} \quad s=k+1,\dots,r.
\end{eqnarray}
These imply that:
\begin{eqnarray}
&&\bar{x} = \left( (A^T_k A_k + \lambda I)^{-1} - \hat{V}_k \hat{S}_k \hat{V}_k^T \right) A^T b, \\
&&\tilde{x_1} = \left( \lambda^{-1} I - V_k S_k V_k^T \right) A_k^T b, \\
&&\hat{x_1} = \left( \lambda^{-1} I - V_k S_k V_k^T \right) A^T b.
\end{eqnarray}

%
\end{lemma}
\begin{proof}
The proof follows by the use of the Woodbury inverse formula \eqref{eq:woodbury_inverse_formula}:
\begin{equation*}
(PTR + D)^{-1} = D^{-1} - D^{-1} P (R D^{-1} P + T^{-1})^{-1} R D^{-1}.
\end{equation*}
We match this with $(A^T_k A_k + \lambda I)^{-1} = (V_k \Sigma_k^2 V^T_k + \lambda I)^{-1}$ 
to get $P = V_k$, $R = V^T_k$, $T= \Sigma_k^2$, and $D = \lambda I$:
\begin{eqnarray*}
(A^T_k A_k + \lambda I)^{-1} &=& \lambda^{-1} I - \lambda^{-1} V_k (V^T_k \lambda^{-1} V_k + \Sigma_k^{-2})^{-1} V^T_k \lambda^{-1} = \lambda^{-1} I - \lambda^{-2} V_k \left( \Sigma_k^{-2} + \lambda^{-1} I \right)^{-1} V_k^T \\
& = & \lambda^{-1} I - \lambda^{-2} V_k \left( \Diag(\sigma_1^{-2},\dots,\sigma_k^{-2}) + \lambda^{-1} I \right)^{-1} V_k^T \\ 
& = & \lambda^{-1} I - \lambda^{-2} V_k \Diag(\sigma_1^{-2} + \lambda^{-1},\dots,\sigma_k^{-2} + \lambda^{-1})^{-1} V_k^T \\ 
& = & \lambda^{-1} I - \lambda^{-2} V_k \Diag\left((\sigma_1^{-2} + \lambda^{-1})^{-1},\dots,(\sigma_k^{-2} + \lambda^{-1})^{-1}\right) V_k^T \\ 
& = & \lambda^{-1} I - \lambda^{-2} V_k \Diag\left(\frac{\lambda \sigma_1^2}{\lambda + \sigma_1^2},\dots,\frac{\lambda \sigma_k^2}{\lambda + \sigma_k^2}\right) V_k^T \\ 
& = & \lambda^{-1} I - V_k \Diag\left(\frac{\sigma_1^2}{\lambda^2 + \lambda \sigma_1^2},\dots,\frac{\sigma_k^2}{\lambda^2 + \lambda \sigma_k^2}\right) V_k^T = \lambda^{-1} I - V_k S_k V_k^T ,
\end{eqnarray*}
which proves \eqref{eq:lemma_error_bnd1_eqn1}.

\vspace{2.mm}
For \eqref{eq:lemma_error_bnd1_eqn2}, we have:
\begin{equation*}
(A^T A + \lambda I)^{-1} = (A^T_k A_k + \hat{V}_k \hat{\Sigma}_k^2 \hat{V}_k^T + \lambda I)^{-1} = (\hat{V}_k \hat{\Sigma}_k^2 \hat{V}_k^T + \M)^{-1} ,
\end{equation*}
with $\M = A^T_k A_k + \lambda I$. Using Woodbury matrix formula:
\begin{equation*}
(\hat{V}_k \hat{\Sigma}_k^2 \hat{V}_k^T + \M)^{-1} = \M^{-1} - \M^{-1} \hat{V}_k \left( \hat{\Sigma}_k^{-2} + \hat{V}_k^T \M^{-1} \hat{V}_k \right)^{-1} \hat{V}_k^T \M^{-1} .
\end{equation*} 
Now, by $\eqref{eq:lemma_error_bnd1_eqn1}$ we have $\M^{-1} = \lambda^{-1} I - V_k S_k V_k^T$ and 
by orthogonality we have $\hat{V}_k^T V_k = 0$:
\begin{eqnarray*}
 \hat{V}_k^T \M^{-1} &=& \hat{V}_k^T (\lambda^{-1} I - V_k S_k V_k^T) = \lambda^{-1} \hat{V}_k^T \\
 \M^{-1} \hat{V}_k &=& (\lambda^{-1} I - V_k S_k V_k^T) \hat{V}_k  = \lambda^{-1} \hat{V}_k .
\end{eqnarray*}
Thus:
\begin{eqnarray*}
(A^T A + \lambda I)^{-1} &=& \M^{-1} - \M^{-1} \hat{V}_k \left( \hat{\Sigma}_k^{-2} + \hat{V}_k^T \M^{-1} \hat{V}_k \right)^{-1} \hat{V}_k^T \M^{-1} \\
&=& \M^{-1} - \lambda^{-1} \hat{V}_k \left( \hat{\Sigma}_k^{-2} + \hat{V}_k^T \lambda^{-1} \hat{V}_k \right)^{-1} \lambda^{-1} \hat{V}_k^T \\
&=&  \M^{-1} - \lambda^{-2} \hat{V}_k \left( \hat{\Sigma}_k^{-2} + \lambda^{-1} I \right)^{-1} \hat{V}_k^T \\
&=& \M^{-1} - \lambda^{-2} \hat{V}_k \Diag\left(\frac{\lambda + \sigma_{k+1}^2}{\lambda \sigma_{k+1}^2} ,\dots,\frac{\lambda + \sigma_r^2}{\lambda \sigma_r^2} \right)^{-1} \hat{V}_k^T \\
&=& \M^{-1} - \lambda^{-2} \hat{V}_k \Diag\left( \frac{\lambda \sigma_{k+1}^2}{\lambda + \sigma_{k+1}^2},\dots,\frac{\lambda \sigma_r^2}{\lambda + \sigma_r^2} \right) \hat{V}_k^T \\
&=& (A^T_k A_k + \lambda I)^{-1} - \hat{V}_k \Diag \left( \frac{\sigma_1^2}{\lambda^2 + \lambda \sigma_{k+1}^2},\dots,\frac{\sigma_r^2}{\lambda^2 + \lambda \sigma_r^2}  \right) \hat{V}_k^T \\
&=& (A^T_k A_k + \lambda I)^{-1} - \hat{V}_k \hat{S}_k \hat{V}_k^T ,
\end{eqnarray*}
which proves \eqref{eq:lemma_error_bnd1_eqn2}.

\noindent
Equations \eqref{eq:lemma_error_bnd1_eqn1} and \eqref{eq:lemma_error_bnd1_eqn2} imply that:
\begin{eqnarray*}
&&\bar{x} = (A^T A + \lambda I)^{-1} A^T b = \left( (A^T_k A_k + \lambda I)^{-1} - \hat{V}_k \hat{S}_k \hat{V}_k^T \right) A^T b ,\\
&&\tilde{x_1} = (A_k^T A_k + \lambda I)^{-1} A_k^T b = \left( \lambda^{-1} I - V_k S_k V_k^T \right) A_k^T b ,\\
&&\hat{x_1} = (A_k^T A_k + \lambda I)^{-1} A^T b = \left( \lambda^{-1} I - V_k S_k V_k^T \right) A^T b.
\end{eqnarray*}
\end{proof}

\noindent
Now we show that $\tilde{x_3}$ (involving the inversion of a $k \times k$ matrix) 
has the same solution as $\tilde{x_1}$ and derive the expression for the difference between 
$\tilde{x_1}$ and $\hat{x_1}$.
\begin{lemma}
\label{lem:lemma_tildexs}
Let $\bar{x}$ be the solution of \eqref{eq:tikhonov_true_solution}, $\tilde{x_1}$ the 
solution of \eqref{eq:tilde_x1}, $\hat{x_1}$ the solution of \eqref{eq:hat_x1} and 
$\tilde{x_3}$ the solution of \eqref{eq:tilde_x3}. Then, we have:
\begin{equation}
\label{eq:lemma_tildexs_eqn1}
\tilde{x_3} = \tilde{x_1},
\end{equation}
and
\begin{equation} 
\label{eq:lemma_tildexs_eqn2}
\hat{x_1} - \tilde{x_1}  = \lambda^{-1} \left(A^T - A_k^T\right)b = \lambda^{-1} \hat{A_k}^T b.
\end{equation}
\end{lemma}
\begin{proof}
First note that:
\begin{equation*}
V_k V_k^T A_k^T b = V_k V_k^T V_k \Sigma_k U_k^T b = V_k \Sigma_k U_k^T b = A_k^T b.
\end{equation*}
Next, we expand:
\begin{eqnarray*}
\tilde{x_1} &=& \left( \lambda^{-1} I - V_k S_k V_k^T \right) A_k^T b = \lambda^{-1} A_k^T b - V_k S_k V_k^T A_k^T b = \lambda^{-1} V_k V_k^T A_k^T b - V_k S_k V_k^T A_k^T b \\
&=& V_k \left( \lambda^{-1} I - S_k \right) V_k^T A_k^T b = V_k \left( \lambda^{-1} I - \Diag\left( \frac{\sigma_{s}^2}{\lambda^2+\lambda\sigma_{s}^2} \right)  \right) V_k^T A_k^T b \\
&=& V_k \Diag\left( \frac{1}{\lambda} - \frac{\sigma_{s}^2}{\lambda^2+\lambda\sigma_{s}^2} \right)  V_k^T A_k^T b = V_k \Diag\left( \frac{(\sigma_{s}^2 + \lambda) - \sigma_{s}^2}{\lambda 
(\sigma_{s}^2 + \lambda)} \right) V_k^T A_k^T b \\
&=& V_k \Diag\left( \frac{1}{\sigma_{s}^2 + \lambda} \right) V_k^T A_k^T b = 
V_k (\Sigma_k^2 + \lambda I)^{-1} V_k^T A_k^T b = \tilde{x_3} ,
\end{eqnarray*}
which proves \eqref{eq:lemma_tildexs_eqn1}. Next, for the difference between $\tilde{x_1}$ and 
$\hat{x_1}$ we have:
\begin{eqnarray*}
&& \tilde{x_1} = \left( A_k^T A_k + \lambda I \right)^{-1} A_k^T b = \left( \lambda^{-1} I - V_k S_k V_k^T \right) A_k^T b = \lambda^{-1} A_k^T b - V_k S_k V_k^T A_k^T b  ,\\
&& \hat{x_1} = \left( A_k^T A_k + \lambda I \right)^{-1} A^T b = \left( \lambda^{-1} I - V_k S_k V_k^T \right) A^T b = \lambda^{-1} A^T b - V_k S_k V_k^T A^T b.
\end{eqnarray*}
Note that:
\begin{equation*}
V_k S_k V_k^T A^T b = V_k S_k V_k^T (V_k \Sigma_k U_k^T + \hat{V}_k \hat{\Sigma}_k \hat{U}_k^T) b = 
V_k S_k V_k^T A_k^T b.
\end{equation*}
Hence:
\begin{equation*}
\hat{x_1} - \tilde{x_1} = \lambda^{-1} A^T b - \lambda^{-1} A_k^T b = \lambda^{-1} (A^T - A_k^T) b = \lambda^{-1} \hat{A_k}^T b,
\end{equation*}
which proves \eqref{eq:lemma_tildexs_eqn2}.
\end{proof}

\vspace{2.mm}
By the result of Lemma \ref{lem:lemma_tildexs}, the only solutions which  
differ from each other are $\tilde{x_1}$ and $\hat{x_1}$. We now analyze these two 
solutions with respect to the true solution $\bar{x}$. 
\begin{proposition}
Let $\bar{x}$ be the solution of \eqref{eq:tikhonov_true_solution} and $\tilde{x_1}$ the 
solution of \eqref{eq:tilde_x1}. Then:
\begin{equation}
\label{eq:lemma_tilde1_eq1}
||\bar{x} - \tilde{x_1}||_2 \leq \frac{\sigma_{k+1}}{\lambda + \sigma^2_{k+1}} ||b||_2 ,
\end{equation}
and 
\begin{equation}
\label{eq:lemma_tilde1_eq2}
\tilde{x_1} = V_k V_k^T \bar{x} .
\end{equation}
\end{proposition}
\begin{proof}
Recall that $\tilde{x_1} = (A_k^T A_k + \lambda I)^{-1} A_k^T b$ and that 
$\bar{x} = (A^T A + \lambda I)^{-1} A^T b$. Next by Lemma \ref{lemma_error_bnd1} and using 
that $A_k \hat{V_k} = (U_k \Sigma_k V^T_k) \hat{V}_k = 0$ and $\hat{V}_k^T V_k = 0$:
\begin{eqnarray*}
(A^T A + \lambda I)^{-1} A^T &=& (A^T A + \lambda I)^{-1} (A^T_k + \hat{V}_k \hat{\Sigma}_k \hat{U}_k^T) = \left( (A^T_k A_k + \lambda I)^{-1} - \hat{V}_k \hat{S}_k \hat{V}_k^T \right) (A_k^T + \hat{V}_k \hat{\Sigma}_k \hat{U}_k^T) \\ 
&=& (A^T_k A_k + \lambda I)^{-1} A_k^T + (A^T_k A_k + \lambda I)^{-1} \hat{V}_k \hat{\Sigma}_k \hat{U}_k^T - \hat{V}_k \hat{S}_k \hat{\Sigma}_k \hat{U}_k^T \\
&=& (A^T_k A_k + \lambda I)^{-1} A_k^T + (\lambda^{-1} I - V_k S_k V_k^T)  \hat{V}_k \hat{\Sigma}_k \hat{U}_k^T - \hat{V}_k \hat{S}_k \hat{\Sigma}_k \hat{U}_k^T \\
&=& (A^T_k A_k + \lambda I)^{-1} A_k^T + \lambda^{-1} \hat{V}_k \hat{\Sigma}_k \hat{U}_k^T - \hat{V}_k \hat{S}_k \hat{\Sigma}_k \hat{U}_k^T \\
&=& (A^T_k A_k + \lambda I)^{-1} A_k^T + \hat{V}_k \left( \lambda^{-1} \hat{\Sigma}_k - \hat{S}_k \hat{\Sigma}_k \right)\hat{U}_k^T .
\end{eqnarray*}
Since $\hat{S}_k = \Diag\left(\frac{\sigma_s^2}{\lambda^2 + \lambda \sigma_s^2}\right)$ for $s=(k+1),\dots,r$:
\begin{equation*}
\lambda^{-1} \hat{\Sigma}_k - \hat{S}_k \hat{\Sigma}_k = \Diag\left( \frac{\sigma_s}{\lambda} - 
\frac{\sigma_s^3}{\lambda(\lambda + \sigma_s^2)} \right) = \Diag \left( \frac{\sigma_s}{\lambda + \sigma_s^2} \right) \quad \mbox{for} \quad s=(k+1),\dots,r .
\end{equation*}
Hence:
\begin{equation*}
(A^T A + \lambda I)^{-1} A^T = (A^T_k A_k + \lambda I)^{-1} A_k^T + \hat{V}_k \Diag\left(  \frac{\sigma_s}{\lambda + \sigma_s^2} \right) \hat{U}_k^T,
\end{equation*}
which implies:
\begin{eqnarray}
\label{eq:lemma_tilde1_innereq1}
\bar{x} &=& (A^T A + \lambda I)^{-1} A^T b = (A^T_k A_k + \lambda I)^{-1} A_k^T b +  \hat{V}_k \Diag\left(  \frac{\sigma_s}{\lambda + \sigma_s^2} \right) \hat{U}_k^T  \\
&=& \tilde{x_1} + \hat{V}_k \Diag\left(  \frac{\sigma_s}{\lambda + \sigma_s^2} \right) \hat{U}_k^T b
\end{eqnarray}
\begin{eqnarray*}
\implies ||\bar{x} - \tilde{x_1}||_2 &=& \left\| \hat{V}_k \Diag\left(  \frac{\sigma_s}{\lambda + \sigma_s^2} \right) \hat{U}_k^T b \right\| = \left\| \Diag\left(  \frac{\sigma_s}{\lambda + \sigma_s^2} \right) \hat{U}_k^T b \right\| \leq \left\| \Diag\left(  \frac{\sigma_s}{\lambda + \sigma_s^2} \right) \right\|_2 ||b||_2 \\ 
&\leq& \frac{\sigma_{k+1}}{\lambda + \sigma_{k+1}^2} ||b||_2 ,
\end{eqnarray*}
which proves \eqref{eq:lemma_tilde1_eq1}.

\vspace{2.mm}
Next, to derive \eqref{eq:lemma_tilde1_eq2}, we have:
\begin{equation*}
A^T b = (V_k \Sigma_k U_k^T)b + (\hat{V}_k \hat{\Sigma}_k \hat{U}_k^T)b,
\end{equation*}
so that 
\begin{equation*}
\hat{V}_k \hat{\Sigma}_k \hat{U}_k^T b = A^T b - (V_k \Sigma_k U_k^T) b \implies 
\hat{V}^T_k \hat{V}_k \hat{\Sigma}_k \hat{U}_k^T b = \hat{\Sigma}_k \hat{U}_k^T b = \hat{V}^T_k A^T b - 0 \implies \hat{U}_k^T b = \hat{\Sigma}_k^{-1} \hat{V}^T_k A^T b
\end{equation*}
\begin{eqnarray*}
\implies \hat{U}_k^T b &=& \hat{\Sigma}_k^{-1} \hat{V}^T_k (A^T A + \lambda I) \bar{x} = \hat{\Sigma}_k^{-1} \hat{V}^T_k (V_k \Sigma_k^2 V^T_k + \hat{V}_k \hat{\Sigma}_k^2 \hat{V}_k^T + \lambda I) \bar{x} = \hat{\Sigma}_k^{-1} \left( \hat{\Sigma}_k^2 \hat{V}_k^T + \lambda \hat{V}^T_k \right) \bar{x} \\
&=& \hat{\Sigma}_k \hat{V}_k^T \bar{x} + \lambda \hat{\Sigma}_k^{-1} \hat{V}^T_k \bar{x} = 
(\hat{\Sigma}_k + \lambda \hat{\Sigma}_k^{-1}) \hat{V}_k^T \bar{x}.
\end{eqnarray*}
Using \eqref{eq:lemma_tilde1_innereq1}, we have:
\begin{eqnarray*}
\bar{x} &=& \tilde{x_1} + \hat{V}_k \Diag \left( \frac{\sigma_s}{\lambda + \sigma_s^2} \right) \hat{U}_k^T b = \tilde{x_1} + \hat{V}_k \Diag \left( \frac{\sigma_s}{\lambda + \sigma_s^2} \right) (\hat{\Sigma}_k + \lambda \hat{\Sigma}_k^{-1}) \hat{V}_k^T \bar{x} \\
&=& \tilde{x_1} + \hat{V}_k \Diag \left( \frac{\sigma_s}{\lambda + \sigma_s^2} \right) \Diag \left(\sigma_s + \frac{\lambda}{\sigma_s}\right) \hat{V}_k^T \bar{x} = \tilde{x_1} + \hat{V}_k  \Diag \left( \frac{\sigma_s}{\lambda + \sigma_s^2} \right) \Diag \left(\frac{\sigma_s^2 + \lambda}{\sigma_s}\right) \hat{V}_k^T \bar{x} \\
&=& \tilde{x_1} + \hat{V}_k \hat{V}_k^T \bar{x} = \tilde{x_1} + (I - V_k V_k^T) \bar{x} = \tilde{x_1} + \bar{x} - V_k V_k^T \bar{x}.
\end{eqnarray*}
This proves \eqref{eq:lemma_tilde1_eq2}:
\begin{equation*}
\tilde{x_1} = V_k V_k^T \bar{x}.
\end{equation*}
\end{proof}

Next, we look at the solution $\hat{x_1} = (A_k^T A_k + \lambda I)^{-1} A^T b$. Recall that  
the difference from $\tilde{x_1}$ is that in $\hat{x_1}$, $A^T b$ is not approximated by $A_k^T b$. 
\begin{proposition}
Let $\bar{x}$ be the solution of \eqref{eq:tikhonov_true_solution} and $\hat{x_1}$ the 
solution of \eqref{eq:hat_x1}. Then:
\begin{equation}
\label{eq:prop_error_bnd1_eq1}
||\bar{x} - \hat{x_1}||_2 \leq \frac{\sigma_{k+1}^3}{\lambda^2 + \lambda \sigma^2_{k+1}} ||b||_2,
\end{equation}
and
\begin{equation}
\label{eq:prop_error_bnd1_eq2}
\frac{||\bar{x} - \hat{x_1}||_2}{||\bar{x}||_2} \leq \frac{\sigma_{k+1}^2}{\lambda}.
\end{equation}
\end{proposition}
\begin{proof}
We use lemma \ref{lemma_error_bnd1} to relate $\bar{x}$ to $\hat{x_1}$.
\begin{eqnarray}
\bar{x} &=& (A^T A + \lambda I)^{-1} A^T b = \left( (A^T_k A_k + \lambda I)^{-1} - \hat{V}_k \hat{S}_k \hat{V}_k^T  \right) A^T b \\
\label{eq:prop_error_bnd1_drv1}
&=& \hat{x_1} - \hat{V}_k \hat{S}_k \hat{V}_k^T A^T b = \hat{x_1} - \hat{V}_k \hat{S}_k \hat{V}_k^T \left( A^T_k + \hat{V}_k \hat{\Sigma}_k \hat{U}_k^T \right) b 
= \hat{x_1} - \hat{V}_k \hat{S}_k \hat{\Sigma}_k \hat{U}_k^T b,
\end{eqnarray}
where the last equality follows from $\hat{V}^T_k A_k^T = 0$ and $\hat{V}^T_k \hat{V} = I$. 
Thus, we have:
\begin{eqnarray*}
||\bar{x} - \hat{x_1}||_2 = ||\hat{V}_k \hat{S}_k \hat{\Sigma}_k \hat{U}_k^T b||_2 = ||\hat{S}_k \hat{\Sigma}_k \hat{U}_k^T b||_2 \leq ||\hat{S}_k \hat{\Sigma}_k||_2 ||\hat{U}_k^T b||_2 \leq 
||\hat{S}_k \hat{\Sigma}_k||_2 ||b||_2 .
\end{eqnarray*}
Now from lemma \ref{lemma_error_bnd1}:
\begin{eqnarray*}
&& \hat{S}_k \hat{\Sigma}_k = \Diag\left( \frac{\sigma_{k+1}^3}{\lambda^2+\lambda\sigma_{k+1}^2}, \frac{\sigma_{k+3}^2}{\lambda^2+\lambda\sigma_{k+2}^2}, \ldots, \frac{\sigma_{r}^3}{\lambda^2+\lambda\sigma_{r}^2} \right) \\
& \implies\ & ||\hat{S}_k \hat{\Sigma}_k||_2 = \max (\hat{S}_k \hat{\Sigma}_k) = \frac{\sigma_{k+1}^3}{\lambda^2+\lambda\sigma_{k+1}^2} \\
& \implies\ & ||\hat{S}_k \hat{\Sigma}_k||_2 ||b||_2 = \frac{\sigma_{k+1}^3}{\lambda^2+\lambda\sigma_{k+1}^2} ||b||_2.
\end{eqnarray*}
So we obtain the bound \eqref{eq:prop_error_bnd1_eq1}:
\begin{equation*}
||\bar{x} - \hat{x_1}||_2 \leq ||\hat{S}_k \hat{\Sigma}_k||_2 ||b||_2 = \frac{\sigma_{k+1}^3}{\lambda^2+\lambda\sigma_{k+1}^2} ||b||_2.
\end{equation*}

\vspace{2.mm}
\noindent
In order to obtain \eqref{eq:prop_error_bnd1_eq2}, we need to get rid of the $||b||_2$ term. 
We appeal back to \eqref{eq:prop_error_bnd1_drv1}:
\begin{eqnarray*}
\bar{x} &=& \hat{x_1} - \hat{V}_k \hat{S}_k \hat{V}_k^T A^T b = \hat{x_1} - \hat{V}_k \hat{S}_k \hat{V}_k^T \left( A^T A + \lambda I \right) \bar{x}  \\
&=& \hat{x_1} - \hat{V}_k \hat{S}_k \hat{V}_k^T \left( V_k \Sigma_k^2 V^T_k + \hat{V}_k \hat{\Sigma}^2_k \hat{V}^T_k  + \lambda I \right) \bar{x} = 
\hat{x_1} - \hat{V}_k \hat{S}_k \left( \hat{\Sigma_k^2} + \lambda I \right) \hat{V}_k^T \bar{x}.
\end{eqnarray*}
It follows that:
\begin{eqnarray*}
&& ||\bar{x} - \hat{x_1}||_2 = ||\hat{V}_k \hat{S}_k \left( \hat{\Sigma_k^2} + \lambda I \right) \hat{V}_k^T \bar{x}||_2 \leq ||\hat{V}_k \hat{S}_k \left( \hat{\Sigma_k^2} + \lambda I \right) \hat{V}_k^T||_2 ||\bar{x}||_2 = || \hat{S}_k \left( \hat{\Sigma_k^2} + \lambda I \right) ||_2 ||\bar{x}||_2 \\
& \implies\ & \frac{||\bar{x} - \hat{x_1}||_2}{||\bar{x}||_2} \leq || \hat{S}_k \left( \hat{\Sigma_k^2} + \lambda I \right) ||_2 \leq || \hat{S}_k ||_2 || \left( \hat{\Sigma_k^2} + \lambda I \right) ||_2 = 
\frac{\sigma_{k+1}^2}{\lambda^2 + \lambda \sigma_{k+1}^2} (\sigma_{k+1}^2 + \lambda),
\end{eqnarray*}
which simplifies to:
\begin{equation*}
\frac{||\bar{x} - \hat{x_1}||_2}{||\bar{x}||_2} \leq \frac{\sigma_{k+1}^2}{\lambda}.
\end{equation*}
\end{proof}

Let us now recall some results we have derived. First of all, we have shown that $\tilde{x_1}$, 
$\tilde{x_2}$ and $\tilde{x_3}$ lead to the same solution. Numerically, however, one may still observe 
some differences if they are not run to convergence. On the other hand, 
$\tilde{x_1}$ and $\hat{x_1}$ differ from each other and have the following absolute error bounds with 
respect to the true solution $\bar{x}$:
\begin{eqnarray*}
&& ||\bar{x} - \tilde{x_1}||_2 \leq \frac{\sigma_{k+1}}{\lambda+\sigma_{k+1}^2} ||b||_2 ;\\
&& ||\bar{x} - \hat{x_1}||_2 \leq \frac{\sigma_{k+1}^3}{\lambda \left( \lambda + \sigma_{k+1}^2 \right)} ||b||_2 \\
\end{eqnarray*}
Recall that the difference between the two is in the right hand side: $\hat{x_1}$ uses the 
un-approximated right hand side, or at least one computed with the wavelet transformed matrix 
(i.e. $A^T b \approx W^{-1} M^T b$). We mention again that one operation with a large 
$A$ or $M$ is not prohibitively expensive as it can be done by splitting the matrix 
into small enough blocks. The plot below in Figure \ref{fig:tildex_hatx_ub_comp_fig}
gives us a sense of how the upper bounds behave. We plot the fraction: 
\begin{equation}
\label{eq:tildex_hatx_ub_comp_eq}
\beta = \frac{\frac{\sigma_{k+1}^3}{\lambda \left( \lambda + \sigma_{k+1}^2 \right)} - \frac{\sigma_{k+1}}{\lambda+\sigma_{k+1}^2}}{\left| \frac{\sigma_{k+1}}{\lambda+\sigma_{k+1}^2} \right|}
\end{equation}
as a function of the value of $\sigma_{k+1}$ for two different choices of $\lambda$. 
The fraction \eqref{eq:tildex_hatx_ub_comp_eq} is simply a relative difference between 
the two upper bounds for the error of the approximate solutions $\hat{x_1}$ and $\tilde{x_1}$. 
From Figure \ref{fig:tildex_hatx_ub_comp_fig}, we may observe 
that the difference fraction is negative (indicating a lower upper bound 
error for $\hat{x_1}$) when the value of $\sigma_{k+1}$ is sufficiently small. However, 
if $k$ is not large enough for $\sigma_{k+1}$ to be sufficiently small then the upper 
bound of $\hat{x_1}$ will be worse than that of $\tilde{x_1}$.
Another observation about the solution $\hat{x_1}$ compared to $\tilde{x_1}$ (and the other 
solutions equivalent to it) is that $\hat{x_1}$ for the same choice of $\lambda$ is expected 
to have a larger norm:

\begin{lemma}
Let $\tilde{x_1}$ be the solution of \eqref{eq:tilde_x1} and $\hat{x_1}$ the solution of 
\eqref{eq:hat_x1} for a fixed value of $\lambda$. 
Then, we have that $||\tilde{x_1}||_2 \leq ||\hat{x_1}||_2$.
\label{lemma:norms_of_x_1}
\end{lemma}
\begin{proof}
Recall that $A = A_k + \hat{A_k}$ and 
\begin{equation*}
\tilde{x_1} = (A_k^T A_k + \lambda I)^{-1} A_k^T b \quad \mbox{;} \quad \hat{x_1} = (A_k^T A_k + \lambda I)^{-1} A^T b .
\end{equation*}
Now by Lemma \ref{lem:lemma_tildexs}:
\begin{eqnarray*}
\hat{x_1} = \tilde{x_1} + \lambda^{-1} \hat{A_k^T} b.
\end{eqnarray*}
Thus, the norms are related as:
\begin{equation*}
||\hat{x_1}||_2^2 = ||\tilde{x_1}||_2^2 + 2 \lambda^{-1} \tilde{x_1}^T \hat{A_k^T} b + ||\hat{A_k^T} b||_2^2 ,
\end{equation*}
where the middle term is zero as we now show. Note that $\hat{A_k} A_k^T = \hat{A_k} V_k = 0$ and:
\begin{equation*}
\left(\tilde{x_1}^T \hat{A_k^T}\right)^T = \hat{A_k} \tilde{x_1} = \hat{A_k} (A_k^T A_k + \lambda I)^{-1} A_k^T b = \hat{A_k} (\lambda^{-1} I - V_k S_k V_k^T) A_k^T b = \lambda^{-1} \hat{A_k} A_k^T b + \hat{A_k} V_k S_k V_k^T A_k^T b = 0.
\end{equation*}
Thus:
\begin{equation*}
||\hat{x_1}||_2^2 = ||\tilde{x_1}||_2^2 + ||\hat{A_k^T} b||_2^2 \implies ||\tilde{x_1}||_2 \leq ||\hat{x_1}||_2 .
\end{equation*}
\end{proof}
Thus, when using $\hat{x_1}$ as an estimate for $\bar{x}$ we typically would like to take a 
larger value of $\lambda$ to obtain a solution with similar norm to that of $\tilde{x_1}$. If we use 
the same $\lambda$ for $\hat{x_1}$ and $\bar{x}$, we will find that the components of the solution 
of $\bar{x}$ have larger amplitudes.
\begin{figure*}[h!]
\centerline{
\includegraphics[scale=0.4]{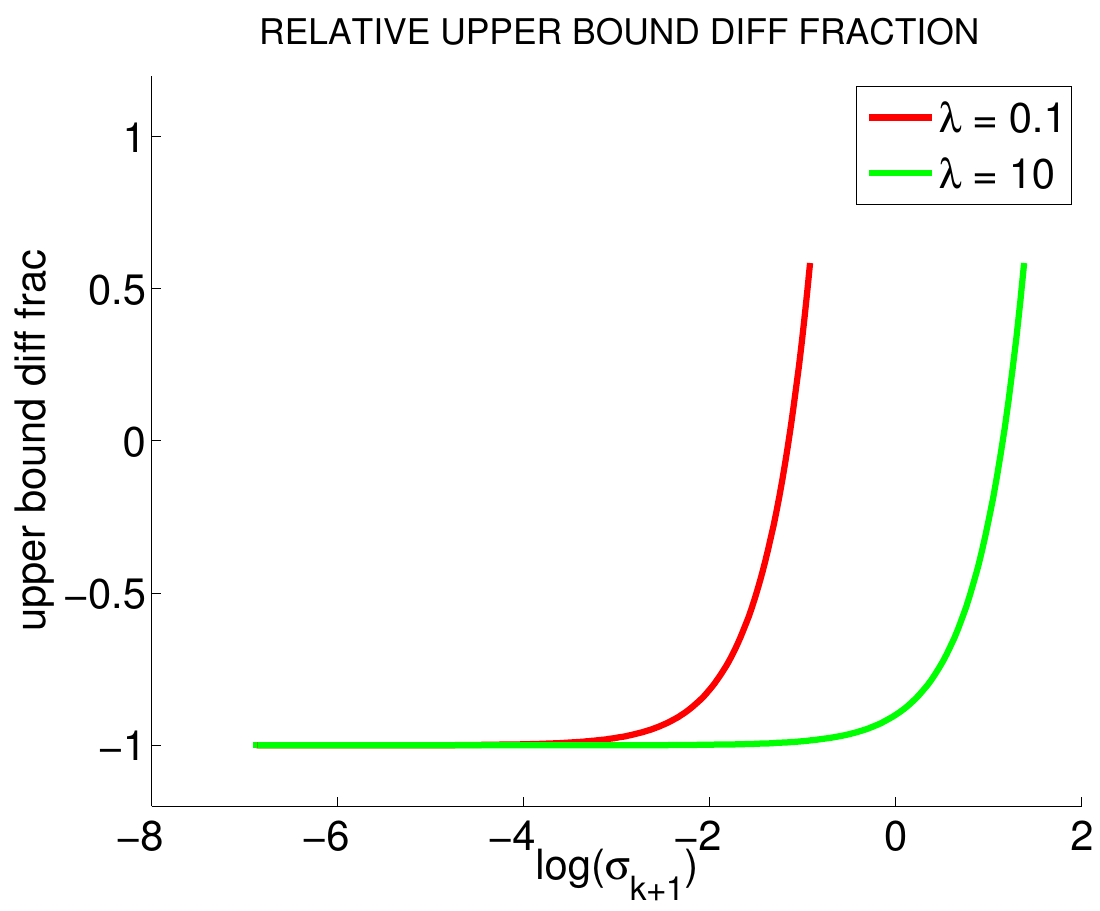}
}
\caption{Relative difference between upper bounds for the errors for approximate solutions 
$\tilde{x_1}$ and $\hat{x_1}$ (fraction \eqref{eq:tildex_hatx_ub_comp_eq}) 
as a function of different values of $\sigma_{k+1}$.}
\label{fig:tildex_hatx_ub_comp_fig}
\end{figure*}

\newpage
\section{Numerical Experiments}
\label{sec:Numerics}
In this section, we give some numerical examples to discuss and illustrate the 
approximation techniques we have discussed. We will use both synthetic data 
and matrices from the seismic tomography application which we have previously referred to 
in order to illustrate the effect of wavelet thresholding and low rank SVD based compression. 

\subsection{Examples with Synthetic Data}
We use three different synthetic matrix types, which we denote by 
$A_{(1)}$, $A_{(2)}$, and $A_{(3)}$. The matrices are of size $1000 \times 1500$, small 
enough to be easily handled in full, but large enough for randomization techniques to work. 
Matrix $A_{(1)}$ is constructed via the 
reverse SVD construction $A_{(1)} = U \Sigma V^T$
where $U$ and $V$ are taken to be orthonormal Gaussian random matrices 
and the singular values in $\Sigma$ are logspaced 
between $10^0$ and $10^{-4}$. That is, the decay 
of singular values of $A_{(1)}$ is relatively fast. Matrix $A_{(2)}$ is a different 
kind of matrix, whose rows are permuted vectorized images. It is constructed 
by choosing at random, one of five images for each row, vectorizing the image 
and then using a randomized permutation of its vector form as a row of the matrix. 
Matrix $A_{(3)}$ is also constructed from the same vectorized images, but its rows are not 
randomly permuted vectors but rather vectors rearranged in a continuous way with 
overlooping boundaries, where we choose at random a starting index within the image vector and 
then go to the end of the array, looping back to the beginning and proceeding in order until we 
have $n$ elements. 

We now comment on the wavelet compressibility of each matrix. By ``wavelet 
compressible'' we mean that the rows of the matrices satisfy the relation 
\eqref{eq:wavelet_approx_for_x}. In our case, we apply the one dimensional 
CDF $9-7$ wavelet transform to each row vector and threshold out all but $\frac{1}{3}$ of 
the largest coefficients by absolute magnitude.
It should be apparent that the rows of $A_{(1)}$ are not readily wavelet 
compressible (as they are vectors picked at random having no apparent structure), 
some but not all of the rows of $A_{(2)}$ are wavelet compressible (as they are image 
vectors re-arranged in random order so that only rows arranged by chance in such 
a way as to have some structure are expected to be compressible), 
and virtually all rows of $A_{(3)}$ are readily wavelet compressible (they are vectorized 
images with a random starting index, but the pixel structure of the original image is preserved). 

We start by constructing the compressed 
wavelet matrices $M_{(1)}$, $M_{(2)}$, and $M_{(3)}$,  
keeping a third of the nonzero wavelet coefficients in the thresholding. We then compare 
the errors induced in approximating 
matrix vector operations with the full matrices $A_{(1)},A_{(2)},A_{(3)}$ via these 
compressed matrices using the relations \eqref{eq:wavelet_approx_matvec_ops}. 
For $100$ Gaussian random vectors 
$x \in \mathbb{R}^{1500}$ and $y \in \mathbb{R}^{1000}$ we 
compare, using \eqref{eq:wavelet_approx_matvec_ops}, the results of the 
operations $A_{(i)} x$ versus $M_{(i)} W^{-T} x$, 
$A_{(i)}^T y$ versus $W^{-1} M_{(i)}^T y$ and $A_{(i)}^T A_{(i)} x$ versus 
$W^{-1} M_{(i)}^T M_{(i)} W^{-T} x$ for $i=1,2,3$ corresponding to the three matrices. 
The resulting percent errors 
(i.e. fractions such as $E = 100 \frac{\|A_{(1)} x - M_{(1)} W^{-T} x \|}{\|A_{(1)} x\|}$ and 
likewise for the other operations) are plotted in column 2 of Figure \ref{fig:synthetic_data1}, 
where we plot median values over $10$ trials and in 
each trial utilize $100$ Gaussian random vectors $x$ and $y$. Notice that in the first case, 
where the matrix was chosen to not compress well, the errors are high. In the 
other two cases, the operations with matrices $A_{i}^T A_{i}$ ($i=2,3$) are approximated 
well. It is especially interesting that this is the case for the second matrix, where some 
of the rows are not wavelet compressible. 

Next, we use the $M_{(i)}$ matrices to compute the 
low rank SVD of $A_{(i)}$ with $k=200$ to achieve further size reduction. That is, we use 
the randomized SVD algorithm previously shown where we utilize matrix $M$ to approximate 
all necessary operations with $A$. Once the low rank SVD components $U_k$, $\Sigma_k$, 
and $V_k$ are obtained, we compare the same operations with $A$ as before to the 
approximation via the low rank SVD:
\begin{equation*}
Ax \mbox{  to  } U_k \Sigma_k V^T_k x \mbox{      ;      } A^T y \mbox{  to  } V_k \Sigma_k U_k^T y \mbox{      ;      } 
A^T A x \mbox{  to  } V_k \Sigma^2_k V_k^T x
\end{equation*} 
For comparison, for each matrix, we also compute the low rank SVD with the full $A_{(i)}$, without 
using $M_{(i)}$ to approximate matrix-vector operations. We expect this to give a more 
accurate low rank SVD. The corresponding percent errors (such as 
$E = 100 \frac{\|A_{(1)} x - U_k \Sigma_k V^T_k x\|}{\|A_{(1)} x\|}$) for the operations are shown in 
column 3 of Figure \ref{fig:synthetic_data1} below. In all cases, the plotted lines are median values 
obtained over $10$ separate trials. The result is interesting but somewhat expected because of the use 
of randomization in the computation: the low rank SVD computed via $M$ produces similar 
results to that computed via $A$ even if for some particular row vectors of $A$, 
the relation \eqref{eq:wavelet_approx_for_x} is not satisfied. However, notice that this 
does not hold for matrix $A_{(1)}$ whose rows are not wavelet compressible. From the last 
column of Figure \ref{fig:synthetic_data1}, we see differences between the results of the low rank 
SVD computed with $A_{(1)}$ and with $M_{(1)}$.

Next, we make a synthetic data vector $x$, and use the three matrices 
$A_{(1)}, A_{(2)}, A_{(3)}$ to construct the right hand side $b_{(i)} = A_{(i)} x + \nu$ 
with $\nu$ a Gaussian random noise vector (we choose to use $10$ percent noise relative 
to the norm of $b_{(i)}$). We then try to reconstruct $x$ with the various approximation 
schemes by computing solutions to the Tikhonov problem with smoothing \\
$\left( A^T A + \lambda_1 I + \lambda_2 L^T L \right) \bar{x} = A^T b$, where 
for $L$ we take the tridiagonal matrix with elements $(-1,2,-1)$. In  
Figure \ref{fig:synthetic_data2}, we present the results of various approximation 
schemes we described. In particular, we plot the following solutions:
\begin{eqnarray*}
(A^T A + \lambda_1 I + \lambda_2 L^T L)\bar{x} &=& (A^T b) \\
(W^{-1} M^T M W^{-T} + \lambda_1 I + \lambda_2 L^T L)^{-1}x_{wav} &=& (W^{-1} (M^T b)) \\
(V_k \Sigma_k^2 V_k^T + \lambda_1 I + \lambda_2 L^T L)x_{svd1} &=& (V_k \Sigma_k U_k^T b) \\
(V_k \Sigma_k^2 V_k^T + \lambda_1 I + \lambda_2 L^T L)x_{svd2} &=& (A^T b) \\
(\Sigma_k^2 + \lambda_1 I_k + \lambda_2 V_k^T L^T L V_k) y_{svd3} &=& (\Sigma_k U_k^T b) \quad 
\mbox{;} \quad x_{svd3} = V_k y_{svd3} \\
\end{eqnarray*}
In each case, we loop over $40$ linearly spaced values of $\lambda_1$ and $\lambda_2$ 
(effecting the degree of norm and smoothing penalty, respectively) and choose the values so that the residual 
norm $||A x_{sol} - b||_2$ of the solution is closest to the norm of the noise vector $||\nu||_2$. In  
Figure \ref{fig:synthetic_data2}, we plot the on the first row the true solution vector $x$ 
followed by the solutions obtained using the full matrix $A$. On the second row, we plot for 
each matrix type ($A_{(i)}$), the residual norms of the different solutions relative to the 
noise norm. On rows three to six, we plot the different solutions obtained with the various 
approximations schemes for the matrices ($A_{(i)}$). We observe that in each case, we can 
obtain reasonable reconstructions using the approximation schemes we introduced. The wavelet 
compressed approach is the most accurate with respect to the full solution, followed by 
the two svd methods. The $k \times k$ method ($x_{svd3}$) produces a suitable reconstruction 
for the third matrix $A_{(3)}$, whose rows are all wavelet compressible. On the 
other hand, the $k \times k$ method does not work well for the first two matrices.  
In summary, the synthetic data examples show that in many practical cases, 
wavelet compression and low rank SVD techniques can be used together to obtain approximate 
regularized solutions, with the SVD matrices obtained using operations with the wavelet 
compressed matrix instead of the original matrix.


\begin{figure*}[ht!]
\centerline{
\includegraphics[scale=0.29]{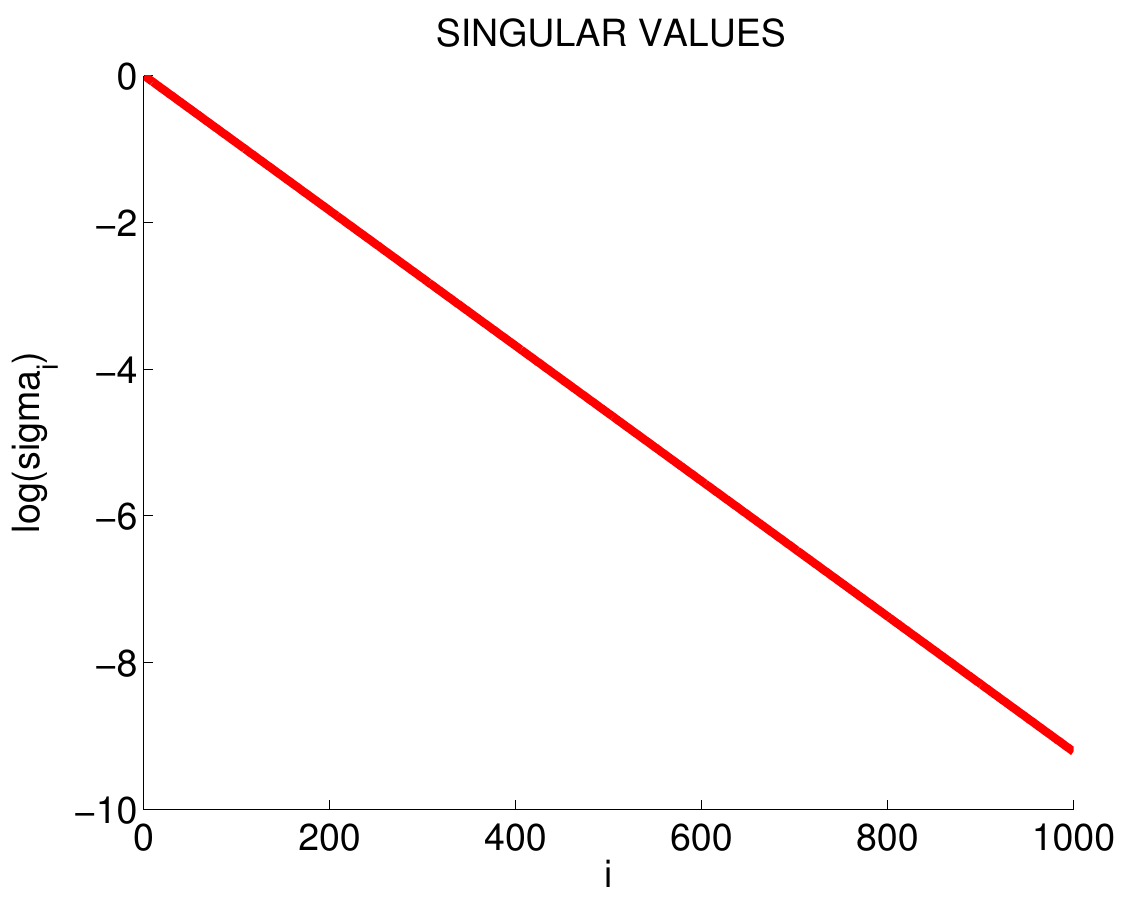}
\includegraphics[scale=0.29]{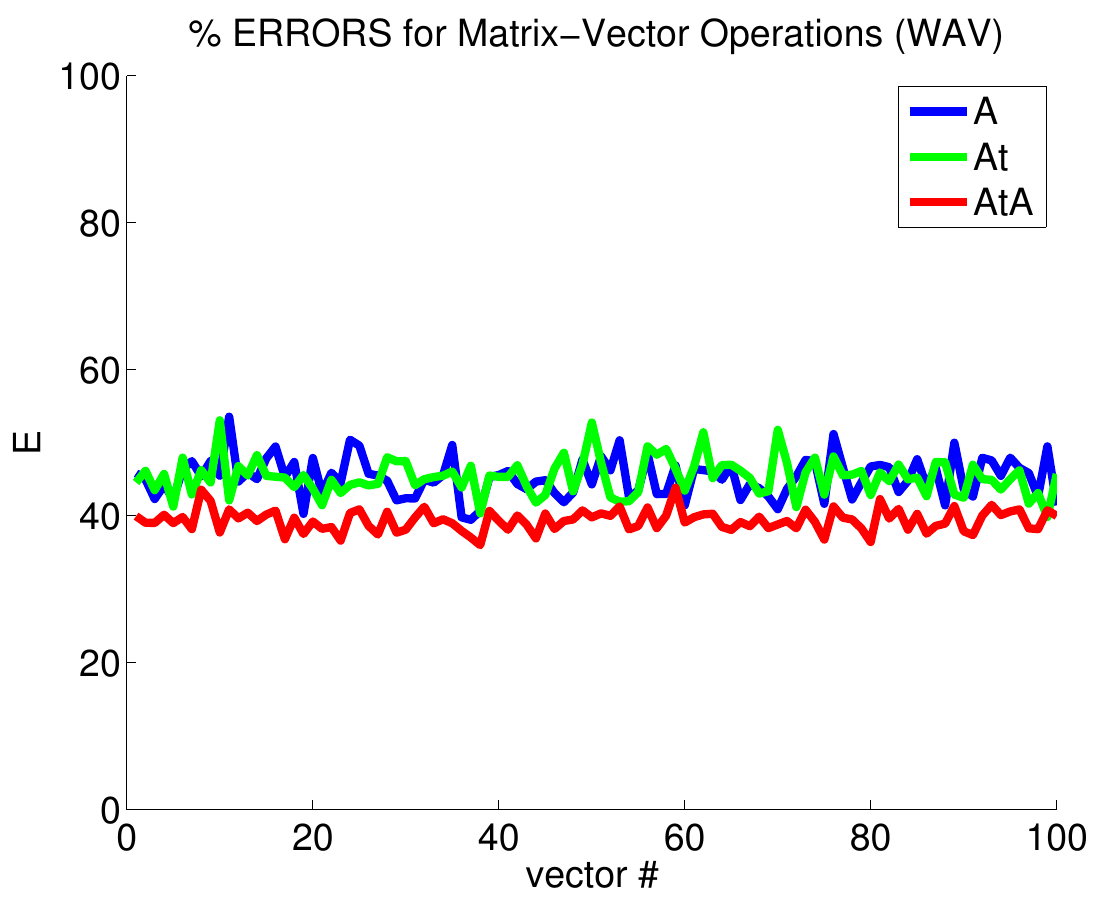}
\includegraphics[scale=0.29]{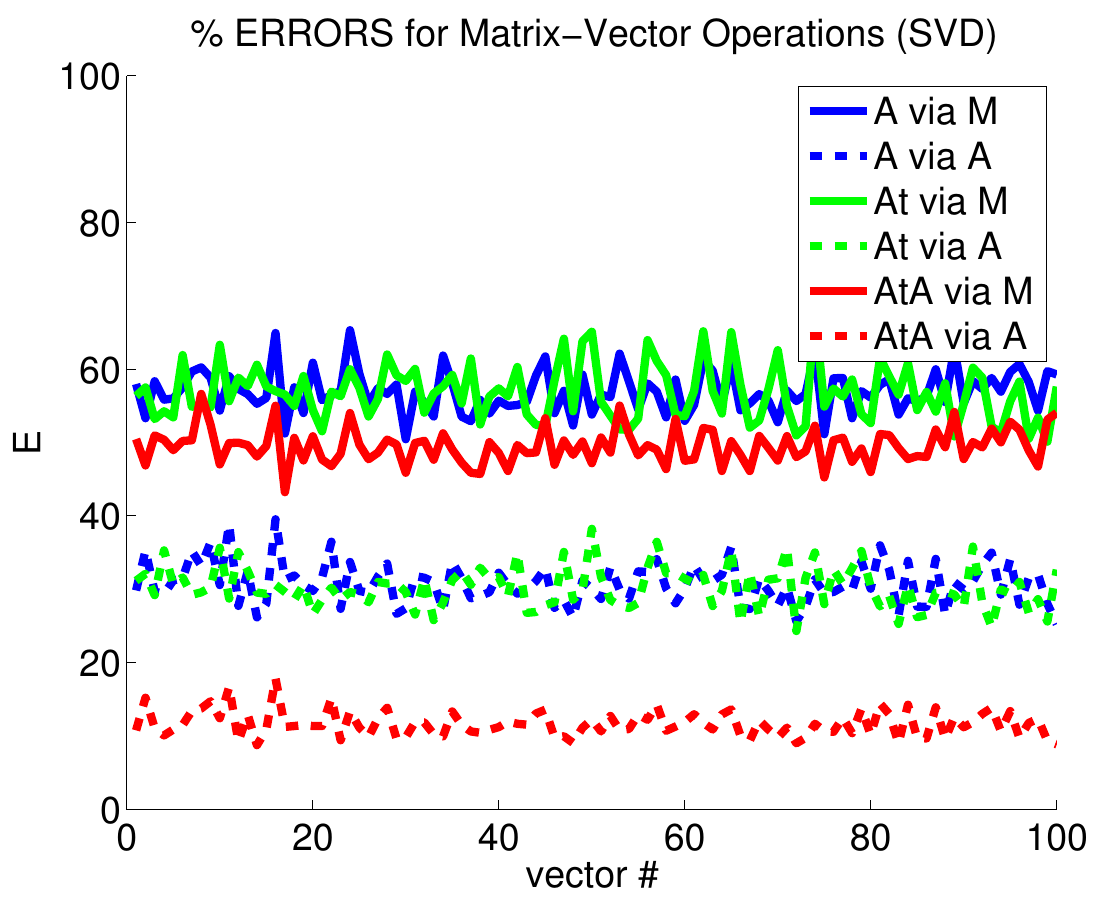}
}
\centerline{
\includegraphics[scale=0.29]{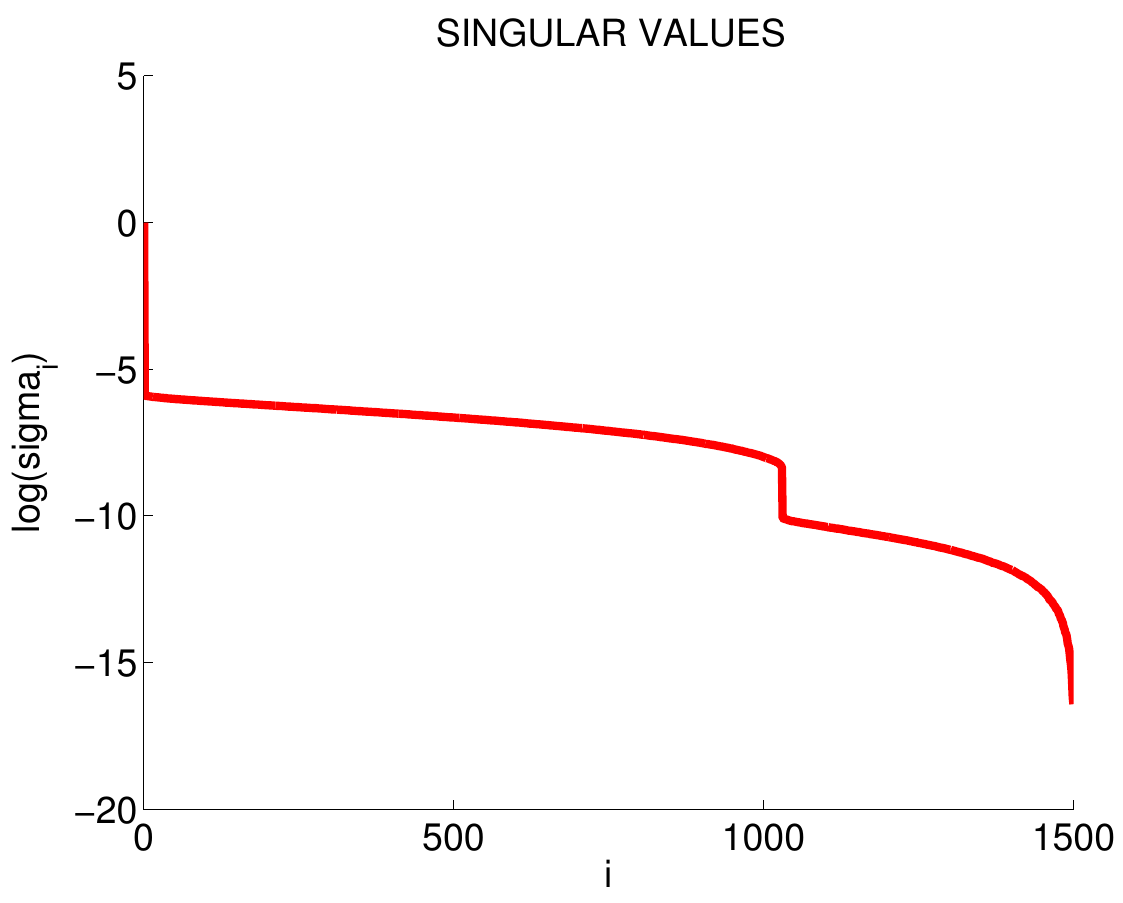}
\includegraphics[scale=0.29]{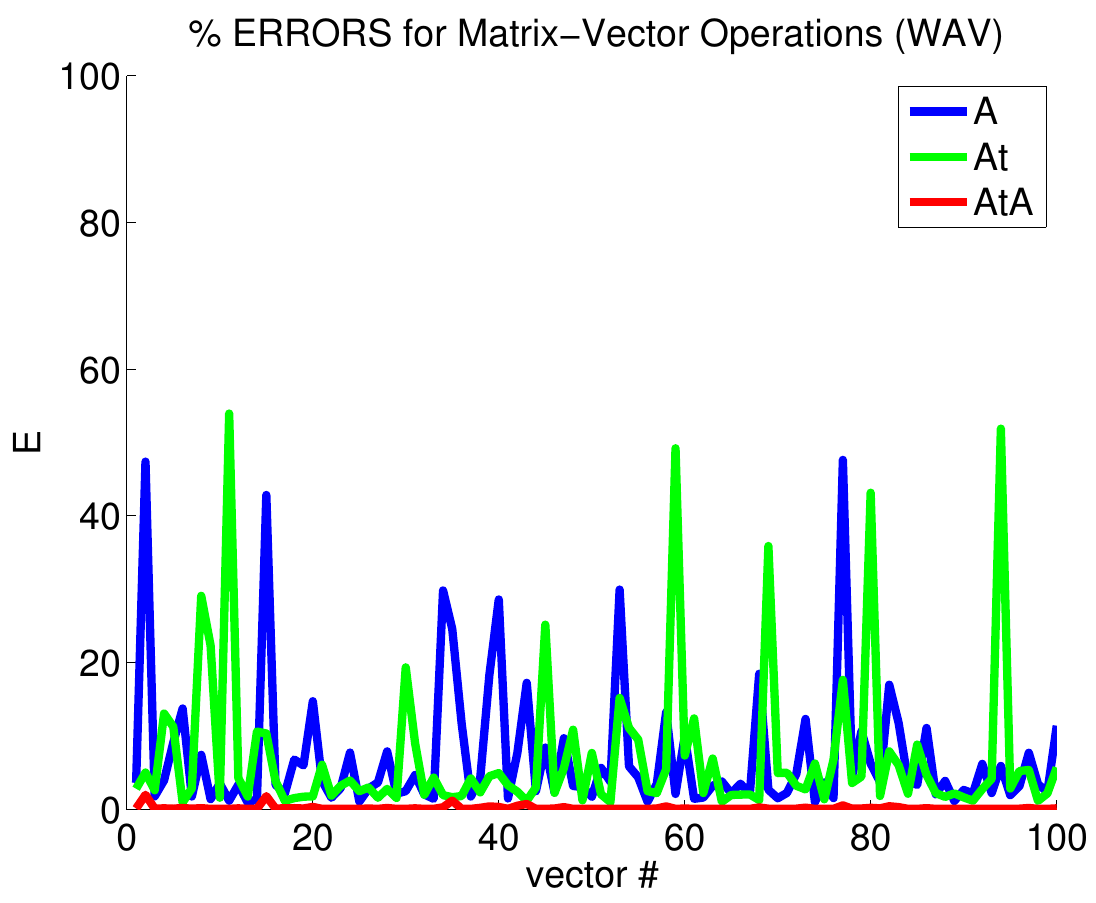}
\includegraphics[scale=0.29]{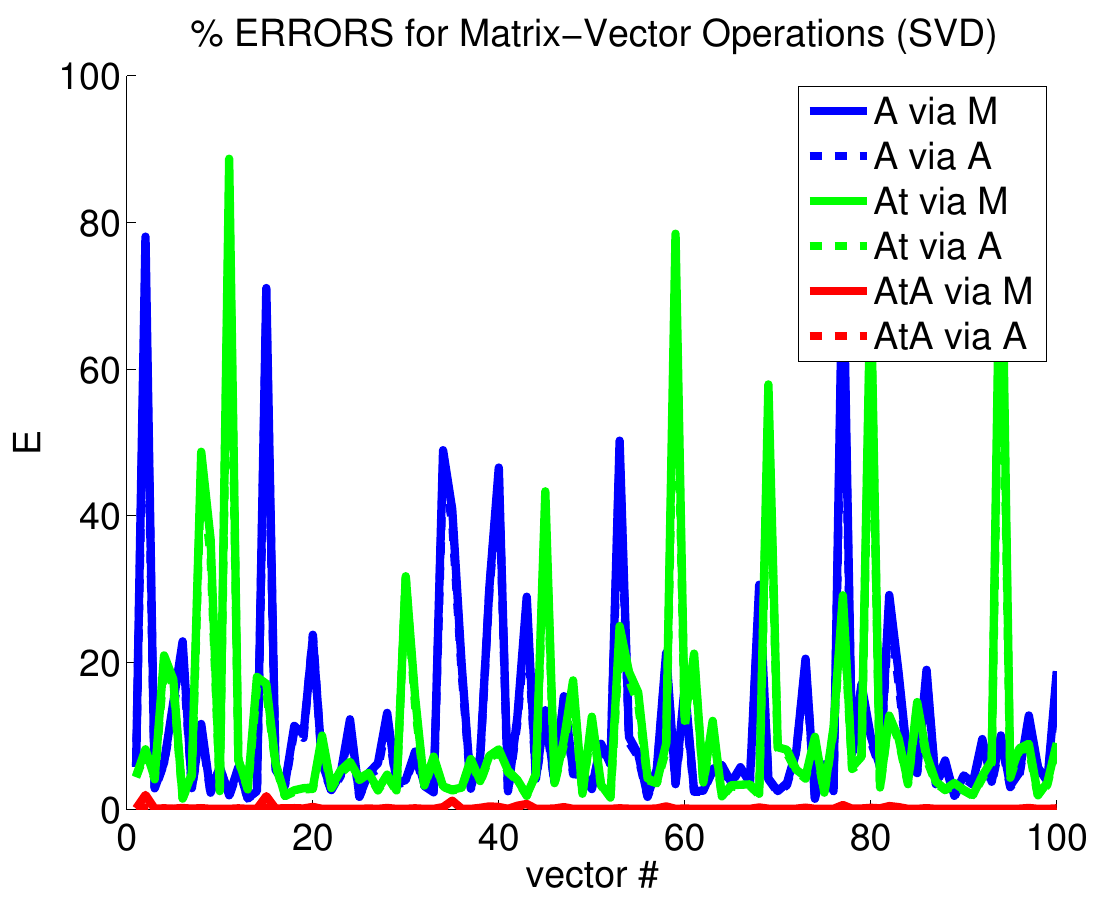}
}
\centerline{
\includegraphics[scale=0.29]{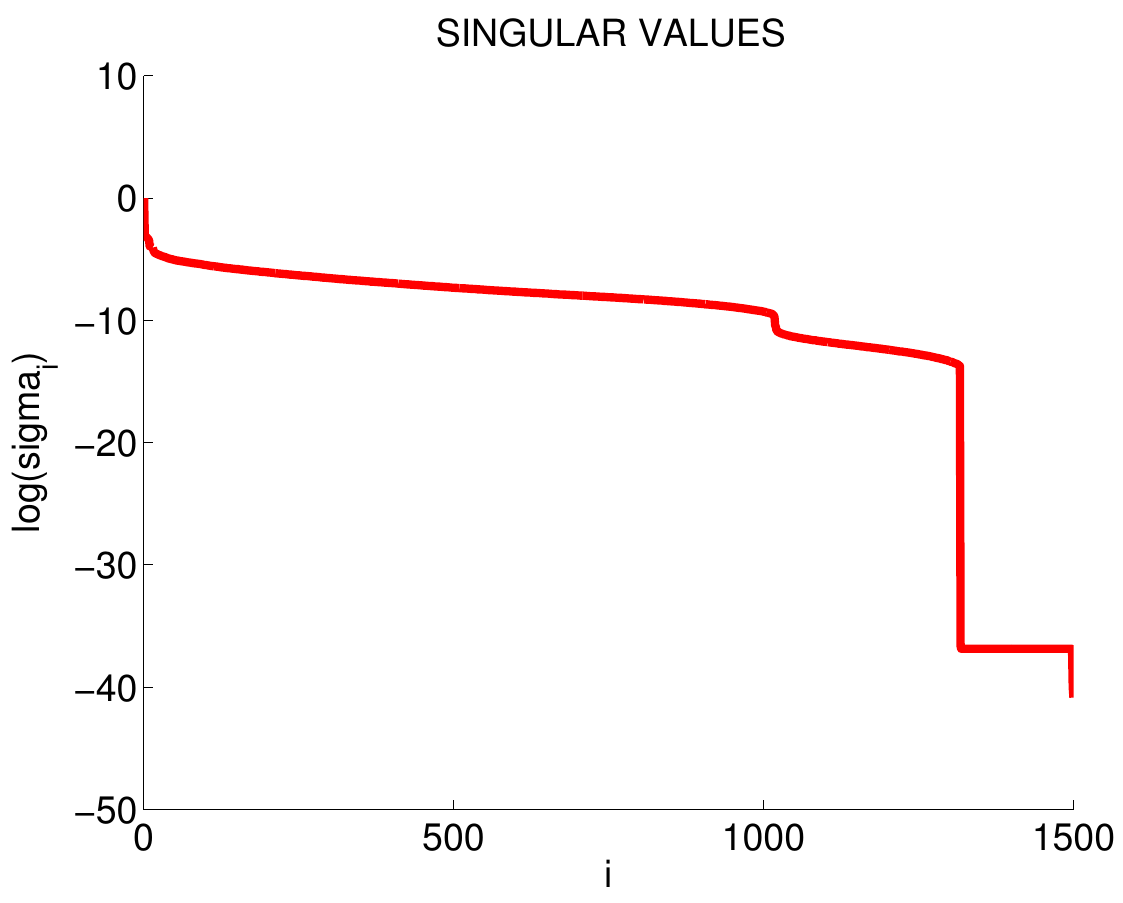}
\includegraphics[scale=0.29]{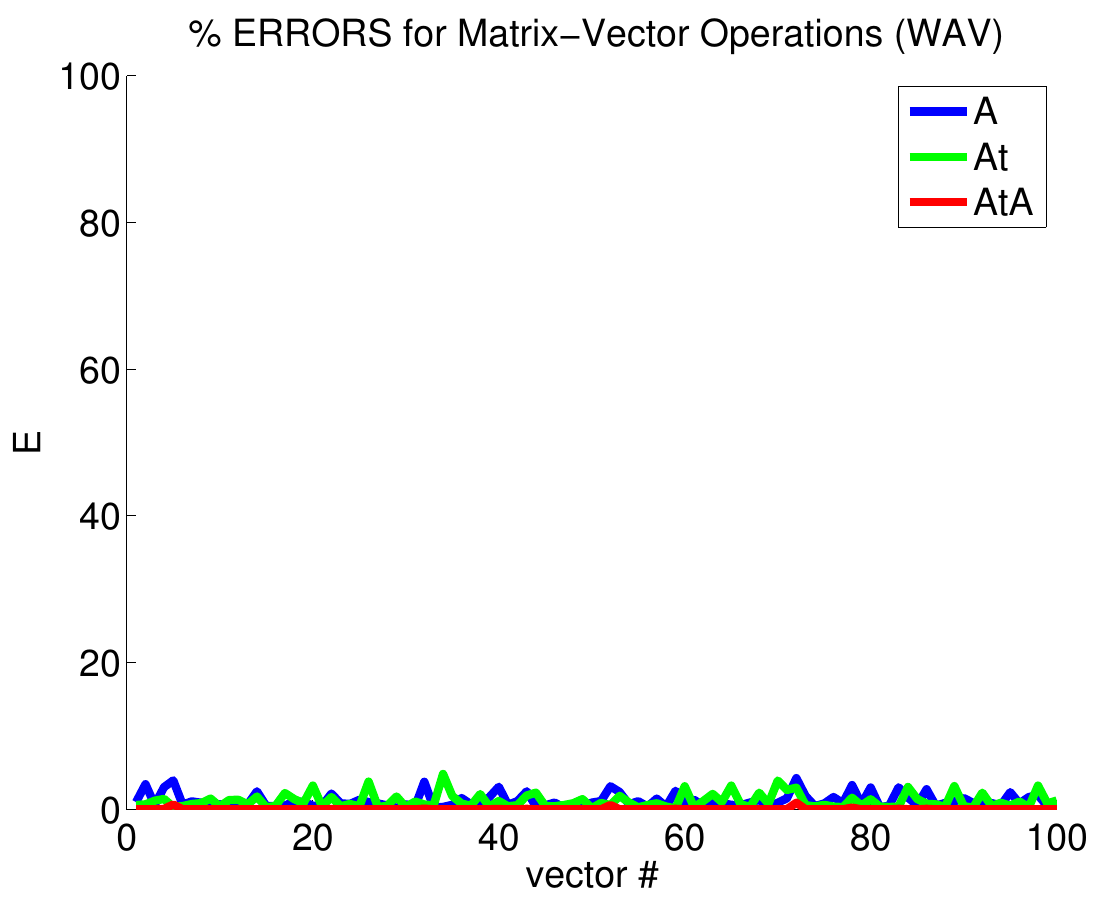}
\includegraphics[scale=0.29]{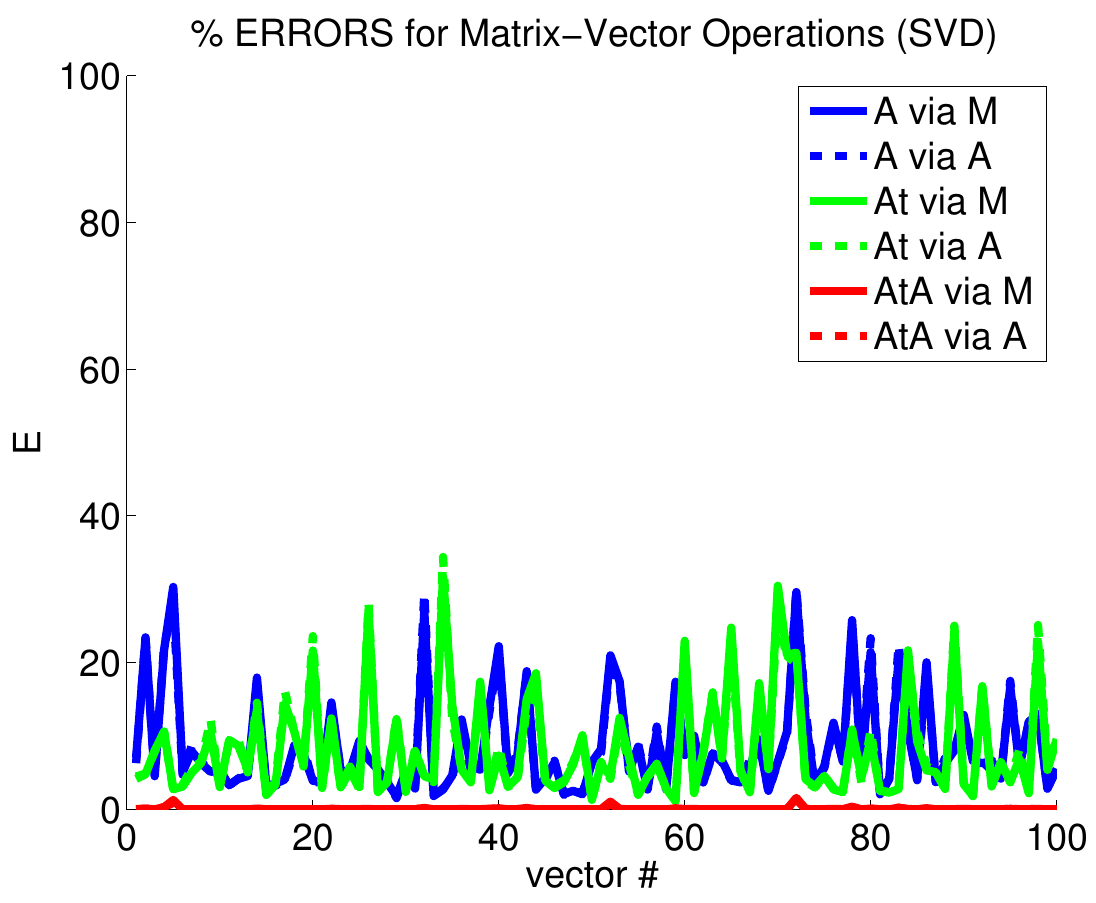}
}
\caption{Percent errors for approximating matrix-vector operations with $A_{(1)}$ (row 1), $A_{(2)}$ (row 2), 
and $A_{(3)}$ (row 3) via the wavelet compressed matrices $M_{(i)}$ and via the low rank SVD. 
Singular values (column 1), percent errors in matrix 
vector operations for $100$ Gaussian random vectors using wavelet compression (column 2) and low 
rank SVD (column 3). In column 3 we plot errors obtained via the low rank SVD approximation obtained 
using the full $A_{(i)}$ matrices and using the corresponding wavelet compressed $M_{(i)}$ matrices.}
\label{fig:synthetic_data1}
\end{figure*} 

\newpage

\begin{figure*}[ht!]
\centerline{
\includegraphics[scale=0.21]{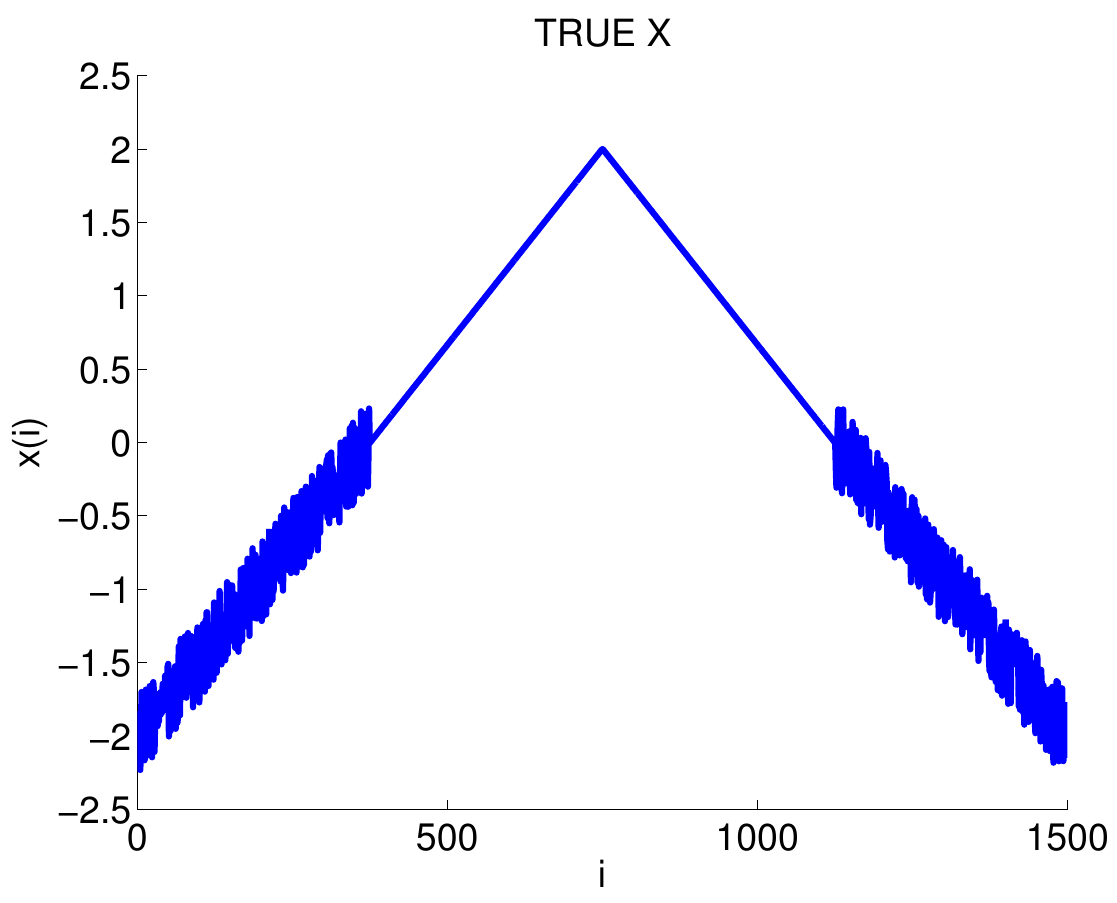}
\includegraphics[scale=0.21]{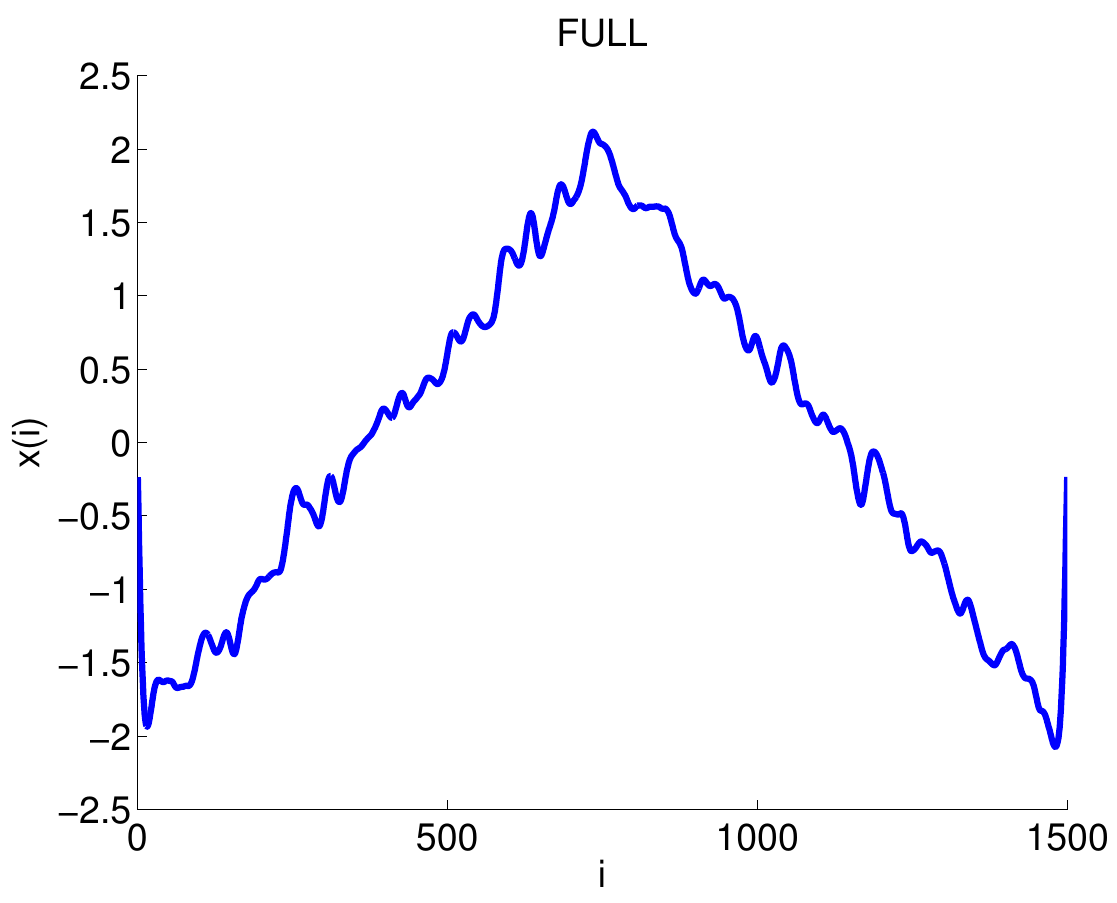}
\includegraphics[scale=0.21]{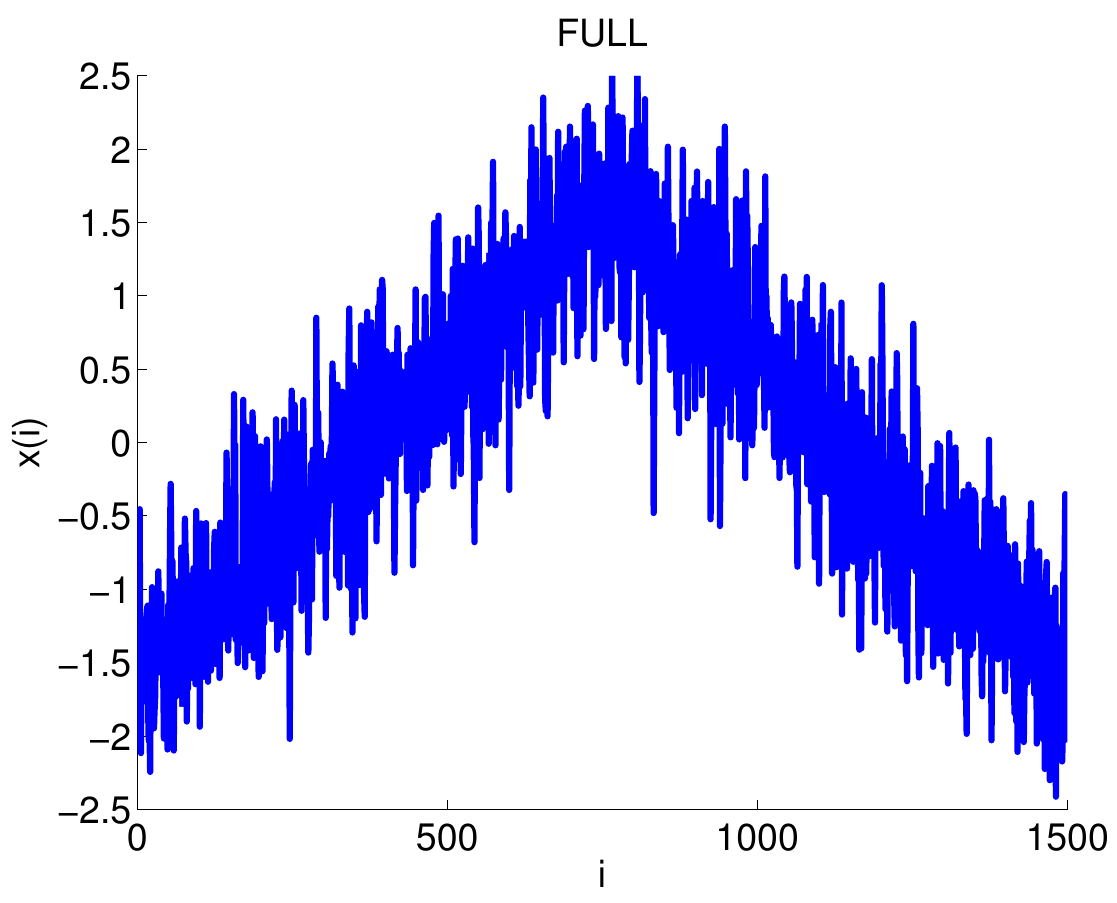}
\includegraphics[scale=0.21]{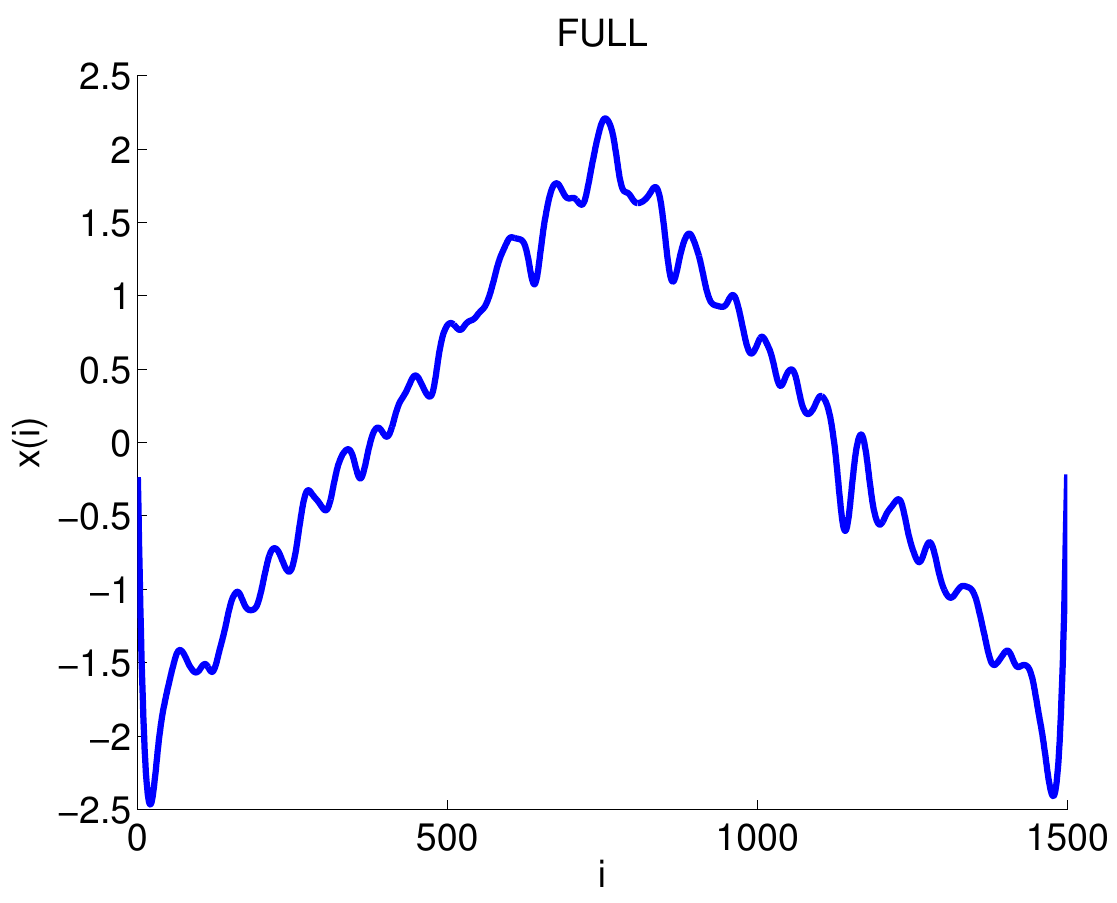}
}
\centerline{
\includegraphics[scale=0.30]{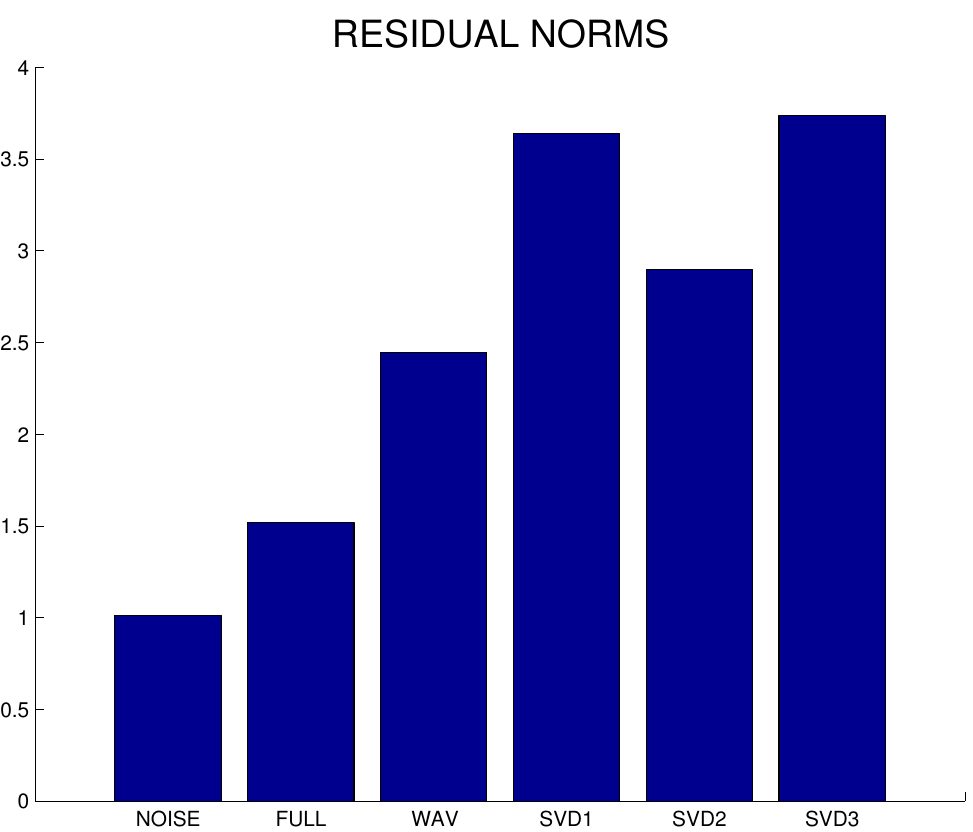}
\includegraphics[scale=0.30]{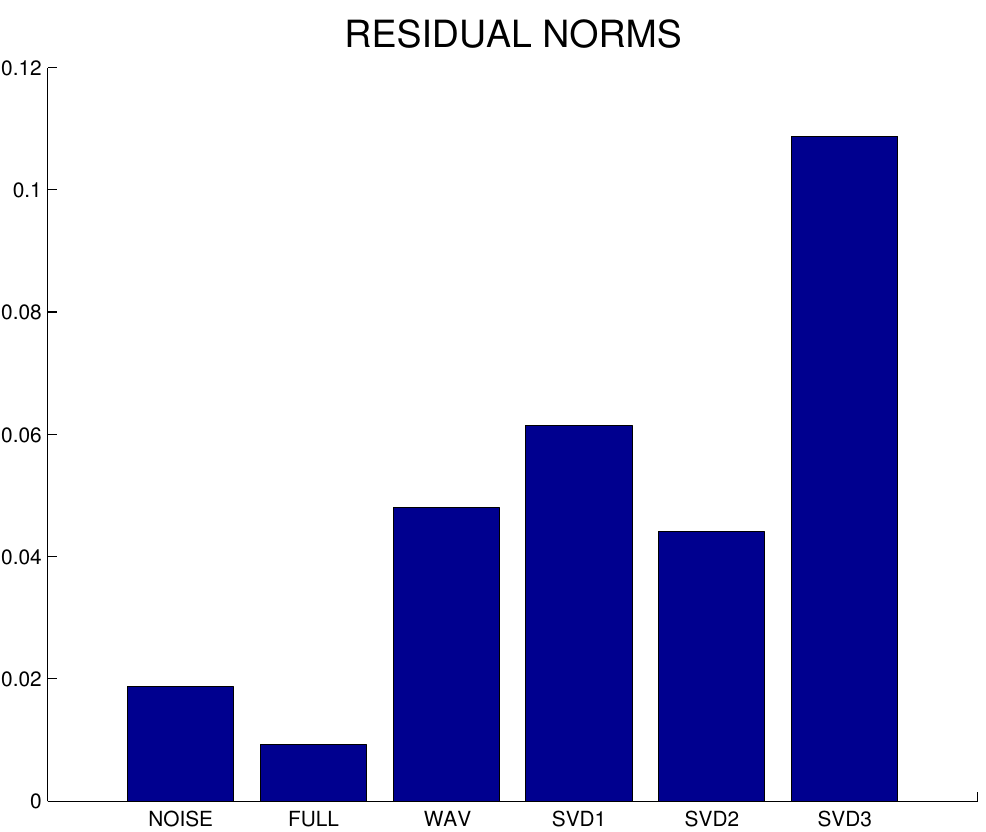}
\includegraphics[scale=0.30]{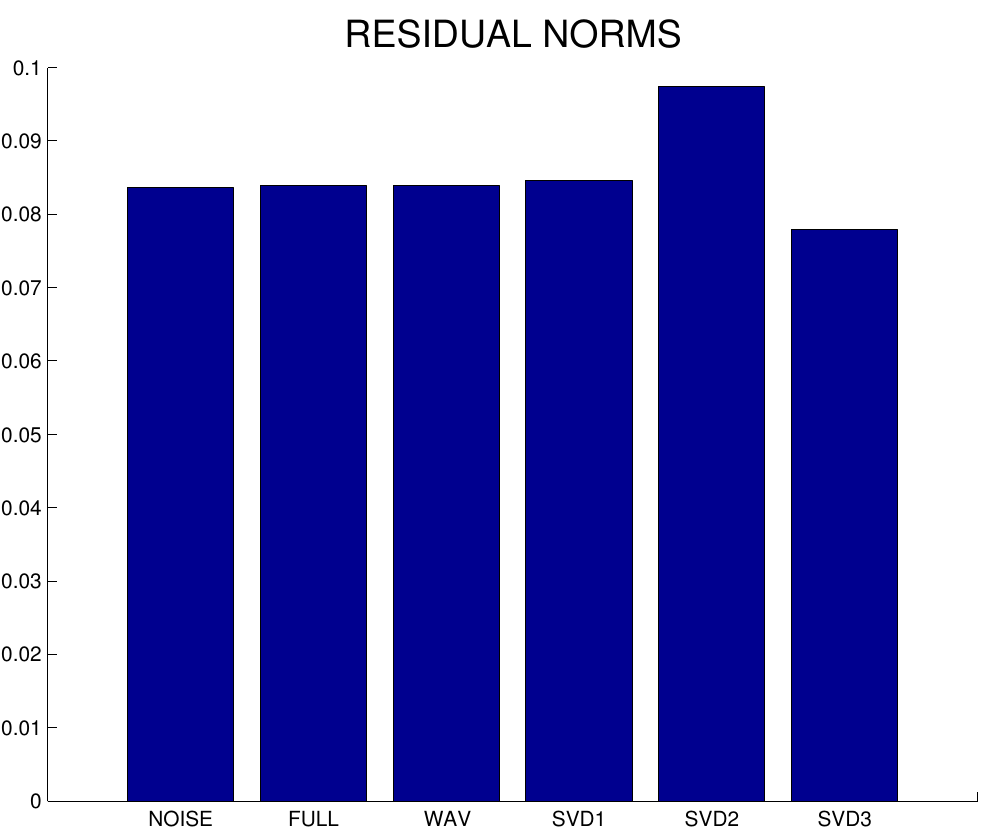}
}
\centerline{
\includegraphics[scale=0.21]{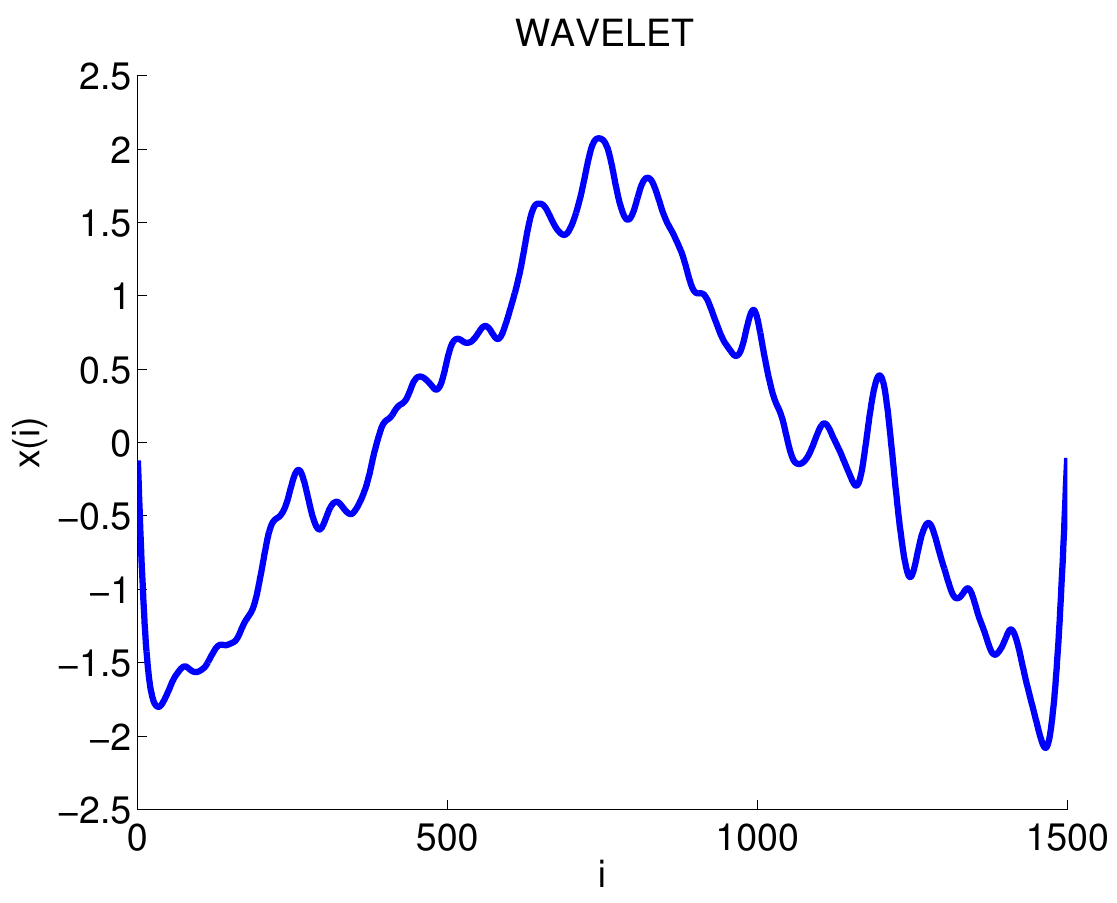}
\includegraphics[scale=0.21]{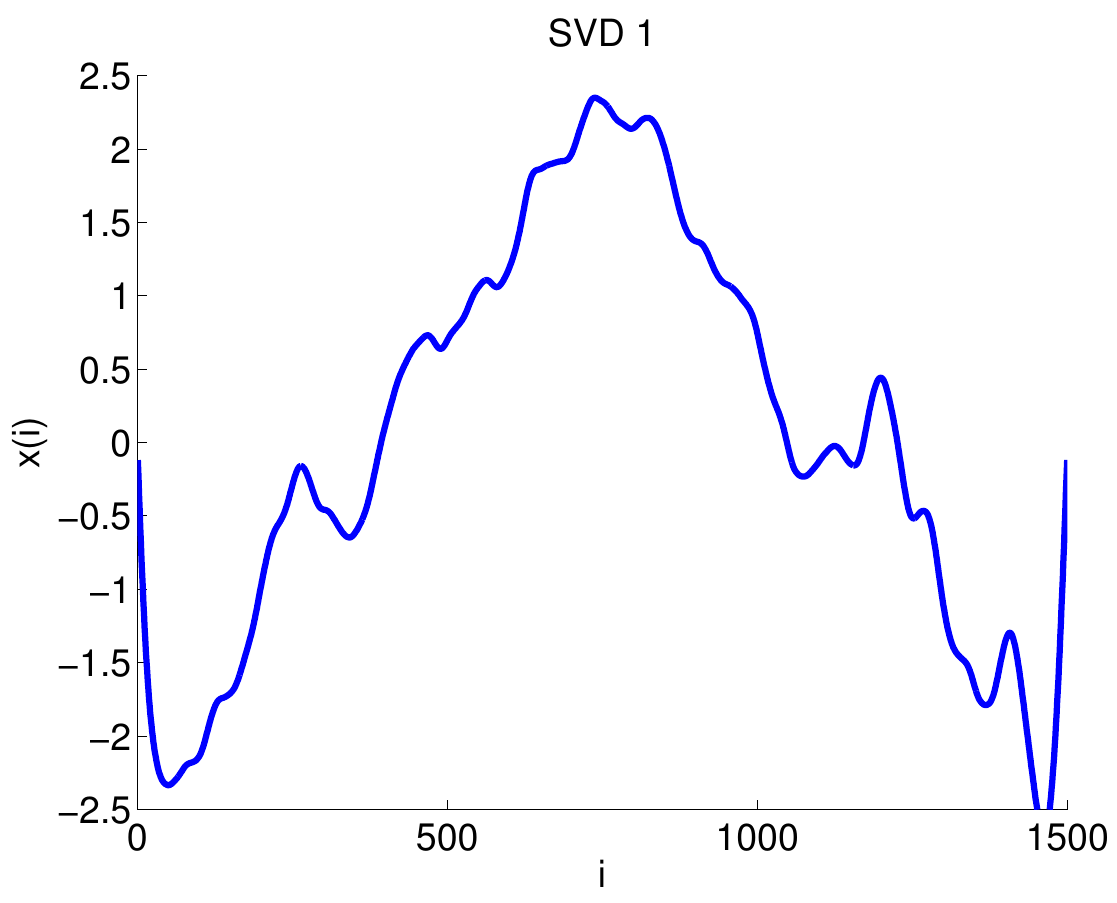}
\includegraphics[scale=0.21]{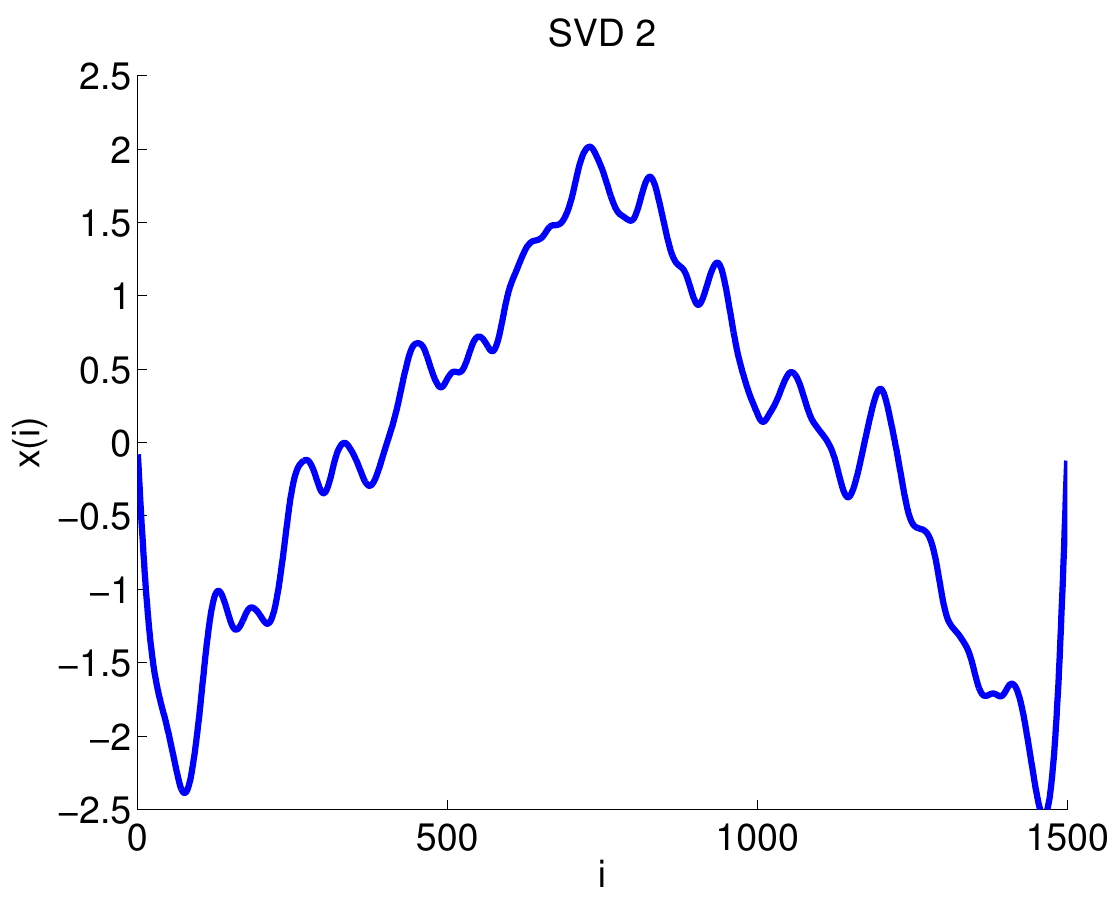}
\includegraphics[scale=0.21]{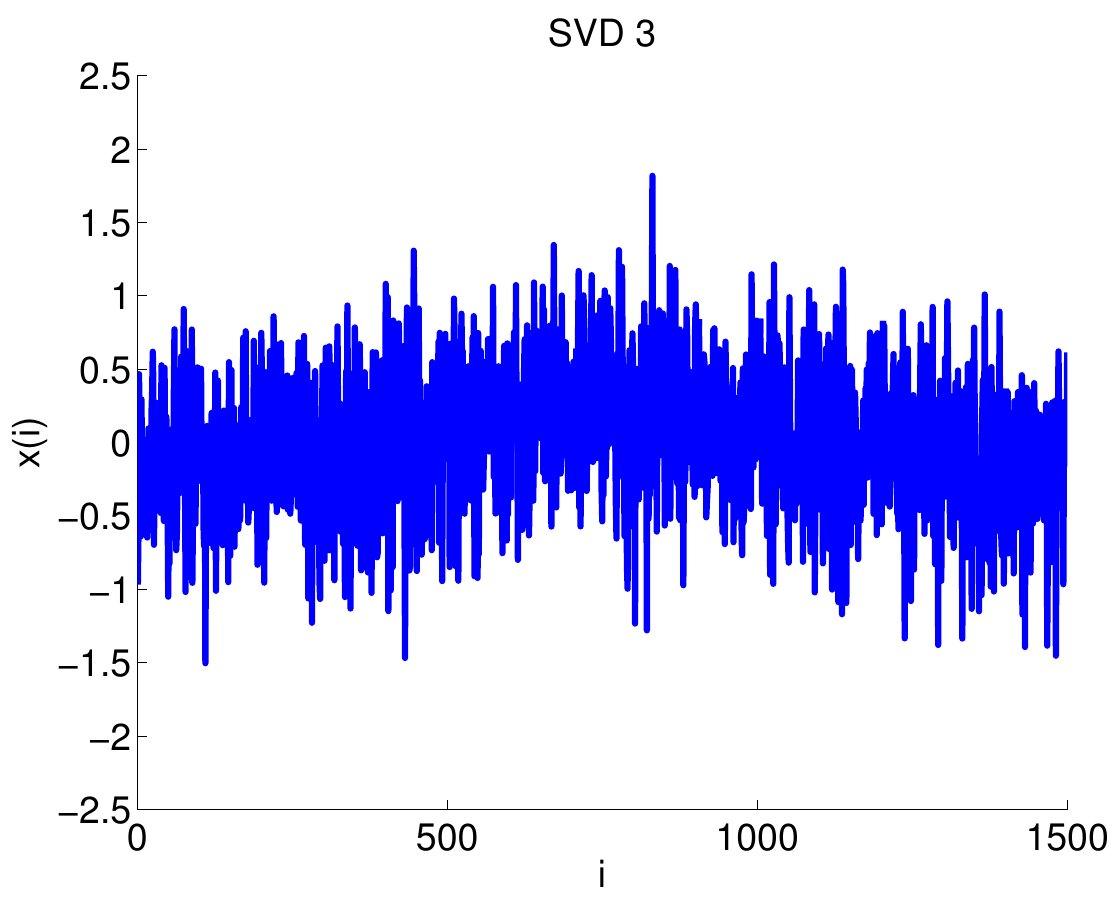}
}
\centerline{
\includegraphics[scale=0.21]{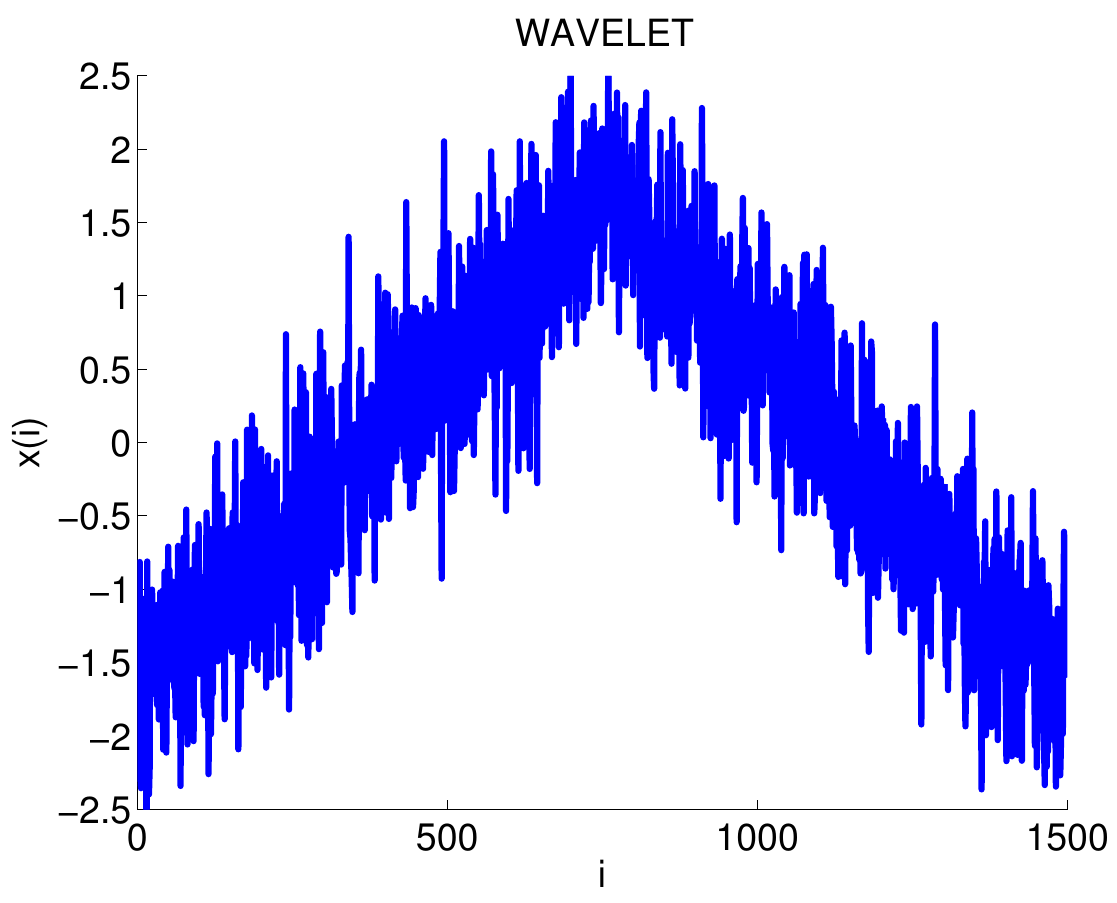}
\includegraphics[scale=0.21]{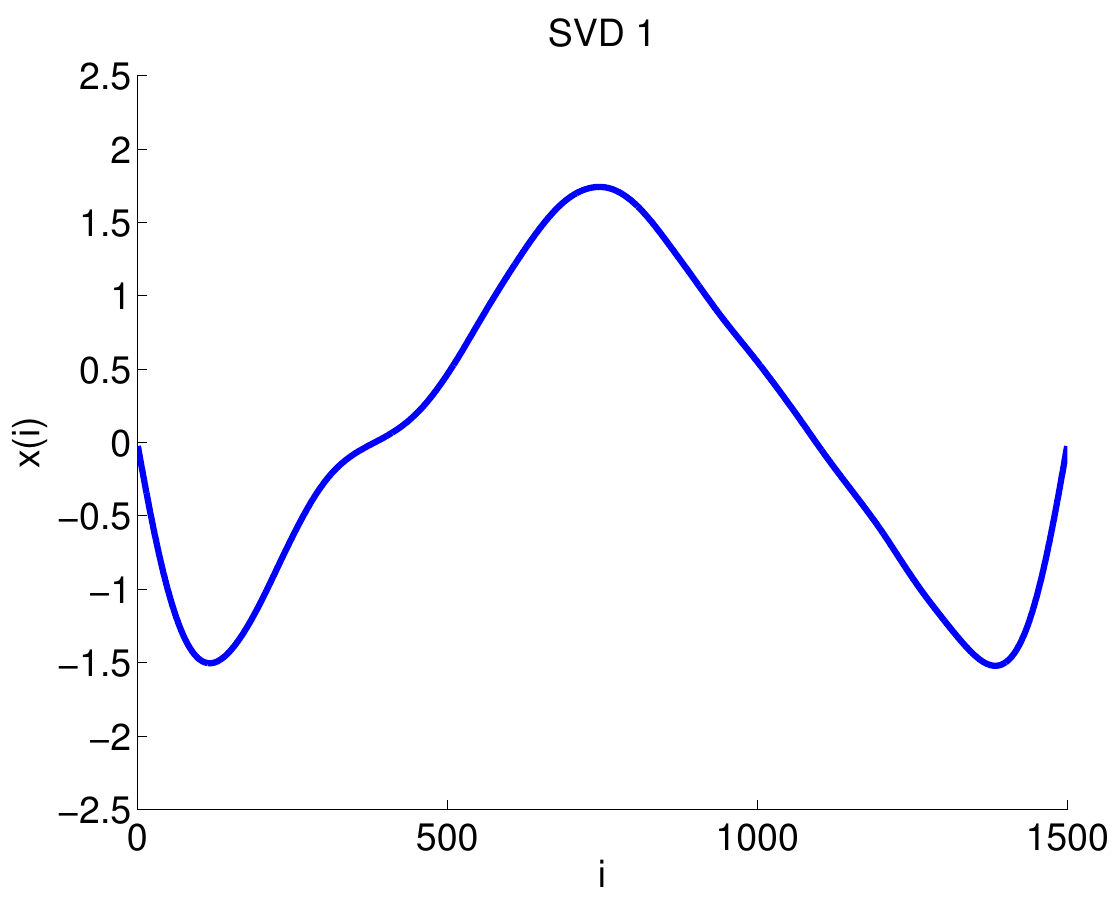}
\includegraphics[scale=0.21]{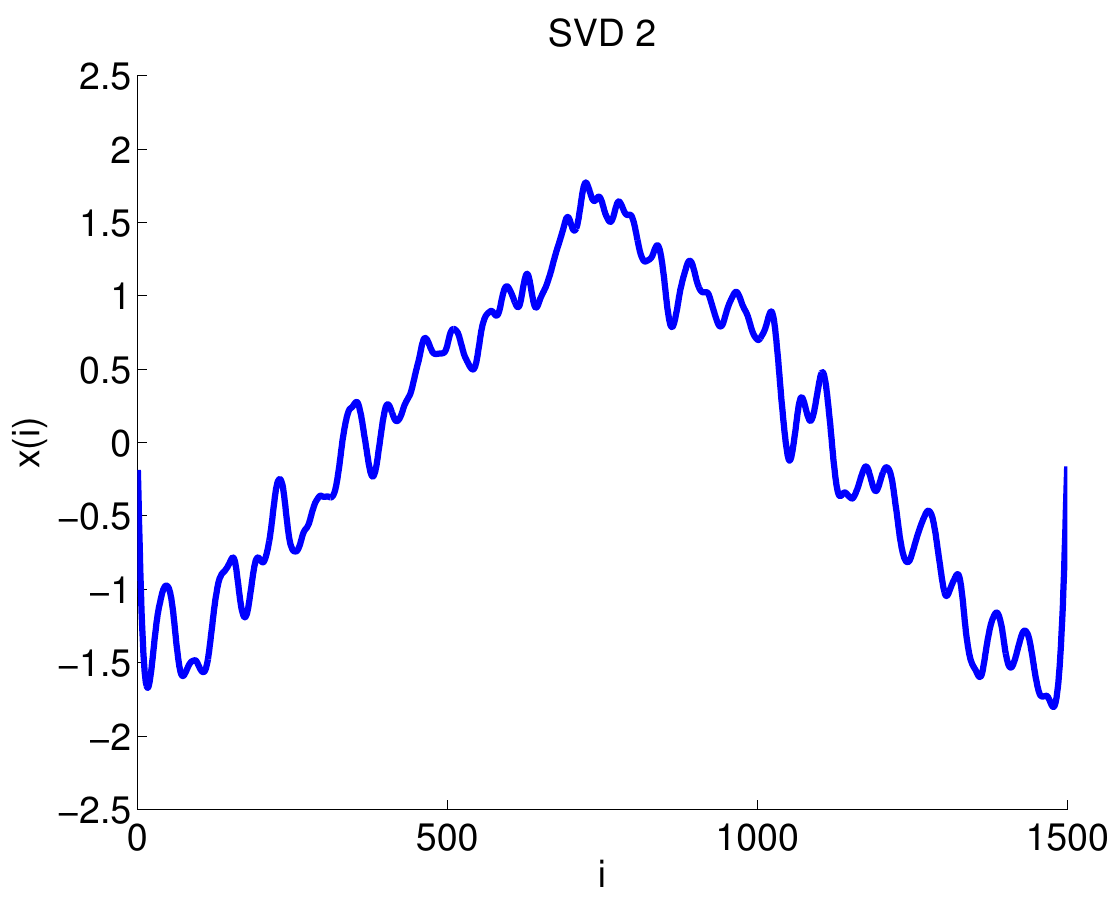}
\includegraphics[scale=0.21]{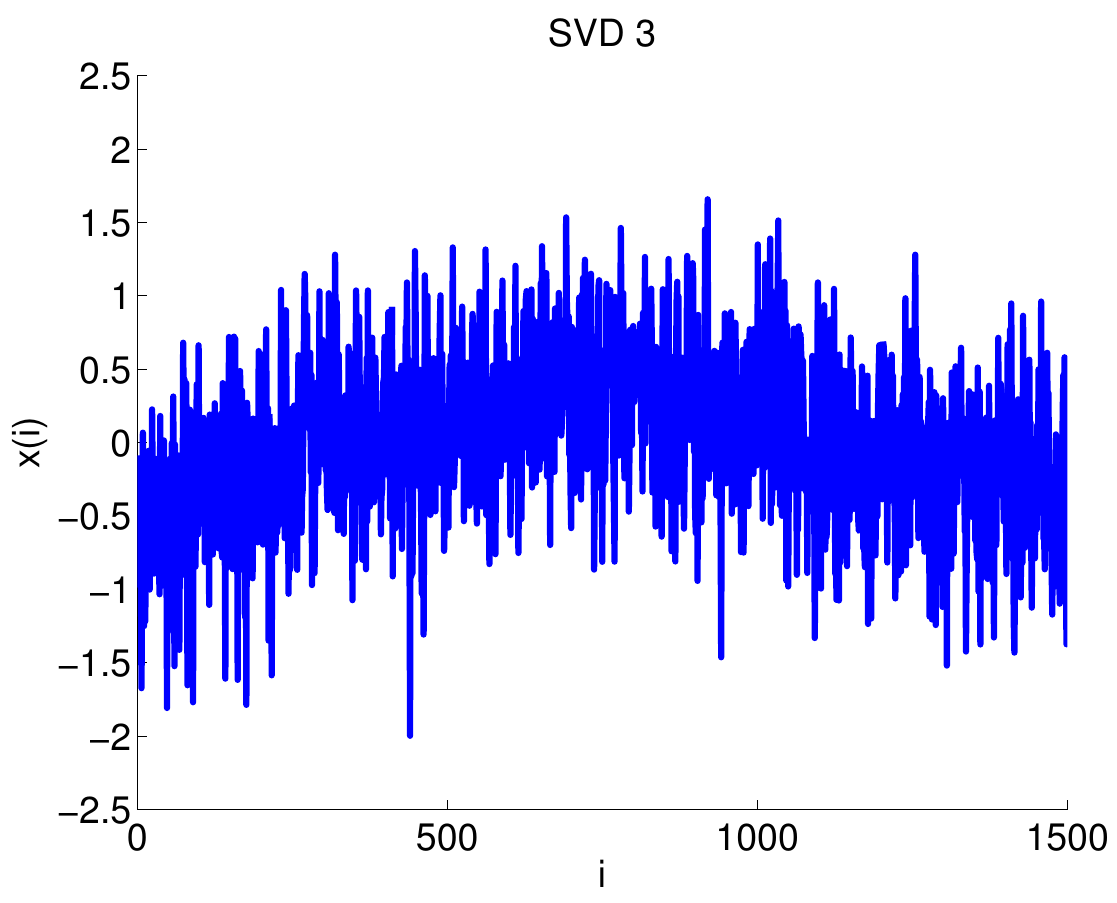}
}
\centerline{
\includegraphics[scale=0.21]{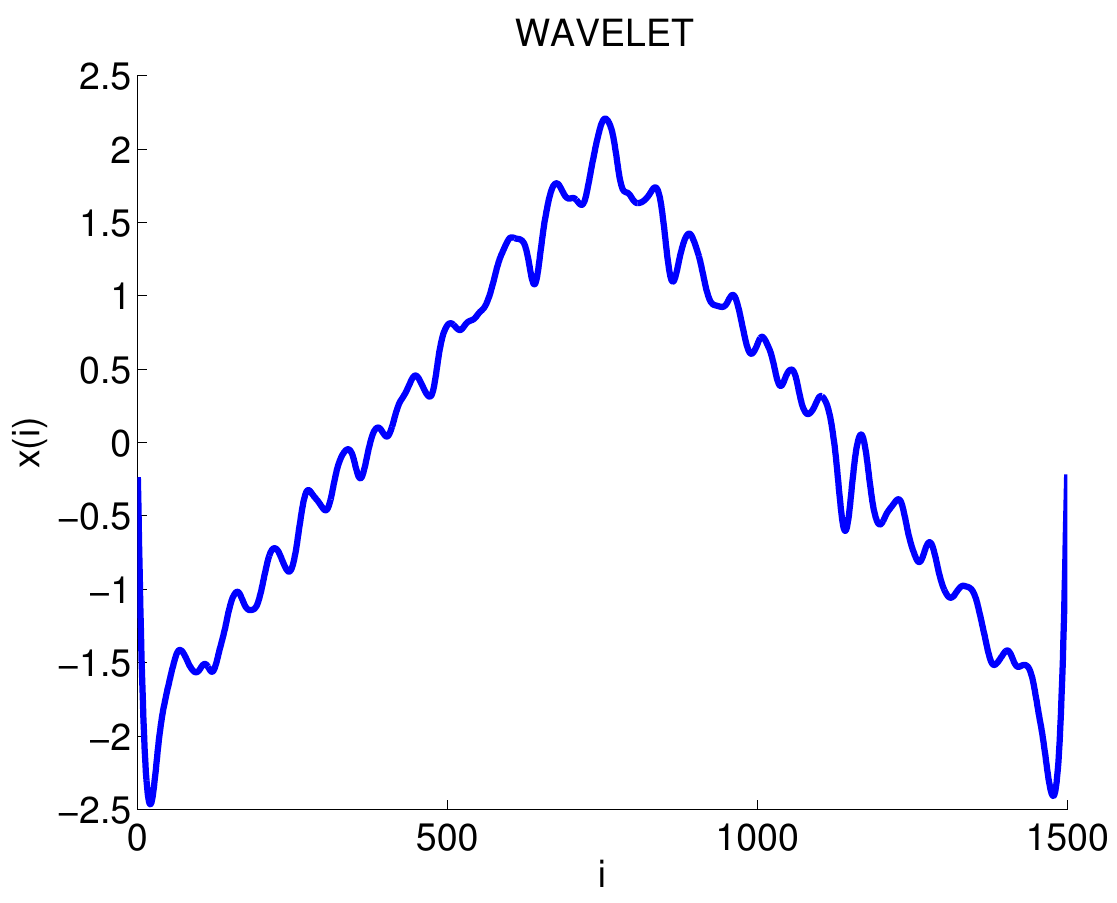}
\includegraphics[scale=0.21]{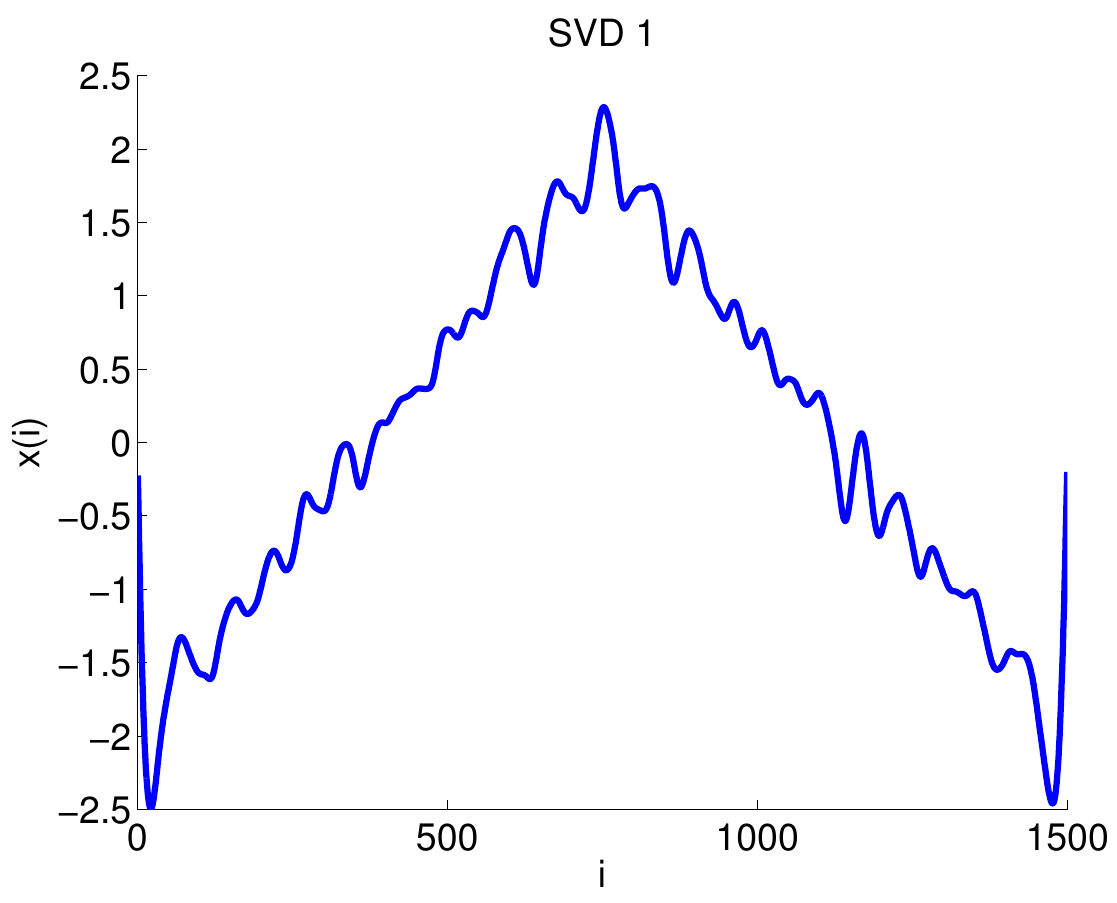}
\includegraphics[scale=0.21]{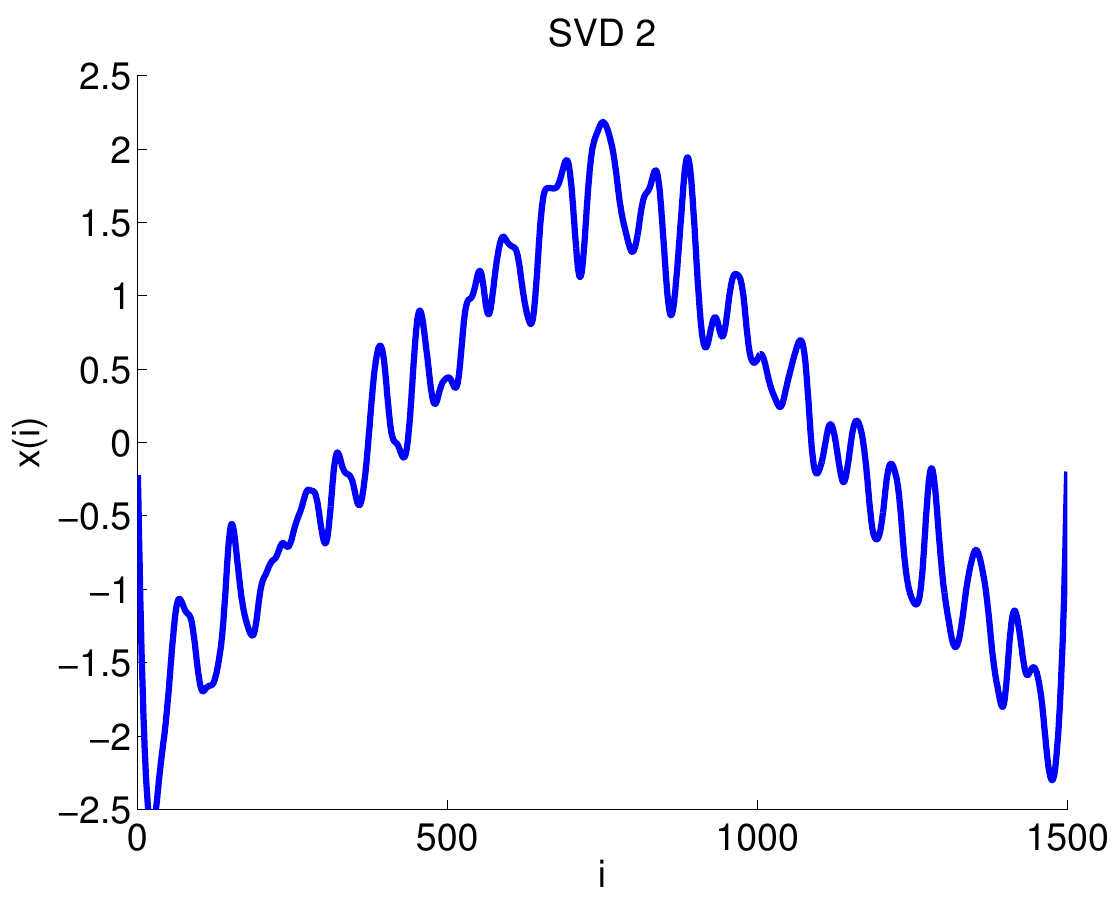}
\includegraphics[scale=0.21]{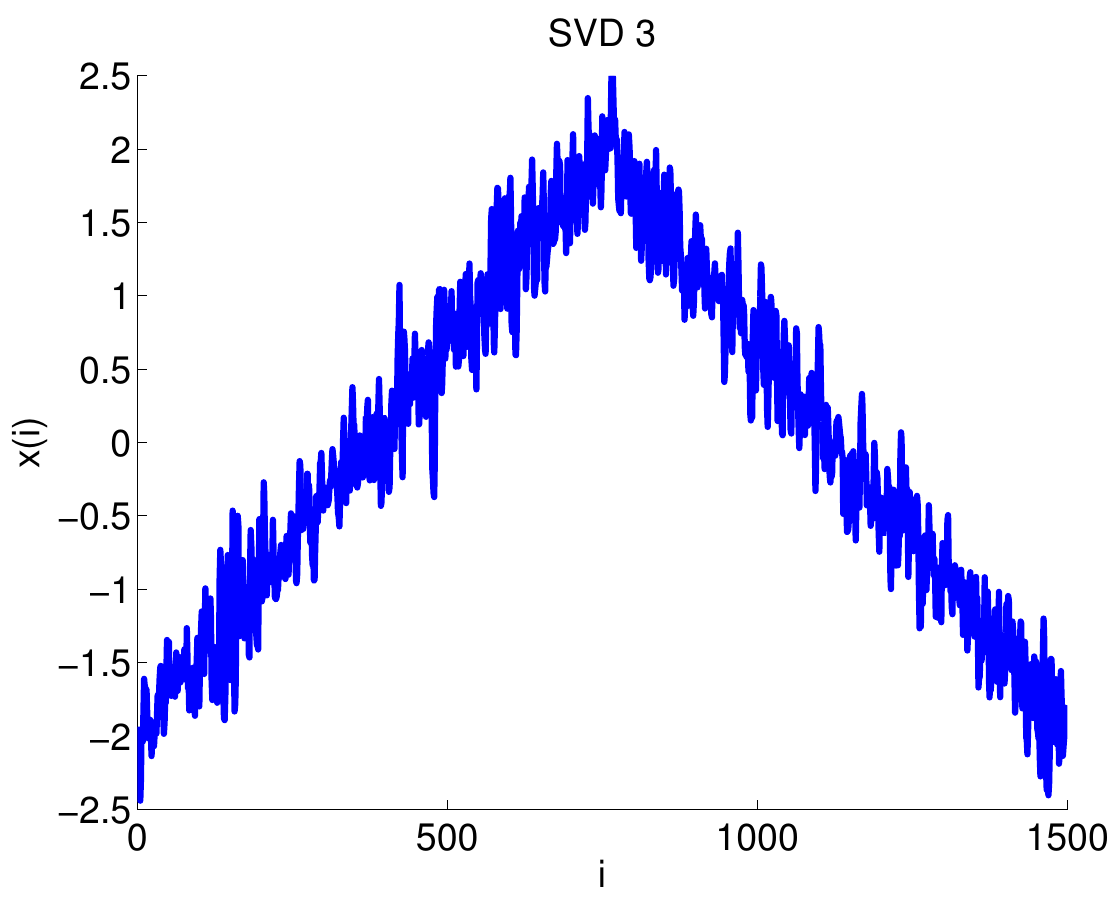}
}
\caption{Row 1: Actual signal $x$ and reconstructions using the full matrices 
$A_{(1)}, A_{(2)}, A_{(3)}$. Row 2: Bar plots of noise norm and solution residual norm values 
for each matrix system. Rows 3-5: Plots of reconstructed solutions using 
the different compressed schemes with wavelet compression and low rank SVD 
for $A_{(1)}$ (row 3), $A_{(2)}$ (row 4), $A_{(3)}$ (row 5). For each SVD solution shown, the low rank 
SVD was obtained via the corresponding $M_{(i)}$ matrix.}
\label{fig:synthetic_data2}
\end{figure*}

\newpage

\subsection{Examples with Real Data}
We now illustrate examples with real data from an application in seismic tomography. We will keep our description of the problem and setup concise. Much details can be found 
in \cite{Simons.Loris.ea2011} and other mentioned references. In short, we have a matrix $A$ and a 
right hand side vector $b$ from which we would like to obtain a vector $x$ corresponding to corrections 
to a spherically symmetric model (which varies only with depth) of the seismic 
wave speeds in the Earth's interior. The idea is that these corrections can be 
used together with the spherically symmetric model in order to 
construct a three dimensional model of the wave speeds. The data comes from measurements made 
by seismometers on the surface of the Earth of different earthquakes in the Earth's interior. 

The rows of our matrix $A$ correspond to earthquake-receiver 
pairs, the number of which is very high (almost $3$ million). It is to our advantage 
to include as many such pairs as possible. The more rows we include, the more information we include in the system and the more detailed the solution 
and hence model, which can be obtained. Each row is constructed from  
a surface wave data set \cite{vanheijst99}, which has information corresponding 
to energy waves from earthquakes only close to the Earth's surface.  
The columns of the matrix correspond to the coordinate system that is used to 
grid the interior of the Earth between the surface and the core 
mantle boundary. Each row of the matrix $A$ is a sensitivity kernel
\cite{GJI:GJI426}, that is defined over a cubed-sphere coordinate system 
\cite{Ronchi199693}, in which the contents at the surface of a sphere of a given
radius are projected onto six faces of a cube. We divide the region 
within the Earth between the core-mantle boundary and the surface into $37$ 
depth layers each divided laterally into $6$ chunks subdivided 
into $128 \times 128$ voxels. Each row of the matrix $A$ (a kernel) has 
information for each of the $37$ depth layers (corresponding to different 
radii from the core-mantle boundary to the Earth's
surface) \cite{Simons.Loris.ea2011}. This translates into approximately 
$3.6$ million columns. 

The matrix $A$ is sparse, having approximately $1.5$ percent nonzeros. The 
resulting matrix is thus very large: the dimensions of the 
matrix $A$ are $2,968,933 \times 3,637,248$ 
and it is approximately $3$ TB in size on the disk in a double precision 
sparse format. The reason for the large size is apparent from a 
typical sparse storage scheme which stores the dimensions, 
the total number of nonzeros, the number of nonzerors 
in each row (or column), and the column (or row) indices of all the nonzeros, 
followed by the floating point values of all the nonzeros. 
We typically use integers to represent everything but the floating point 
values for which we use floats or doubles. The resulting binary file can easily 
be several terabytes in size when the dimensions and number of nonzeros are large. 

Since the matrix $A$ is too large for us to handle directly, 
we split the matrix $A$ into $20$ different blocks: 
\begin{equation*}
A = 
\begin{bmatrix}
A_1 \\
A_2 \\
\vdots \\
A_{20} 
\end{bmatrix}.
\end{equation*}
In our illustrations, we will use also the smaller submatrix $A_1$ of 
the full matrix $A$. The submatrix has dimensions $438,674 \times 3,637,248$ and is 
about $115$ GB in uncompressed form. We can load this matrix into memory. 
In Figure \ref{fig:singular_vals_A1_and_A}, we show the 
fist $2000$ singular values of $A_1$ and $A$ (approximated numerically via the randomized 
low rank SVD algorithm) with the first 
singular value scaled to be $1$. We note that the 
singular values of $A$ drop off significantly faster than those of $A_1$ because $A$ is a 
much larger matrix with significantly more linear dependence. This type of singular value 
behavior is common for matrices from similar applications, so as we illustrate later in this 
section, the low rank approximation techniques 
we describe here work relatively well even when the rank $k$ is marginal compared 
to matrix dimensions. It's important to note again that our schemes rely mostly on operations with the 
$A_1^T A_1$ and $A^T A$ matrices for which the decay of the singular values is very rapid, being the square of 
the illustrated rate for $A_1$ and $A$.

In order to get an idea of the structure and wavelet compressibility 
of our matrices, we take a look at a randomly chosen row of $A$, 
which represents a sensitivity kernel and its representation with different numbers of 
wavelet coefficients as per \eqref{eq:wavelet_approx_for_x}, using the same CDF $9-7$ transform as before. 
In Figure \ref{fig:kernel_compression1}, 
we plot the sensitivity kernel near the surface of the 
Earth (at $135$ km depth). That is, we plot part of a row of matrix, 
representing a certain depth layer near the surface. From the figure, we can clearly see that the kernel looks like a continuous image and is hence similar to a row of matrix $A_{(3)}$ in the previous section, which as we 
saw, was wavelet compressible. In the top of Figure \ref{fig:kernel_compression1}, the leftmost plot is the 
original kernel while the rightmost plot is the reconstructed kernel with 
about $10$ percent of the coefficients retained after transforming. 
We see a notable degradation in quality. However, 
when we keep about $25$ percent of the largest coefficients, we have much less noticeable 
reconstruction error. We clearly observe that while some details are lost as less coefficients are 
retained, the majority of the structure is preserved. We have performed such plots of several randomly chosen 
rows and we conclude that our matrix $A$ is at least as good for wavelet compression 
as synthetic matrix $A_{(2)}$ (where 
at least a subset of the rows compressed well), 
but likely significantly better, with most rows being wavelet compressible. In Figure 
\ref{fig:kernel_compression1}, we also plot a curve of the percent error 
$E = 100 \frac{\| r - \left(W^{-1} \left( \Thr(W r^T) \right)\right)^T \|}{\|r\|}$
versus the percent of coefficients retained by the thresholding function. By percent coefficients 
retained we mean the quantity $100 \frac{\mbox{nnz}\left(\Thr(W r^T)\right)}{\mbox{nnz}\left(W r^T\right)}$, 
where $r$ is either the whole row vector or part of a row (corresponding either to all depth layers or to a 
certain depth near the surface) and nnz is the number of nonzeros. 
Notice that the error over all depths (all the entries of the kernel row) 
is greater than just at the particular depth layer at which it is plotted; 
but it is acceptable as long as we keep about $25$ percent or more coefficients after transforming.  

Since we find that the rows of $A$ are in large part wavelet compressible, we will again use 
wavelet compression and the low rank SVD, in order to approximate 
matrix vector operations with the matrices $A$ and $A_1$ and the solutions:
\begin{equation*}
(A_1^T A_1 + \lambda I) \bar{x}_1 =  A_1^T b \quad \mbox{and} \quad (A^T A + \lambda I) \bar{x}_2 = A^T b \quad \mbox{and} \quad (A^T A + \lambda_1 I + \lambda_2 L^T L) \bar{x}_3 = A^T b .
\end{equation*}
with $L$ a Laplacian smoothing operator, which we build from scratch as a sparse 
matrix. Just as with our synthetic data examples, we first obtain 
the wavelet thresholded matrices $M_1$ and $M$ corresponding to 
$A_1$ and $A$ and use these smaller matrices to obtain the low rank SVD of 
the $A_1$ and $A$ matrices, to achieve further compression. Notice also that as our data comes from a surface 
wave data set, the resolution of our inversions is primarily limited to a region 
close to the Earth's surface, a point we remind the reader of several times in this section. 

\newpage

\begin{figure*}[ht!]
\centerline{
\includegraphics[scale=0.30]{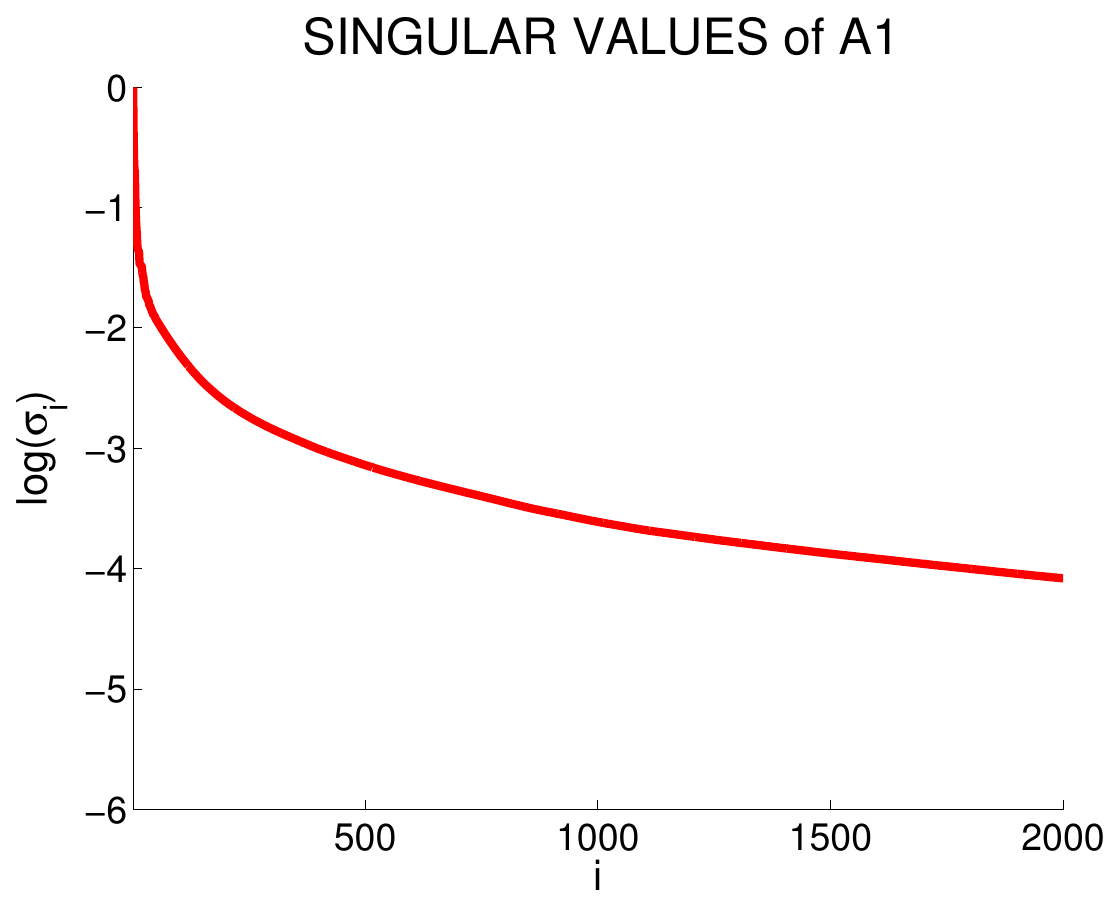}
\includegraphics[scale=0.30]{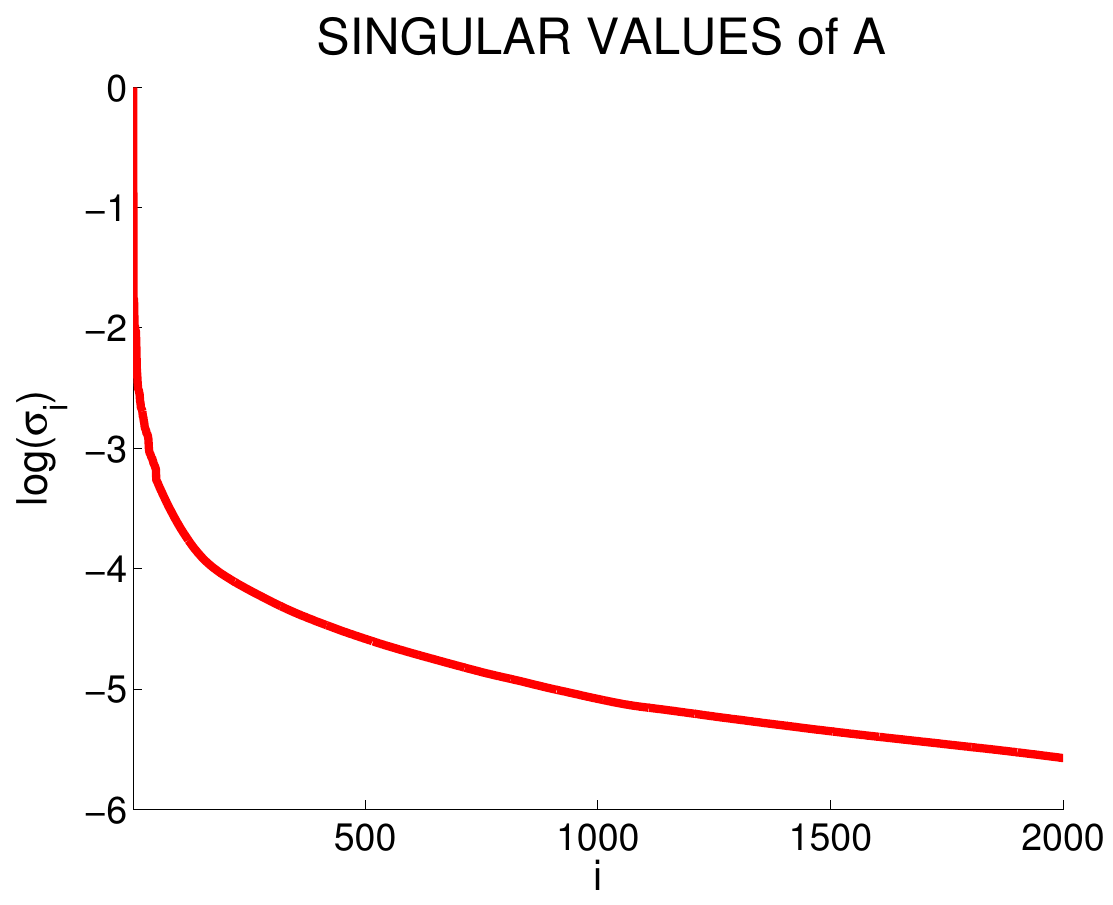}
}
\caption{First $2000$ singular values of $A_1$ and $A$ (numerically approximated)}
\label{fig:singular_vals_A1_and_A}
\end{figure*}

\begin{figure*}[ht!]
\centerline{
\includegraphics[scale=0.4]{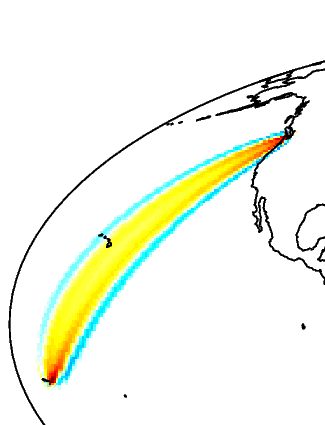}
\quad
\quad
\quad
\quad
\includegraphics[scale=0.4]{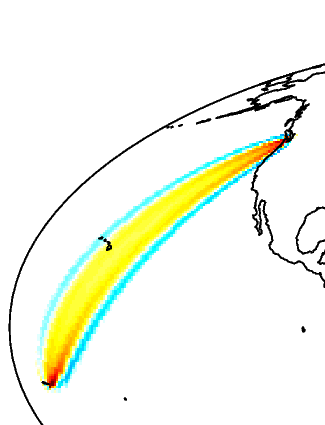}
\includegraphics[scale=0.4]{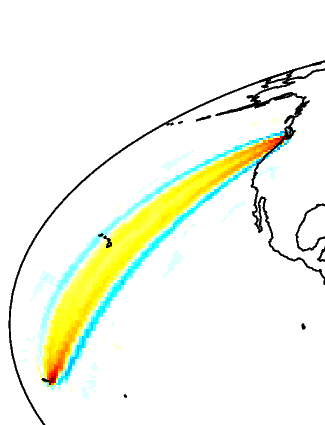}
\includegraphics[scale=0.4]{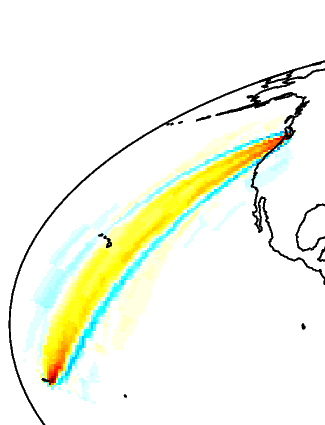}
}
\vspace{3.mm}
\centerline{
\includegraphics[scale=0.35]{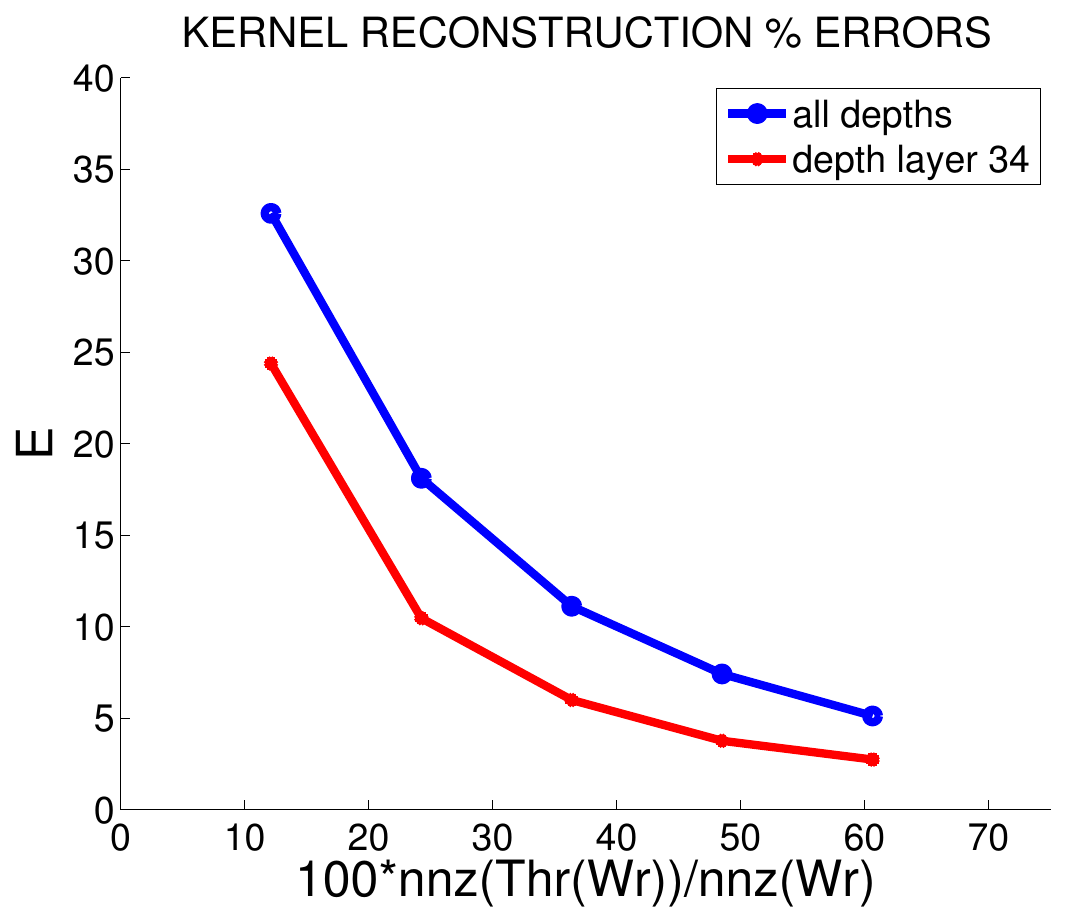}
}
\caption{Original kernel $r$ and reconstructed compressed kernels 
$\left(W^{-1} \left( \Thr(W r^T) \right)\right)^T$ (at $135$ km depth) with different numbers of 
coefficients retained after thresholding: approximately $48$, $24$, and $10$ percent coefficients, 
respectively. The bottom plot shows the percent error curve between the reconstructed and 
original kernel versus the number of nonzeros retained: errors for all depths and only for the 
displayed depth are shown.}
\label{fig:kernel_compression1}
\end{figure*}

\subsubsection{Wavelet and SVD compression with smaller matrix $A_1$}
We now discuss the results of some experiments with matrix $A_1$ which was just small enough 
for us to load in RAM in uncompressed form. 
We form the corresponding wavelet thresholded matrix $M_1 = \Thr(A_1 W^T)$ by replacing each 
row $r$ of $A_1$ by $\left(\Thr(W r^T)\right)^T$. 
We retain one third of the largest coefficients by absolute magnitude. 
The full matrix $A_1$ is of size $115$ GB 
while the matrix $M_1$ computed with our chosen threshold comes out to be $35$ GB. 
In Figure \ref{fig:wavelet_compression_errors_in_matvecops} 
we show the errors that result when we use the compressed matrix $M_1$ to approximate 
matrix vector operations with $A_1$. For $50$ random Gaussian vectors $x$ and $y$ compatible 
with the dimensions of $A_1$ and $A_1^T$, we plot the percent errors  
between $A_1 x$ and $M_1 W^{-T} x$, $A_1^T y$ and $W^{-1} M_1^T y$, 
and between $A_1^T A_1 x$ and $W^{-1} M_1^T M_1 W^{-T} x$. The error quantity for the first 
case is simply $E = 100 \frac{\| A_1 x - M_1 W^{-T} x \|}{\|A_1 x\|}$, as before in the synthetic data 
tests. 

We use the same CDF $9-7$ wavelet transform 
as in the synthetic tests for $W$, but do not build $W$ explicitly as a matrix 
and cannot obtain the inverse-transpose matrix $W^{-T}$ by transposing the inverse of $W$. 
This is because $W$ is a very large $n \times n$ matrix and is very costly to build for large $n$. 
Hence, we instead use a routine for applying $W$ and $W^{-T}$ to vectors. 
Unlike with synthetic data where $W^{-T}$ is exact, the implemented routine for the inverse 
transpose transform is approximate. We programmed the inverse transpose routine by applying 
the forward transform with the inverse filters but it did not exactly equal to the inverse 
of the transpose of $W$ because of complicated boundary data treatment. We see that this 
increases the errors somewhat when approximating matrix-vector operations with $A_1$ and $A_1^T A_1$.  
We see that the error for approximating the operation $A_1^T A_1 x$ is for some vectors  
higher than the approximation for $A_1 x$ and $A_1^T y$. However, from the figure we see that 
all operations are approximated with errors below about $20$ percent 
(which, although significant, will not give rise to large errors in regularized solutions). 

Next, as we previously did with synthetic data, we go on to compute the approximate low rank SVD of 
$A_1 \approx U_{1_k} \Sigma_{1_k} V^T_{1_k}$ using  
the wavelet compressed matrix $M_1$ to approximate matrix-vector operations with $A_1$ in the 
randomized low rank SVD algorithm. The dimensions and sizes of the various matrices turn out as follows:
\begin{itemize}
\item $A_1$, dimensions $(438,674 \times 3,637,248)$, size is $115$ GB 
\item $M_1$, dimensions $(438,674 \times 3,637,248)$, size is $35$ GB
\item $U_{1_k}$, $\Sigma_{1_k}$, $V_{1_k}$, 
dimensions $(438,674 \times 2000), (2000 \times 2000), (3637248 \times 2000)$, sizes are $7$ GB, 
$30$ MB, $55$ GB ($\approx 62$ GB total)
\end{itemize} 
We show the errors that result in approximating matrix-vector 
operations with $A_1$ and $A_1^T$ using the low rank SVD in the same Figure 
\ref{fig:wavelet_compression_errors_in_matvecops} where we plot, for $50$ randomly generated vectors 
$x$ and $y$, percent errors between $A_1 x$ and $U_{1_k} \Sigma_{1_k} V^T_{1_k} x$, 
$A_1^T y$ and $V_{1_k} \Sigma_{1_k} U^T_{1_k} y$, and between 
$A_1^T A_1 x$ and $V_{1_k} \Sigma^2_{1_k} V^T_{1_k} x$. The error quantity for the first 
case is simply $E = 100 \frac{\| A_1 x - U_{1_k} \Sigma_{1_k} V^T_{1_k} x \|}{\| A_1 x \|}$, 
as before in the synthetic data tests. From the figure we see that for approximating the 
$A_1^T A_1 x$ operation with $V_{1_k} \Sigma^2_{1_k} V^T_{1_k} x$, the 
errors are similar to those obtained via the wavelet thresholded $W^{-1} M_1^T M_1 W^{-T} x$ approximation, 
though they do jump to about $50$ percent for a few vectors in the set. 
The errors are significantly lower for the approximated $A_1^T A_1 x$ 
operation then for operations with $A_1$ or $A_1^T$ individually. This is because the decay of singular 
values of $A_1^T A_1$ is much more rapid than that of $A_1$ and the matrix is thus well approximated 
with a low rank $k$. Notice, however, that for the low rank SVD of $A_1$, 
the total size of the SVD components (which are not sparse matrices) is greater than the size 
of the matrix $M_1$. Hence, it may not be very practical to use the low rank SVD decomposition for this 
smaller matrix. However, it is useful to use in this case for illustrative purposes.

We go on to obtain some approximate regularized solutions using the wavelet compressed 
matrix $M_1$ and the low rank SVD components $U_{1_k}, \Sigma_{1_k}, V_{1_k}$ and compare 
to the full solution we get with matrix $A_1$. The solutions we plot in 
Figure \ref{fig:full_and_approx_solutions_for_A1} are obtained by doing $250$ iterations 
of the CG algorithm for the systems listed below.
\begin{equation}
\label{eq:approxSolnA1}
\begin{array}{rcll}
 (A_1^T A_1 + \lambda I) x_1 &=& A_1^T b_1 
   &\quad \mbox{\footnotesize solution with full matrix } A_1 
\\
 (W^{-1} M_1^T M_1 W^{-T} + \lambda I) x_2 &=& W^{-1} M_1^T b_1
   & \quad \mbox{\footnotesize wavelet compressed solution with } M_1 
\\
 (V_{k_1} \Sigma_{k_1}^2 V_{k_1}^T + \lambda I) x_3 &=& V_{k_1} \Sigma_{k_1} U_{k_1}^T b_1 
   & \quad \mbox{\footnotesize replacing all instances of } A_1 \mbox{ by low rank SVD}
\\
 (V_{k_1} \Sigma_{k_1}^2 V_{k_1}^T + 5 \lambda I) x_4 &=& W^{-1} M_1^T b_1 
   & \quad \mbox{\footnotesize using the low rank SVD only on the left hand side}
\end{array}
\end{equation}
In the figure, we plot the solution at a certain depth near the surface because the data set we used  
in the construction of $A$ (and hence $A_1$) is a surface wave data set, so there is minimal resolution 
far down from the surface. We mention more on this later in this section. At the depth we show, 
the differences between the solutions are very small. The SVD solutions do show some minor degradations. We have observed the same behavior slightly above and below the 
current depth: that is, for all regions where we have significant resolution with our data set.    
Notice that the wavelet compressed solution $x_2$ is very close to the full solution. With $x_3$ and 
$x_4$ small differences can be observed. The latter solution $x_4$ actually reveals somewhat more 
details than $x_3$. Note also that in Figure \ref{fig:full_and_approx_solutions_for_A1} we plot the 
depth profiles for each solution, where we show a depth slice for a section of the Earth, from the surface 
to the core mantle boundary. As expected, nonzero data is only present at depth layers 
near the surface and the quality of the approximations decrease at the bottom layers. The loss 
of detail with the low rank SVD solutions at the lower layers is visible in these plots. 

Also in Figure \ref{fig:full_and_approx_solutions_for_A1} we show the plots of solution norm
and $\chi^2$ value versus iteration for the different solutions. The norm of the solution is the $\ell_2$ norm of the iterate
$x^n$ at iteration $n$. The $\chi^2$ value is calculated using the formula:
\begin{equation*}
\chi^2 = \frac{1}{P} \displaystyle\sum_{k \textup{ not outlier}} |r^n_k|^2,
\end{equation*}
where $r^n = A_1 x^n - b$ and $P = m - m_0$ (number of rows minus number of outliers).
For each datum, we estimate standard errors in the data before inversion, then scale
the system to be univariant (i.e. all standard errors are equal to 1). We define outliers
as entries of the vector $r^n$ that are not within three standard errors.
In the inversions we present, the outliers are identified
after $5$ and $25$ iterations, corresponding to dips in the $\chi^2$ that may be seen in the plots.
Since our systems are univariant, we would like for the $\chi^2$ of the converged solution
to be close to one. However, this is not possible for this data set without including
extra correction terms for spatial uncertainty in the earthquake coordinates and instrument error 
in the data. Hence the $\chi^2$ values are quite a bit higher.
In the figure, we can see that the curves for the full and wavelet 
thresholded case are very close to each other; 
the first SVD solution has a lower norm and the second a slightly higher solution norm 
at the chosen value of $\lambda$. 

For the $\chi^2$ calculation, we calculate the product $A_1 x^n$ using the full matrix $A_1$ in 
the solution $x_1$, using the approximation $M_1 W^{-T} x^n$ in the solution $x_2$, and 
using the approximation $U_{1_k} \Sigma_{1_k} V^T_{1_k} x^n$ in the solutions $x_3$ and $x_4$.
Notice that since the operation $A_1 x^n$ is not as well approximated as the operation $A_1^T A_1 x^n$, 
we have a noticeable difference in $\chi^2$ values between solutions $x_1$ and $x_2$ and between 
$x_3$ and $x_4$. For the latter two solutions, the calculated $\chi^2$ value comes out 
higher than it really is. To illustrate this fact, we include in Figure 
\ref{fig:full_and_approx_solutions_for_A1} a bar plot which shows the 
the mean $\chi^2$ value after $50$ iterations from the two SVD solutions $x_3$ and $x_4$ computed using the
low rank matrix $U_{1_k} \Sigma_{1_k} V_{1_k}^T$ and using the full matrix $A_1$. 
The same solutions have correspondingly lower $\chi^2$ values when the residual
$r^n = A_1 x^n - b_1$ is approximated via $A_1 x^n - b_1$ instead of
$U_{1_k} \Sigma_{1_k} V_{1_k}^T x^n - b_1$. Thus, while the solutions themselves are approximated 
well with the SVD approximations, quantities such as $\chi^2$ which involve calculations with $A_1$ 
instead of $A_1^T A_1$ can be far less accurate when computed with the low rank SVD matrices. 
Given the results with the matrix $A_1$, we summarize a few key points which we observe. 
\begin{itemize}
\item In the case of matrix $A_1$ which is not so large, wavelet thresholding makes the most sense, as the 
low rank SVD does not provide compression, unless the $k \times k$ methods are used. This is because 
the low rank SVD matrices are dense while the original matrix is sparse.
\item Approximate solutions with both wavelet thresholding and the low rank SVD are quite accurate compared
to those with the full matrix.
\item In matrix vector operations, the error in the approximation to operations with $A_1^T A_1$ 
is significantly less than for the approximations to operations with $A_1$ and $A_1^T$. Hence, 
quantities such as $\chi^2$ are not accurately computed if the low rank SVD matrix is used to compute 
the residual; instead one should use the wavelet compressed or full matrix (for one computation) 
to accurately estimate the $\chi^2$ value of the solution vector.
\item When computed with $A_1$ or $M_1$, the $\chi^2$ values for the approximate solutions are 
very similar to that of the full solution.
\item The difference between the full and approximate solutions becomes significant at lower 
depths, where the data set resolution is poor.
\end{itemize}

\begin{figure*}[ht!]
\centerline{
\includegraphics[scale=0.28]{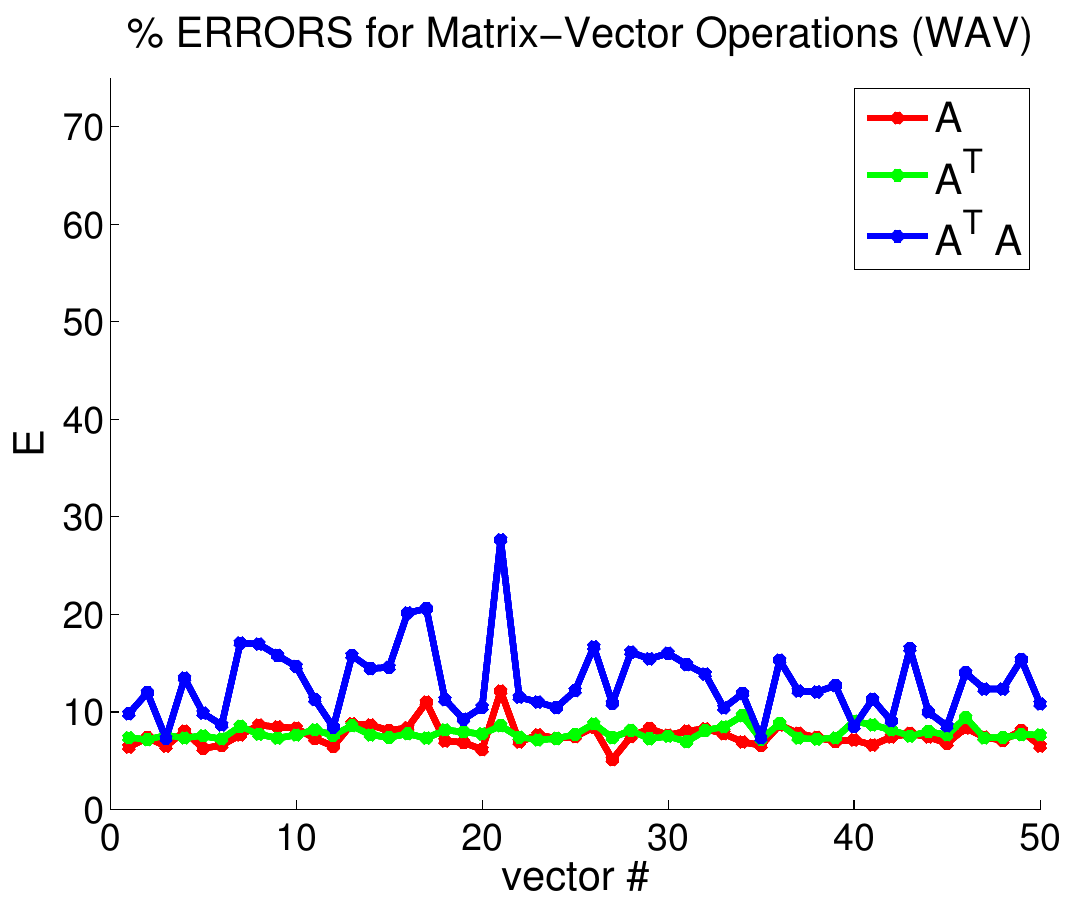}
\includegraphics[scale=0.28]{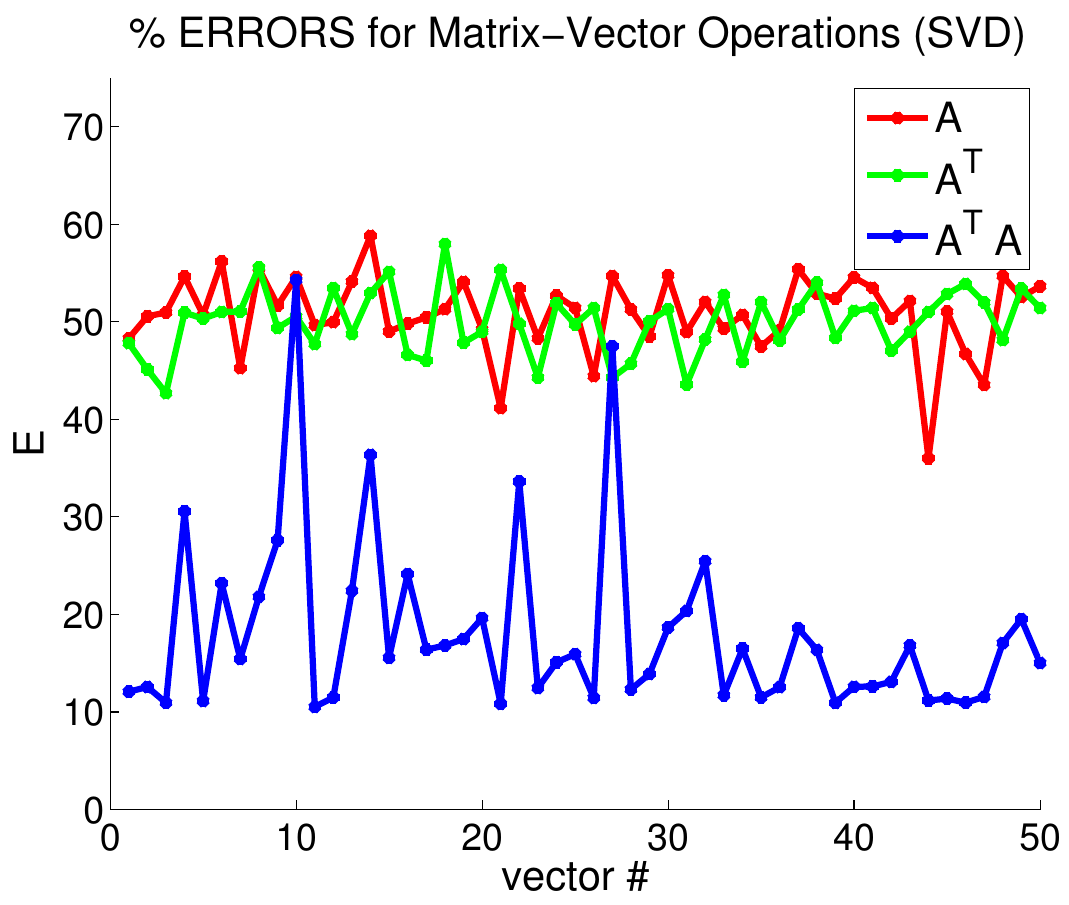}
}
\caption{Percent errors for $50$ Gaussian random vectors $x$ and $y$ between the vectors 
$A_1 x$, $A_1^T y$, $A_1^T A_1 x$ and their approximations through wavelet compressed and low rank SVD methods.}
\label{fig:wavelet_compression_errors_in_matvecops}
\end{figure*}

\newpage

\begin{figure*}[ht!]
\centerline{
\includegraphics[scale=0.40]{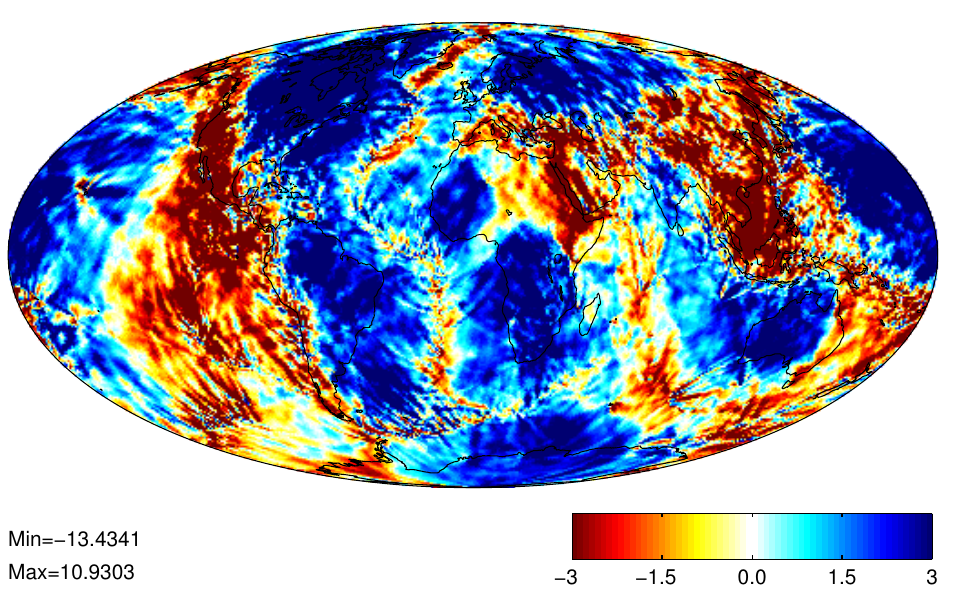}
\quad
\includegraphics[scale=0.40]{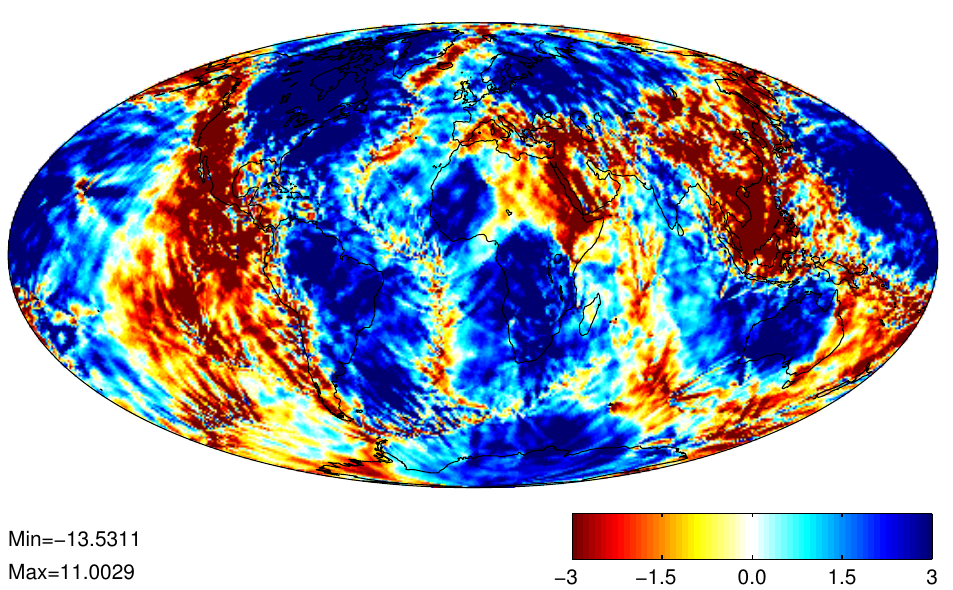}
}
\centerline{
\includegraphics[scale=0.40]{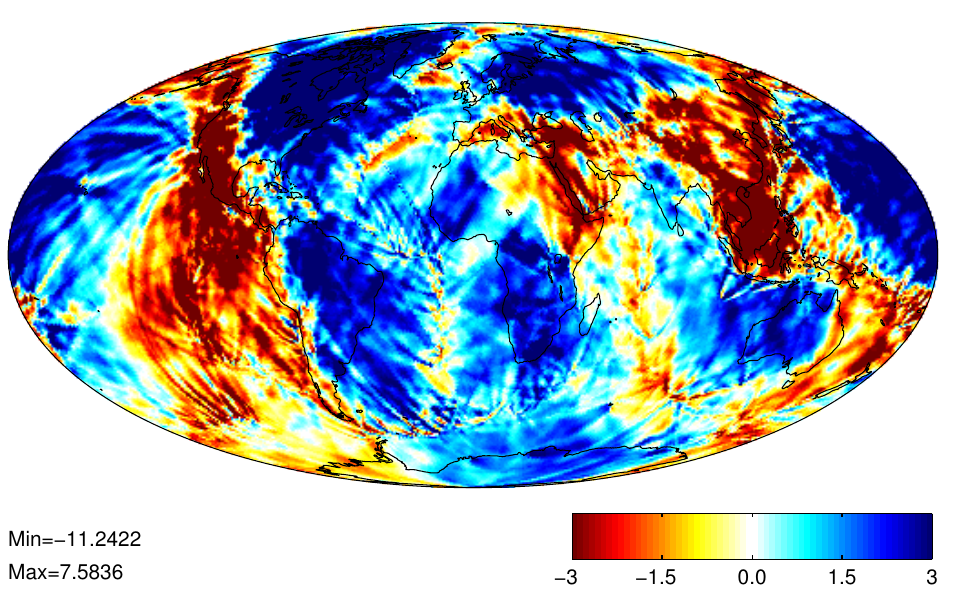}
\includegraphics[scale=0.40]{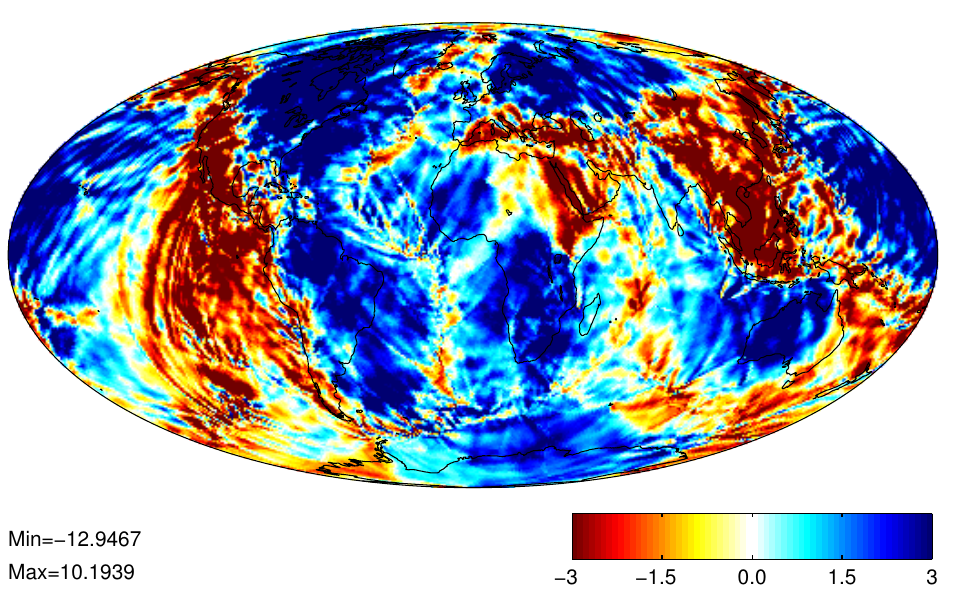}
}
\vspace{3.mm}
\centerline{
\includegraphics[scale=0.25]{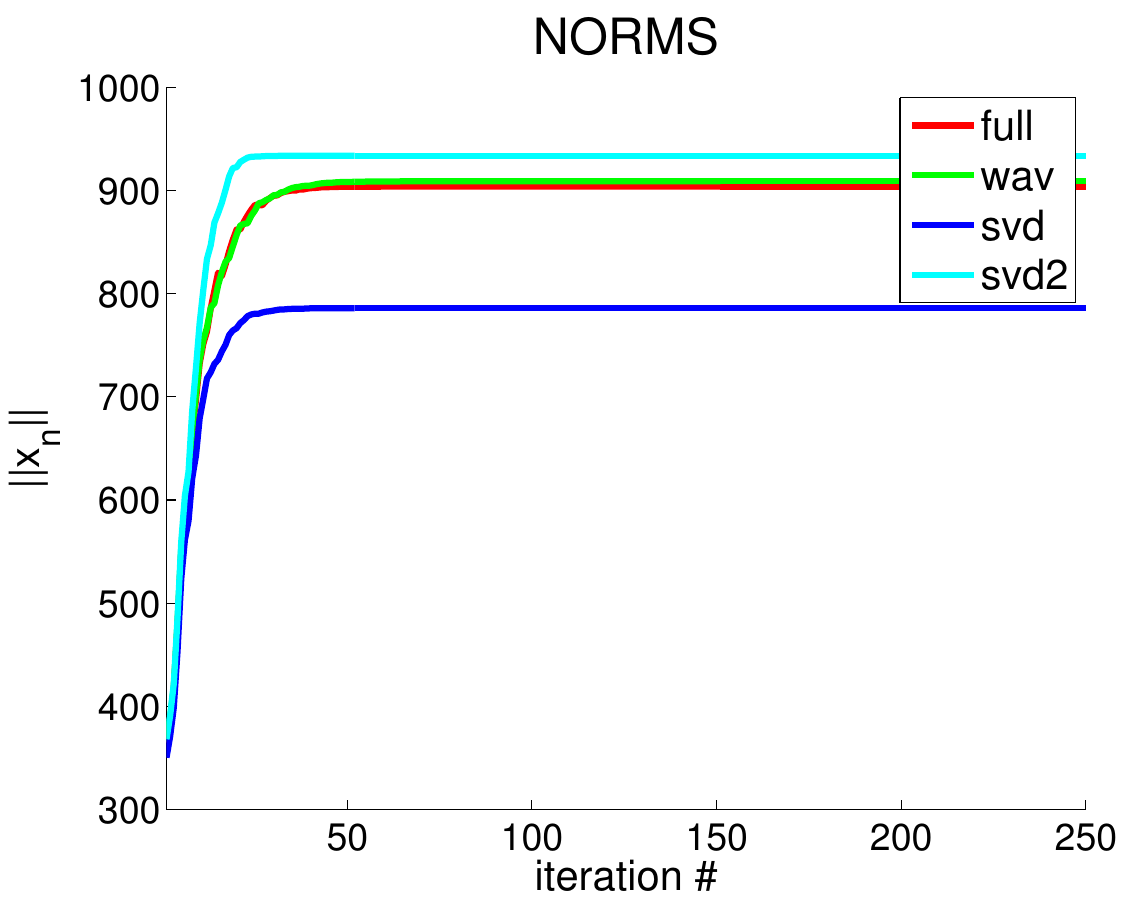}
\quad
\includegraphics[scale=0.25]{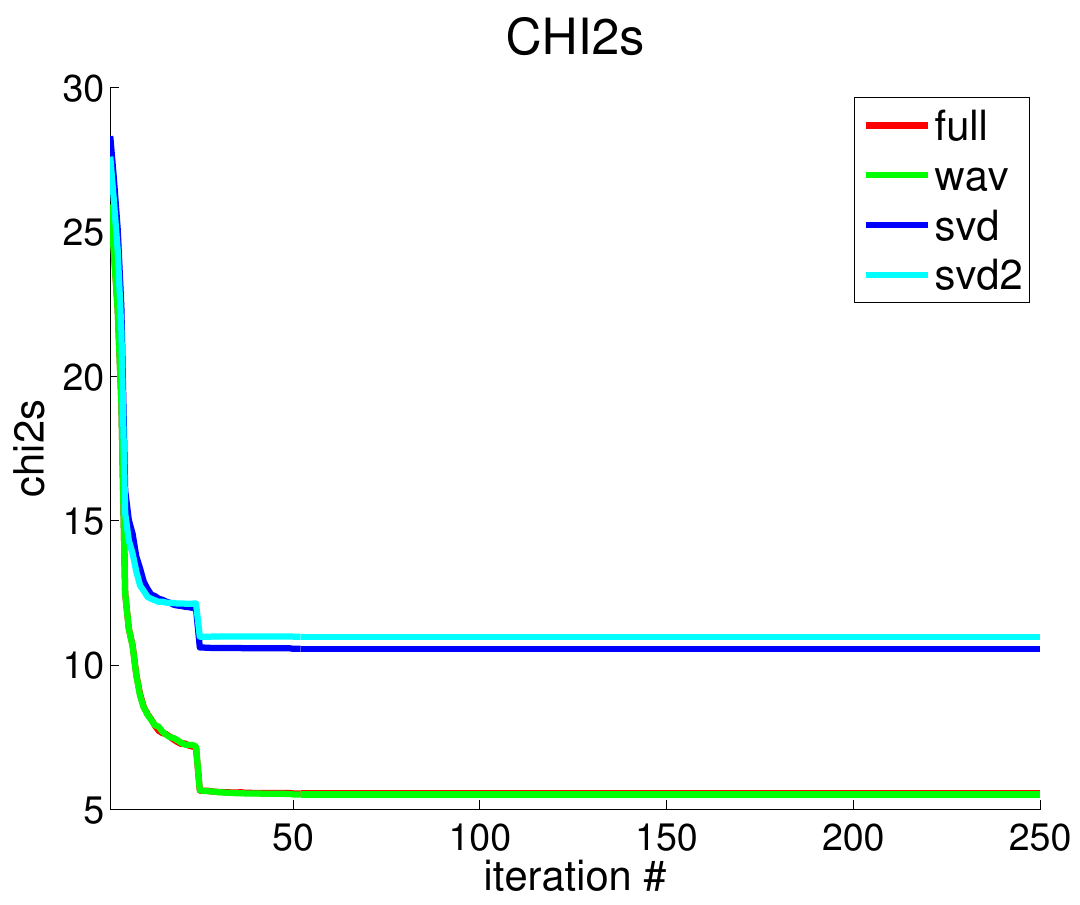}
\quad
\includegraphics[scale=0.25]{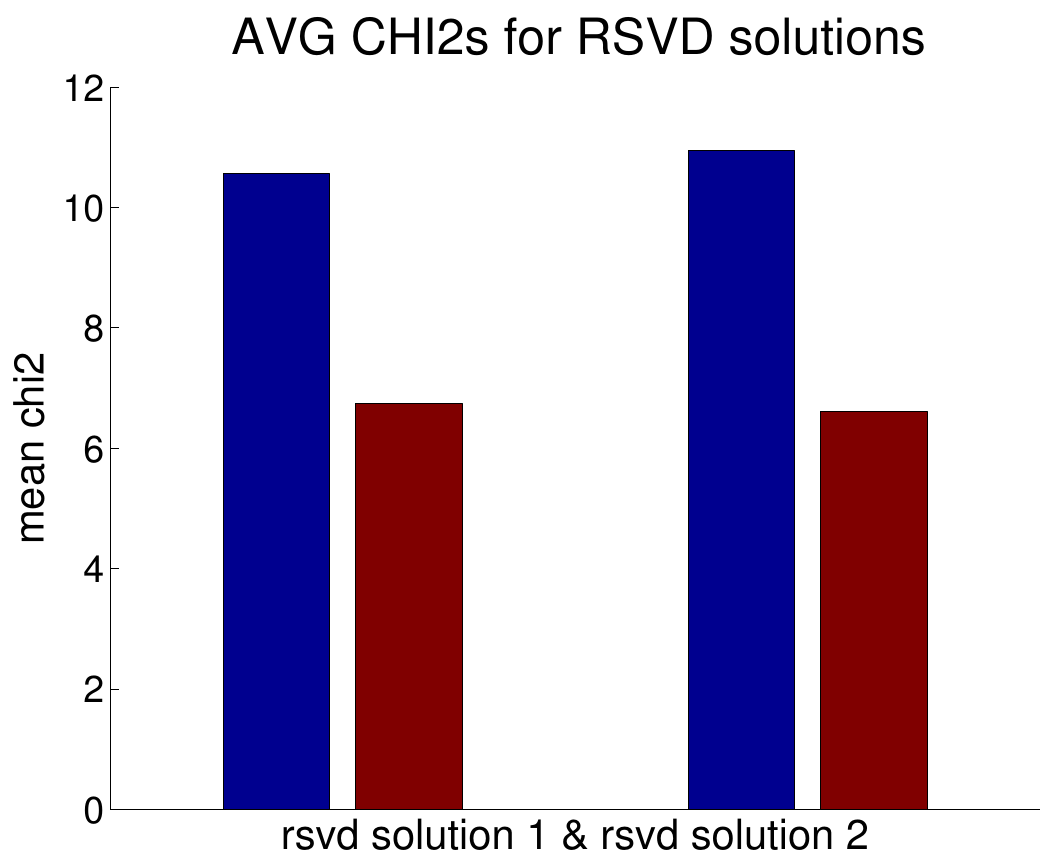}
}
\vspace{3.mm}
\centerline{
\includegraphics[scale=0.5]{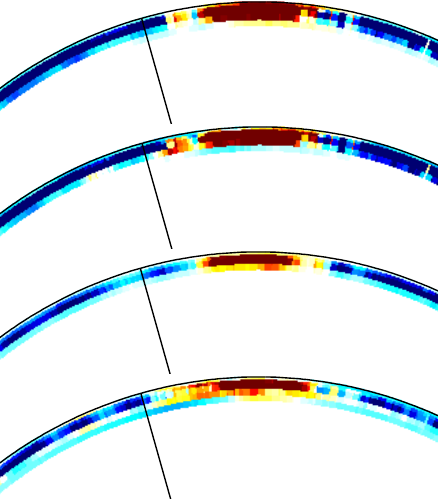}
}
\caption{Rows $1-2$: Regularized solutions $x_1$ (full matrix - row 1, left), $x_2$ (wavelet), $x_3$ (svd 1 - row 2, left), and $x_4$ (svd 2) 
from \eqref{eq:approxSolnA1} plotted at $135$ km depth. 
Row $3$: solution norms and $\chi^2$ values versus iteration, bar plot comparing average $\chi^2$ 
of the two SVD solutions computed using the low rank SVD matrix and the wavelet compressed matrix. Row $4$: depth profiles of the four solutions in a portion of the globe with variations (the top arcs represent the Earth's surface).} 
\label{fig:full_and_approx_solutions_for_A1}
\end{figure*}

\newpage

\subsubsection{Wavelet and SVD compression with matrix $A$}

We now describe some results of wavelet and low rank SVD compression for our very large matrix $A$. Due to the size of $A$, even after wavelet compression, 
the resulting $M$ is too big to load into RAM all at once on a single machine.  
For this reason, we do not compute the wavelet thresholded $M$ in one shot.
Instead we operate on blocks of $A$ at a time and construct the block based:
\begin{equation*}
M = 
\begin{bmatrix}
M_1 \\
M_2 \\
\vdots \\
M_{20} 
\end{bmatrix} = 
\begin{bmatrix}
\Thr(A_1 W^{T}) \\
\Thr(A_2 W^{T}) \\
\vdots \\
\Thr(A_{20} W^{T}) \\
\end{bmatrix}
\end{equation*}
This way, operations with $A$ can be approximated using relations \eqref{eq:block_wavelet_operations} 
and the components of $M$ can be stored in parallel over several different machines. 

We now state the sizes and dimensions of the matrices involved:
\begin{itemize}
\item $A$, dimensions $(2,968,933 \times 3,637,248)$, size is $3.2$ TB (approximate, never computed)
\item $M$, dimensions $(2,968,933 \times 3,637,248)$, size is $1$ TB
\item $U_k$, $\Sigma_k$, $V_k$, dimensions $(2,968,933 \times 2000), (2000 \times 2000), (3,637,248 \times 2000)$, sizes are $45$ GB, $30$ MB, $55$ GB ($\approx 100$ GB total) 
\end{itemize} 
Notice that in this case, for the much larger matrix $A$, the low rank SVD 
provides for very substantial memory savings. 

Since we cannot use $A$ directly, we can only compare results with the wavelet 
compressed matrix $M$ to results obtained with the low rank SVD 
$A_k = U_k \Sigma_k V^T_k$. As before, we have first formed $M$ and then used $M$ in the 
randomized SVD scheme to form the approximate low rank SVD of $A$. We again used $k=2000$ 
(a very small number relative to the dimensions of $A$).  
In Figure \ref{fig:wavelet_vs_svd_compression_mat_vec_errors}, we plot the percent errors 
for matrix vector operations done with the computed low rank SVD compared to those approximated 
via the wavelet thresholded matrix $M$. We plot the percent errors for 
$50$ random Gaussian vectors $x$ and $y$ compatible with the dimensions 
of $A$ and $A^T$: that is, between $M W^{-T} x$ (approximating $A x$) and 
$U_k \Sigma_k V^T_k x$, $W^{-1} M^T y$ (approximating $A^T y$) and 
$V_k \Sigma_k U_k^T y$ and between $W^{-1} M^T M W^{-T} x$ 
(approximating $A^T A x$) and $V_k \Sigma_k^2 V_k^T x$. 
The error quantity for the first 
case is simply 
$E = 100 \frac{\| M W^{-T} x - U_k \Sigma_k V^T_k x \|}{\| M W^{-T} x \|}$.
The plots again indicate that the operation $A^T A x$ is likely to be 
well approximated even with a low rank $k$ we choose. In this case, for the large $A$, 
the singular values of $A^T A$ decay very rapidly, with the square of the decay rate
observed in Figure \ref{fig:singular_vals_A1_and_A}.

As previously mentioned, the matrices we use come from a surface wave 
data set \cite{vanheijst99}, such that only the top few depth layers near the surface carry nonzero 
information and even the bottom of these layers can already offer 
limited resolution. Thus the quality of approximations 
can vary somewhat for different depth layers. In order for the reader to 
have an idea of the data set we use, we present some checkerboard 
reconstructions using the matrix $A$ and a synthetically constructed 
checkerboard model $x_{\textit{chk}}$. We define $x_{\textit{chk}}$ 
to be a checkerboard grid, over the top few layers (near the surface). The result is 
plotted in Figure \ref{fig:checkerboard_output} using the 
depth profile (a cross-section plot showing the model representation over all depth layers) 
and corresponding cubed-sphere representations at certain depths 
(we plot at each depth layer shown the projection onto the six cube faces). 
Then we form $b = A x_{\textit{chk}}$ and solve the regularized system 
$(A^T A + \lambda I) x_{\textit{chkrec}}  = A^T b$ 
with $\lambda = 5$. We plot the solution $x$ in Figure \ref{fig:checkerboard_output} using the 
same formats. We use the wavelet transformed 
and thresholded matrix $M = \Thr(A W^T)$ to approximate the matrix vector operations with $A$. 
The comparison between $x_{\textit{chk}}$ and the corresponding reconstruction $x_{\textit{chkrec}}$ 
gives us a summary of what the data set can pick up. In particular, we see 
from Figure \ref{fig:checkerboard_output} that the resolution is limited to layers near 
the surface and gets worse with increasing depth, as expected. 
Also and perhaps more important is that the checkers used are about the size of what we 
we can successfully resolve. We have tried using smaller checkers which did not lead to good reconstructions, 
even for depth layers near the surface. 

The checkerboard test shows the clear limitation of the matrix $A$: we are unable to resolve features 
at all depths, nor are we able to resolve particularly small 
features. Hence, we expect that we can safely use relatively high compression ratio approximation 
methods we have discussed (using aggressive thresholding with wavelet based approximation and small 
$k$ relative to matrix dimension in the SVD based schemes). 
Even though the solutions which result from these methods may not resolve some fine scale 
features in comparison with using the full (or even wavelet thresholded) matrix, it is important to keep 
in mind that these fine scale features which appears in the more detailed solutions may not be 
realistically explainable by the data we have available. This is true in many applications 
similar to ours. 

We will consider the following linear systems for approximating the regularized 
solution to $Ax = b$:
\begin{eqnarray*}
&& (W^{-1} M^T M W^{-T} + \lambda I) x_5 = W^{-1} M^T b \quad \mbox{ wavelet compressed solution for } A \\
&& (V_k \Sigma_k^2 V_k^T + \lambda I) x_6 = V_k \Sigma_k U_k^T b \quad \mbox{ replacing all instances of } A \mbox{ by low rank SVD}\\
&& (V_k \Sigma_k^2 V_k^T + 10 \lambda I) x_7 = W^{-1} M^T b \quad \mbox{ using the low rank SVD only on the left hand side}\\
\end{eqnarray*}
with $\lambda = 1$. We also show the following solutions 
corresponding to the system with Laplacian smoothing included:
\begin{eqnarray*}
&& (W^{-1} M^T M W^{-T} + \lambda_1 I + \lambda_2 L^T L) x_8 = W^{-1} M^T b \quad \mbox{ wavelet compressed with smoothing } \\
&& (V_k \Sigma_k^2 V_k^T + \lambda_1 I + \lambda_2 L^T L) x_9 = V_k \Sigma_k U_k^T b \quad \mbox{ SVD 1 with smoothing }\\
&& (V_k \Sigma_k^2 V_k^T + 10 \lambda_1 I + \lambda_2 L^T L) x_{10} = W^{-1} M^T b \quad \mbox{ SVD 2 with smoothing }
\end{eqnarray*}
The results for a depth layer close to the surface are given in Figure \ref{fig:svd_compression_x5_x6_x7}. 
Again, we find that the results for depth layers around the given depth are quite similar to what we present. 
We can readily notice the effect of the smoothing operator $L$ on the solutions. Notice that the  
wavelet compressed solution without smoothing offers a great deal of detail. However, based on our checkerboard 
experiments, it's unlikely that the smaller scale features we find in this detailed solution 
are real, since they are generally smaller than the checkers we used in our resolution test.
In the figure, we also plot the same plots as for the smaller matrix $A_1$, including plots of the solution 
norms, $\chi^2$ values, and of the depth profiles of the solutions $x_8, x_9, x_{10}$ (with Laplacian 
smoothing). We find similar behavior in the $3$ solutions without the Laplacian. 
As before, we plot a bar chart showing the $\chi^2$ of the SVD based solutions using the SVD and wavelet 
compressed matrices. We see similar behavior in the sense that if $\chi^2$ is computed with matrix $M$, 
it is close to that of the wavelet compressed solution. The depth profile plots in 
Figure \ref{fig:svd_compression_x5_x6_x7} show significant differences between the wavelet 
compressed and low rank SVD solutions at lower depths, although the resolution there is likely very low. 

Given the results with the big matrix $A$, we summarize a few key points which we observe. 
\begin{itemize}
\item Both the wavelet thresholded and the low rank SVD approach allow us to use much smaller matrices 
and still resolve the main solution features (in the case of $A$, the low rank SVD 
components are collectively less than $30$ times the size of the full matrix and offer superior compression gains).
\item For a matrix of this size, block matrix techniques we have discussed are likely 
necessary for practical implementation, so that different parts of the matrices used can 
be stored on different machines. Blocking can be applied both to wavelet compression 
via \eqref{eq:block_wavelet_operations} or to the low rank SVD schemes via 
e.g. \eqref{eq:uta_utb_block_projected_system}.
\item In matrix vector operations, the error in the approximation to operations with $A^T A$ is significantly 
less than for the approximations to operations with $A$ and $A^T$. This again has implications for the
 $\chi^2$ calculation as previously discussed. 
\item The solutions with the low rank SVD do show loss of detail when compared to the wavelet 
thresholding solution. There is significantly less loss of detail when Laplacian smoothing is used, since 
the smaller scale features are smoothed out in that case.  
\item A checkerboard test is a good way to measure matrix resolution. 
The smallest clearly resolved checker size corresponds roughly to the scale of 
properly resolved features in the solution. If the resolution 
is poor, Laplacian smoothing should be used to avoid presenting false fine scale details. 
In this case, the low rank SVD solutions can offer a good approximation with the use of much 
smaller matrices.
\end{itemize}

\begin{figure*}[h!]
\centerline{
\includegraphics[scale=0.32]{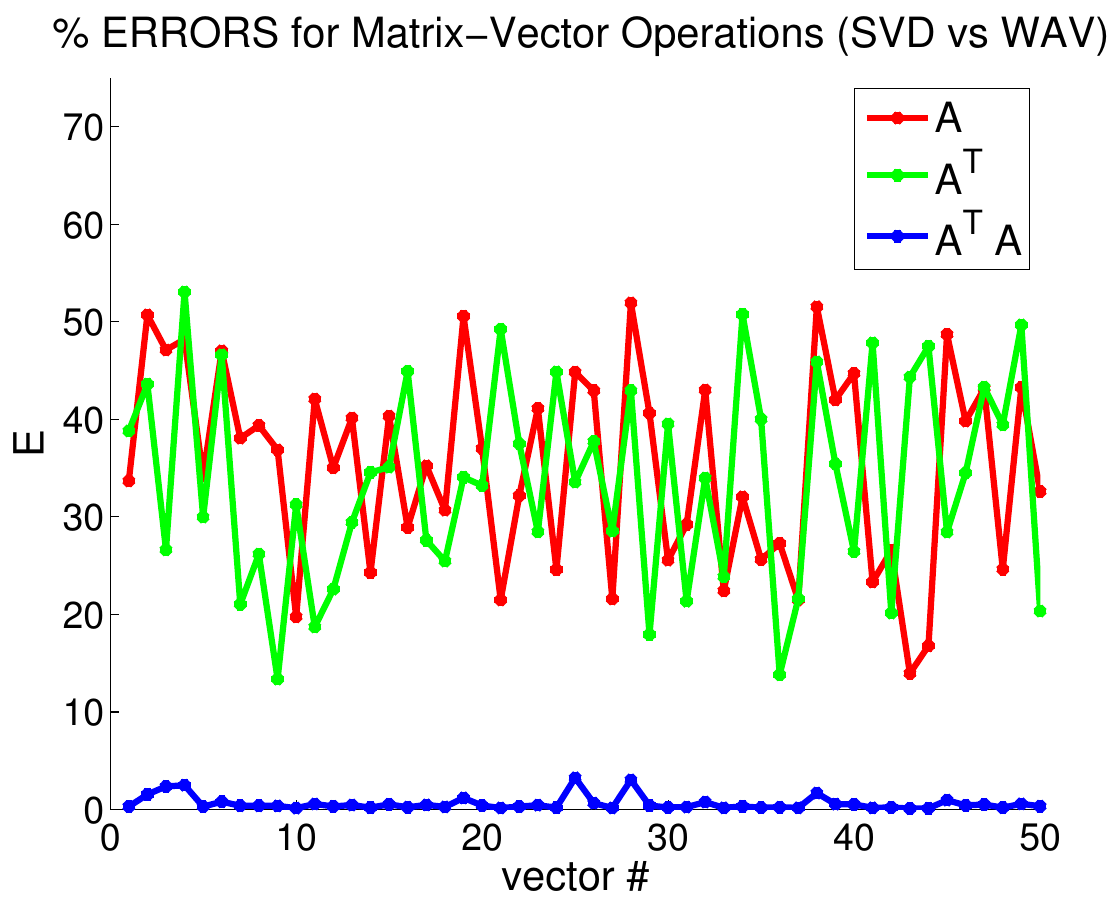}
}
\caption{Percent errors in matrix-vector operations for $50$ Gaussian random vectors 
with $A$, $A^T$, and $A^T A$ approximated via the wavelet compressed 
matrix $M$ and compared to results obtained with the low rank SVD (obtained via $M$).}
\label{fig:wavelet_vs_svd_compression_mat_vec_errors}
\end{figure*}

\newpage

\begin{figure*}[h!]
\centerline{
\includegraphics[scale=0.40]{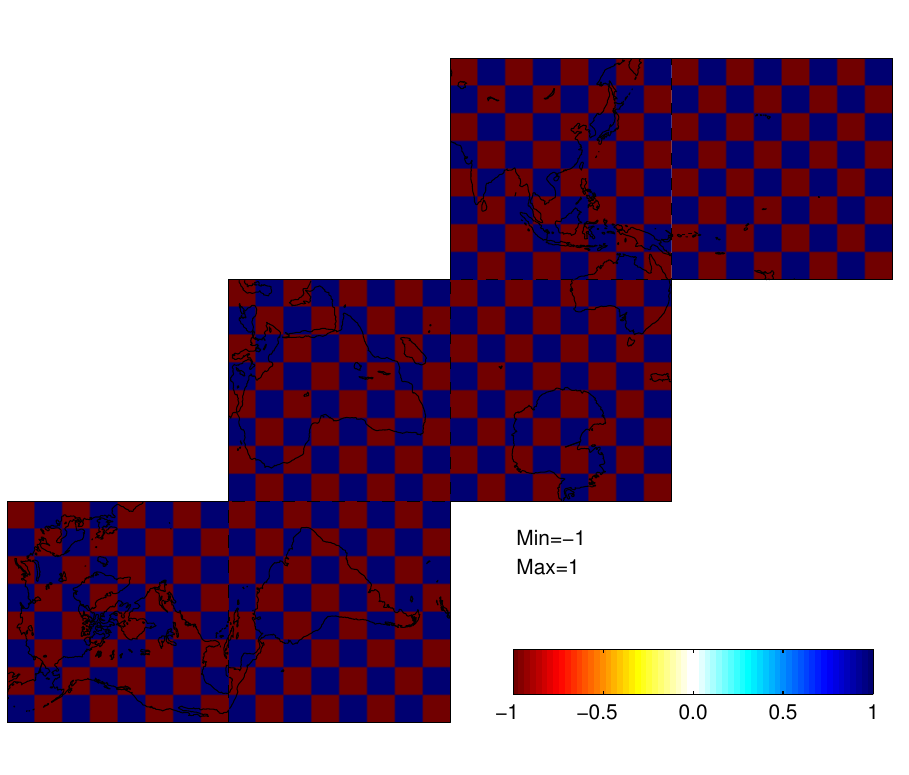}
\includegraphics[scale=0.40]{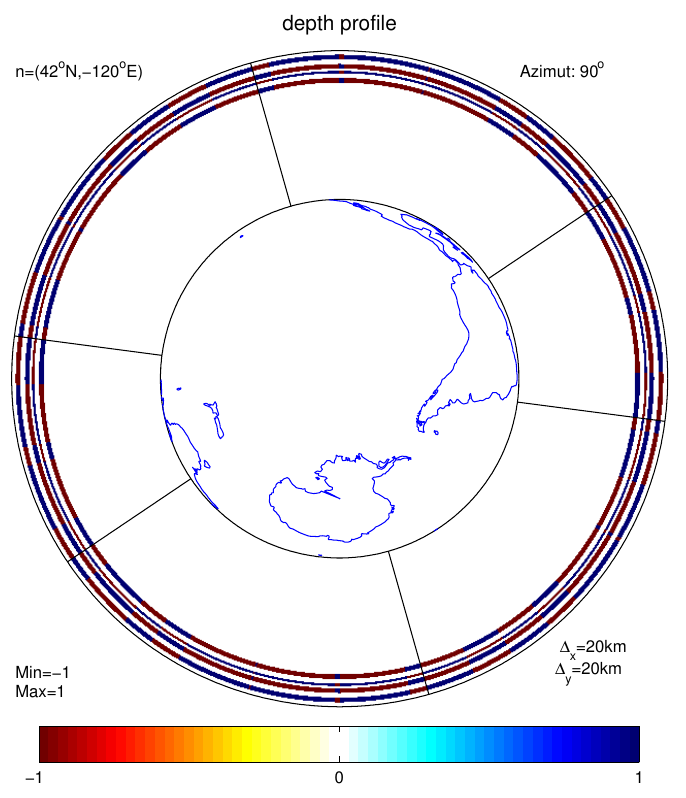}
\includegraphics[scale=0.40]{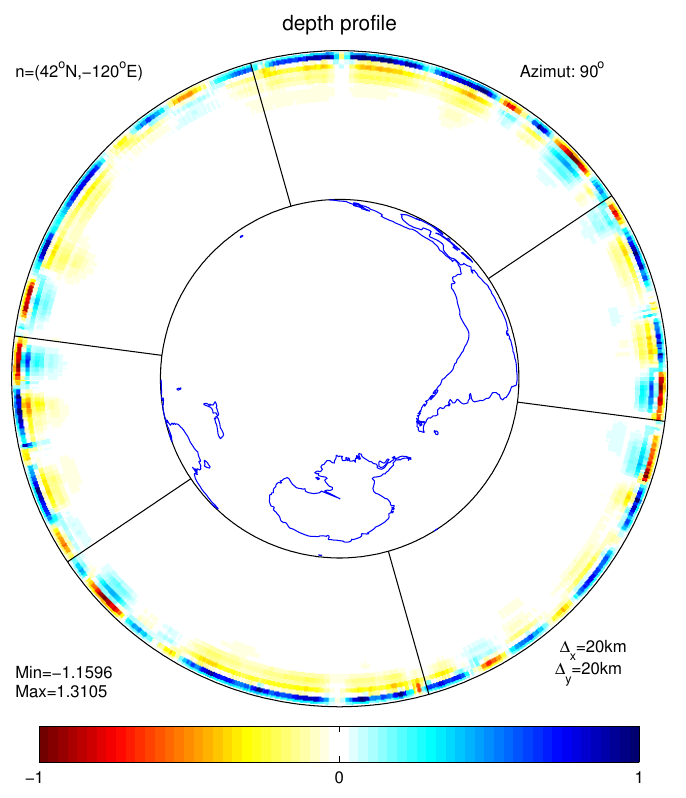}
}
\centerline{
\includegraphics[scale=0.40]{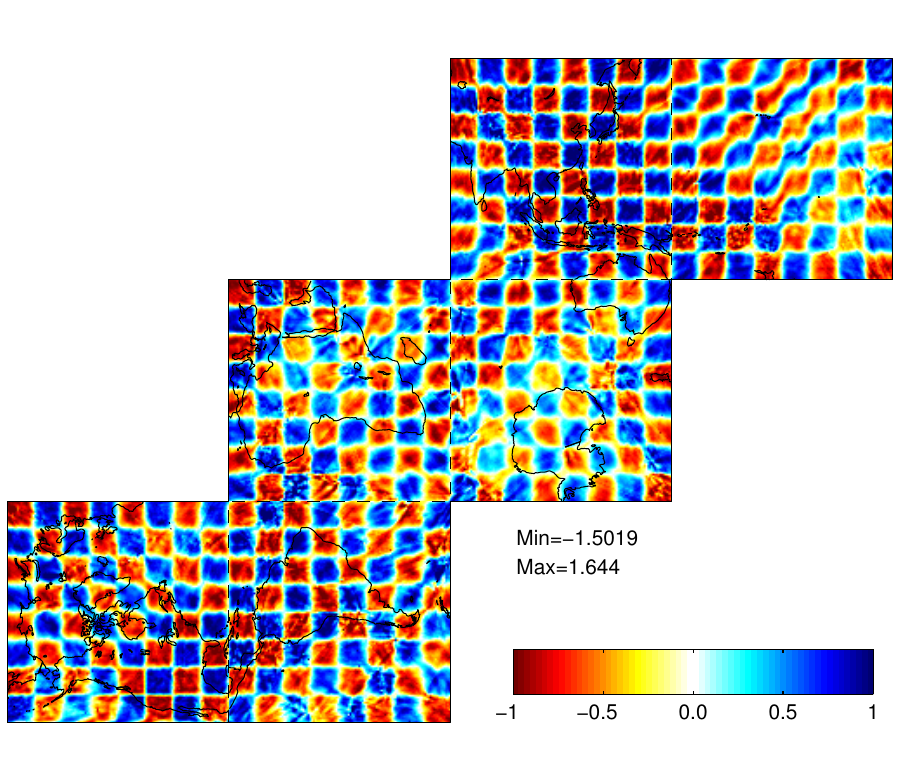}
\includegraphics[scale=0.40]{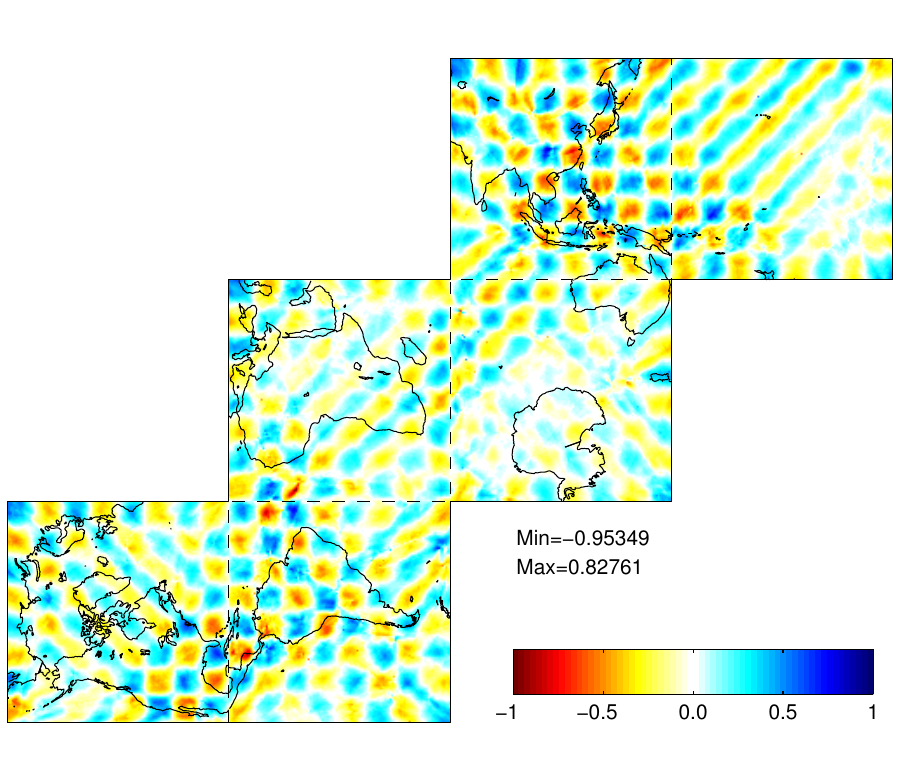}
}
\centerline{
\includegraphics[scale=0.40]{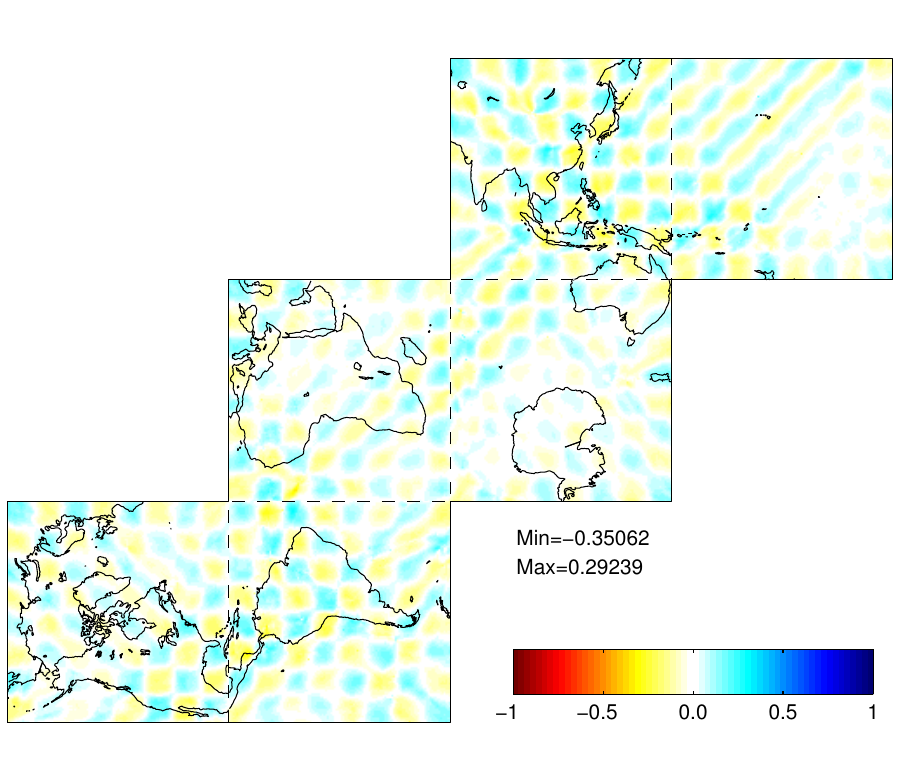}
\includegraphics[scale=0.40]{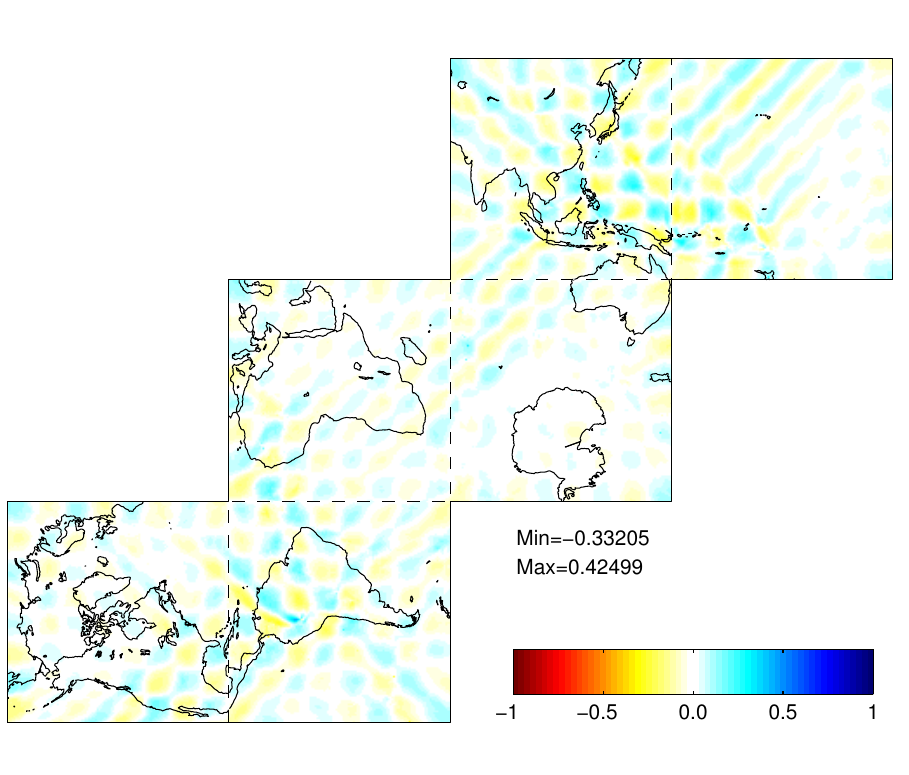}
}
\caption{Checkerboard model and reconstructions at different depths. Row 1: Synthetic $x_{\textit{chk}}$ model and it's depth profile from the surface to the core mantle boundary followed by the depth 
profile of the reconstructed solution $x_{\textit{chkrec}}$. At adjacent  
layers, checkerboards differ only by a sign change. 
Row 2: reconstructed layers $34$ and $32$ ($135$ and $316$ km depth). 
Row 3: reconstructed layers $30$ and $28$ ($428$ and $586$ km depth).}
\label{fig:checkerboard_output}
\end{figure*}

\newpage

\begin{figure*}[ht!]
\centerline{
\includegraphics[scale=0.33]{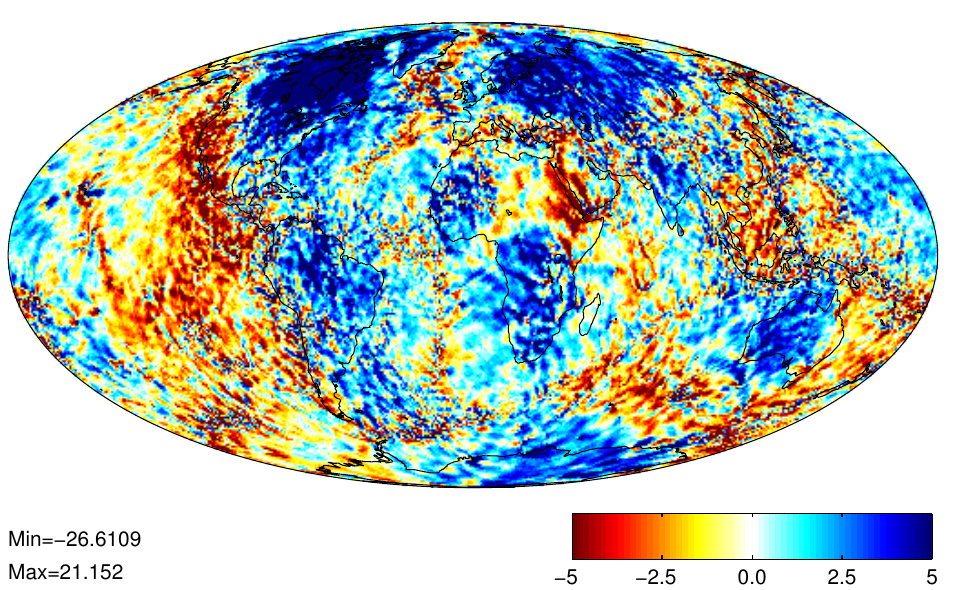}
\includegraphics[scale=0.33]{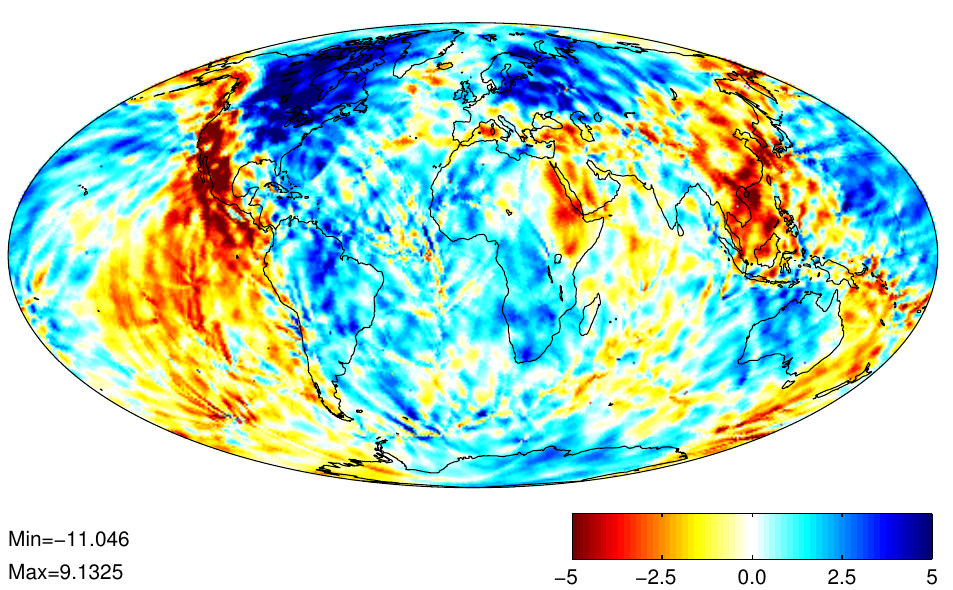}
\includegraphics[scale=0.33]{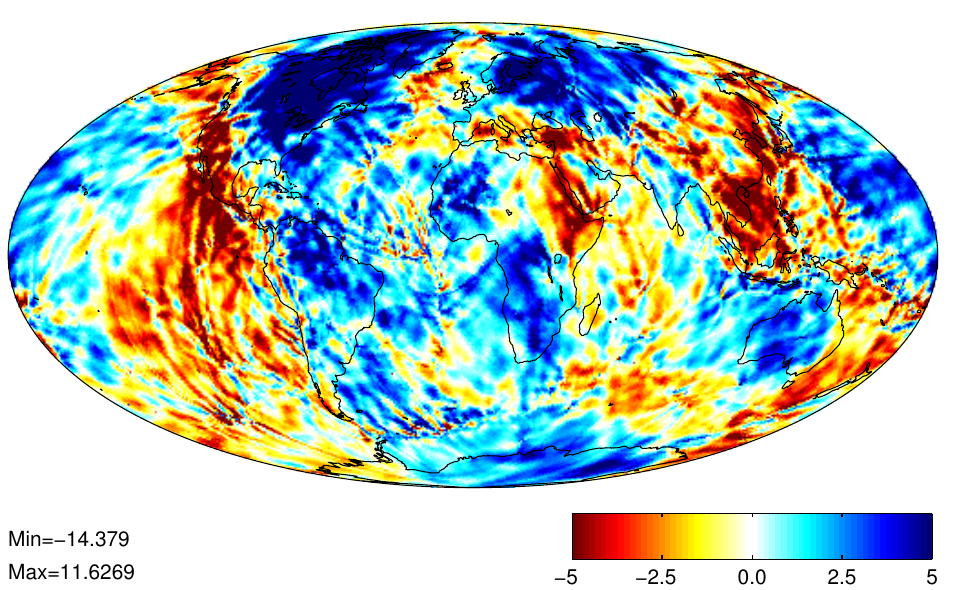}
}
\centerline{
\includegraphics[scale=0.33]{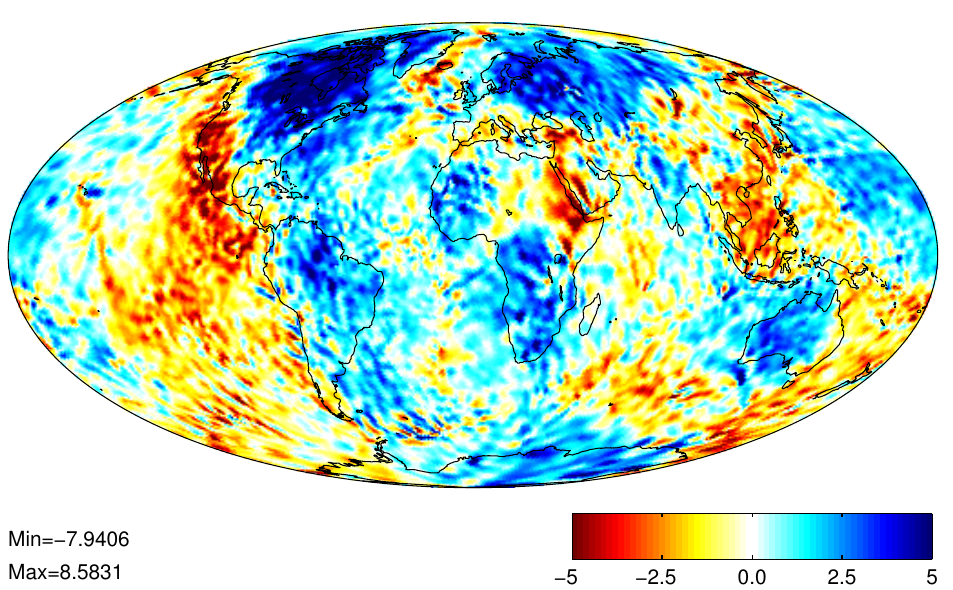}
\includegraphics[scale=0.33]{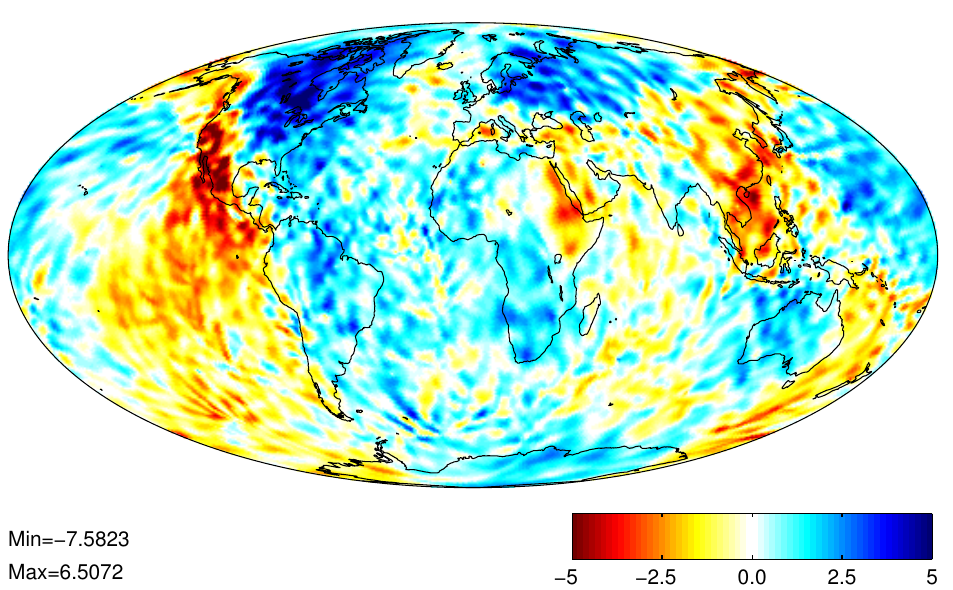}
\includegraphics[scale=0.33]{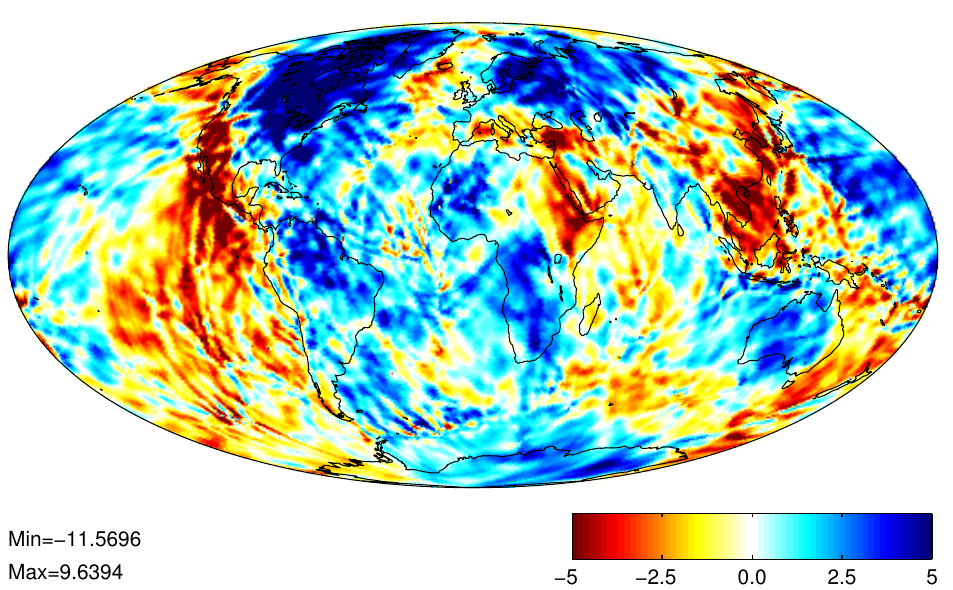}
}
\centerline{
\includegraphics[scale=0.25]{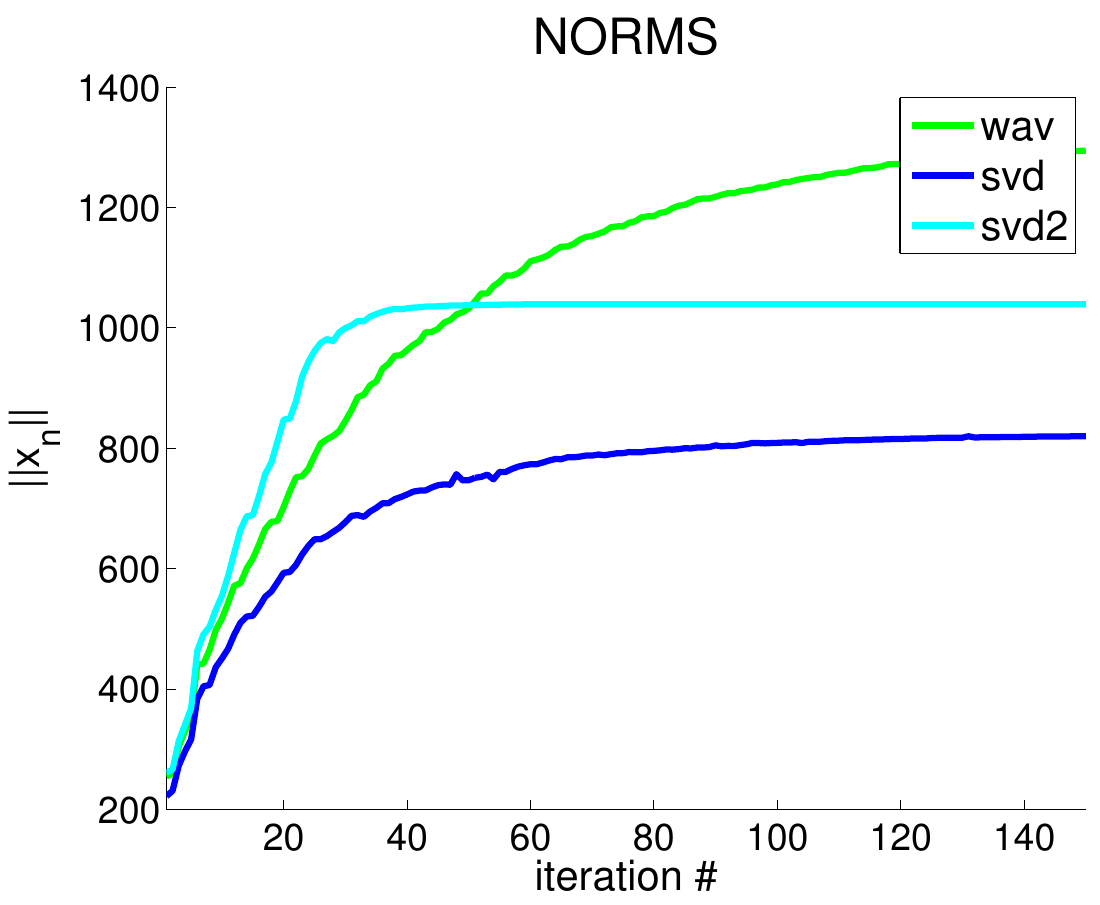}
\includegraphics[scale=0.25]{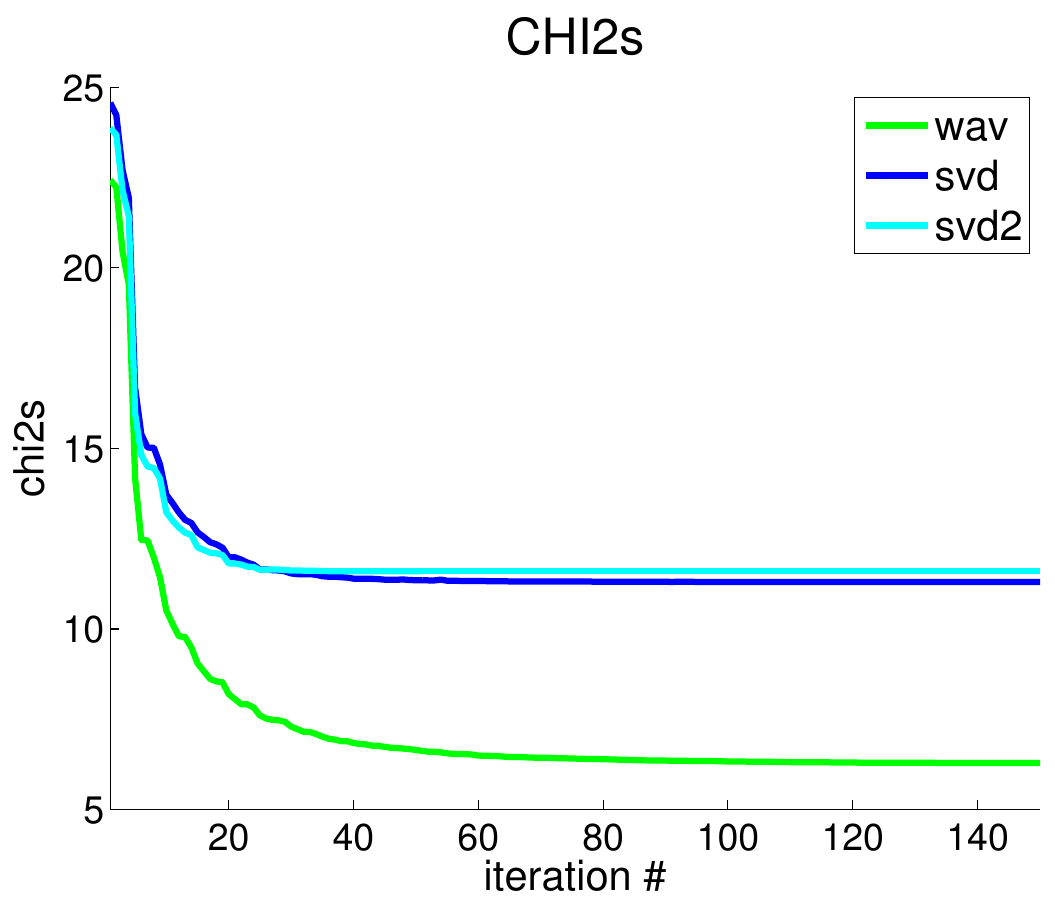}
\includegraphics[scale=0.25]{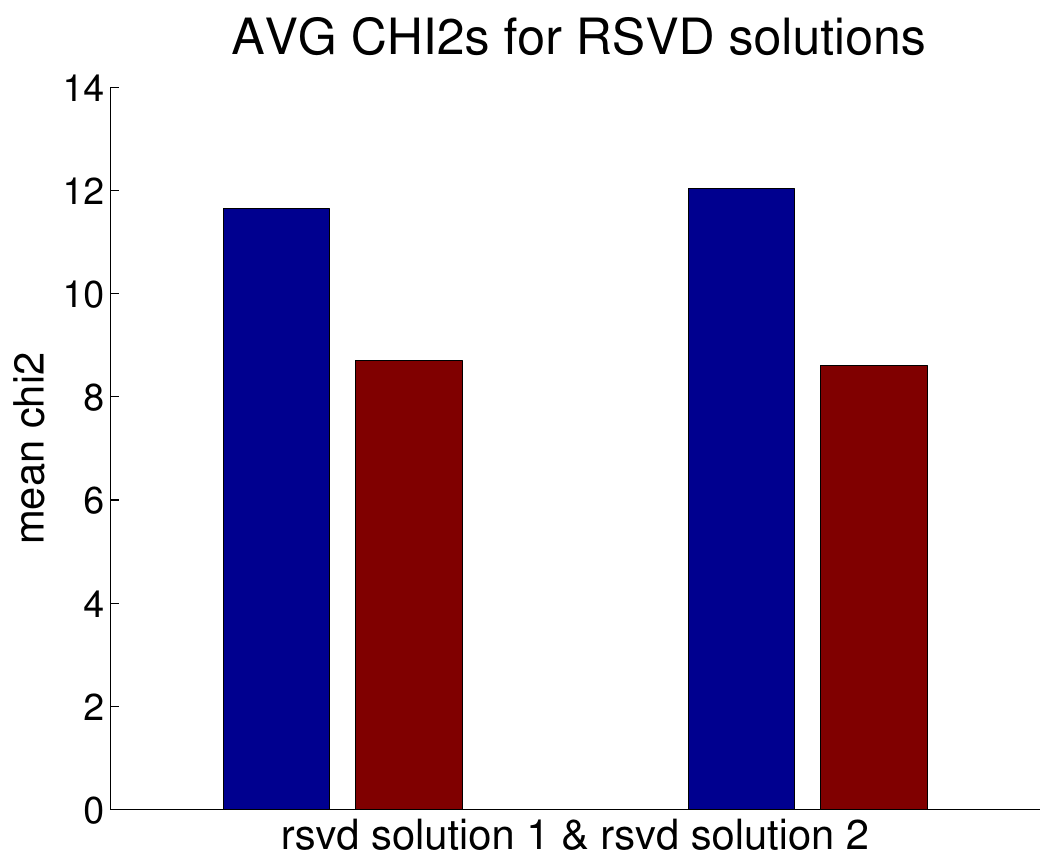}
\includegraphics[scale=0.48]{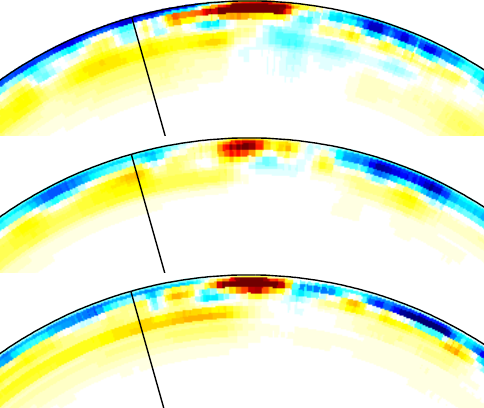}
}
\caption{Plots for regularized solutions $x_5$, $x_6$, $x_7$ (row 1) and $x_8$, 
$x_9$, $x_{10}$ (row 2). First and second row: solutions plotted at $135$ km depth. 
Third row: norms of solution and 
$\chi^2$ value at each iteration, bar plot of average $\chi^2$ of the two SVD 
solutions computed with the low rank SVD matrix and the wavelet compressed matrix, depth 
profiles of solutions $x_8$, $x_9$, and $x_{10}$ in a portion of the globe with variations 
(the top arcs represent the Earth's surface).}
\label{fig:svd_compression_x5_x6_x7}
\end{figure*}

\newpage

\section{Conclusions}
We have presented the use of wavelet compression and low rank SVD techniques for obtaining approximate 
solutions to regularization problems. We illustrate the application 
of these techniques to $\ell_2$ regularization for synthetic data and for a large scale inverse 
problem from seismic tomography, where we show the pros and cons of these approximation methods in a 
practical setting. We have also presented some mathematical analysis for 
the various SVD based schemes we have considered, showing interesting equivalence between 
different schemes with different memory requirements. 
The techniques we present are also well applicable to other types 
of optimization problems. In fact, the methods presented here can be of use to any 
application where matrix-vector operations with large matrices are required, especially if the 
matrices are not well conditioned and have nonlinear decay of singular values. 

The wavelet compressed approach is found to be very accurate and gives 
close reconstructions to the true solution, assuming the data are wavelet compressible. 
Based on our experiments, applications utilizing similar data and wavelet transform can benefit from 
a compression ratio of at least $3$ times, with minimal accuracy loss. In our examples, 
we used a simple one dimensional transform for each row. Recognizing the 
rows as multi-dimensional images and transforming them via a 
multi-dimensional transform would likely give even greater compression.

For large matrices, the compression with wavelets alone 
may not be sufficient. The low rank SVD approach can give significantly 
better compression ratios ($>10$) and resolve the main solution features.  
The low rank SVD can be obtained through an efficient randomized algorithm using 
operations with the smaller wavelet compressed matrix instead of the full matrix, 
so that the two compression techniques we present can be utilized together.
The approaches we discuss lead to the use of $k \times n$ or $k \times k$ 
matrices (which can also be split in several smaller blocks), 
in place of the original $m \times n$ matrix, which can result in 
very substantial compression ratios.  

For both wavelet compressed and low rank SVD based methods, the accuracy and 
compression ratio are inversely proportional and controlled by the user. In the case of 
wavelet compression, the time it takes to form the compressed matrix is nearly independent 
of the threshold used. However, for the computation of the low rank SVD, the work involved 
substantially grows as the rank $k$ increases. Often, 
a checkerboard style test can be performed to see the 
resolution a data set is capable of. In large problems, the resolution possible with a given matrix is 
often limited. The approximation techniques we propose can often be well justified physically, as 
the fine scale details they may remove or smooth out may not be realistically resolved by the 
data set.

\vspace{5.mm}
\section{Acknowledgements}
The authors would like to thank Ignace Loris, Gunnar Martinsson and Frederik Simons 
for very helpful discussion. Support from the 
ERC (Advanced Grant 226837), the Defense Advanced Projects Research Agency 
(contract N66001-13-1-4050) and the National Science Foundation (contracts 1320652 and 0748488) 
is greatly appreciated.


\begin{thebibliography}{}

\end{thebibliography}


\newpage



\begin{thebibliography}{10}

\bibitem{Akansu:1992:MSD:573878}
A.N. Akansu and R.A. Haddad.
\newblock {\em Multiresolution Signal Decomposition: Transforms, Subbands, and
  Wavelets}.
\newblock Academic Press, Inc., Orlando, FL, USA, 1992.

\bibitem{Calvetti2000423}
D.~Calvetti, S.~Morigi, L.~Reichel, and F.~Sgallari.
\newblock Tikhonov regularization and the {$L$}-curve for large discrete
  ill-posed problems.
\newblock {\em J. Comput. Appl. Math.}, 123(1-2):423--446, 2000.
\newblock Numerical analysis 2000, Vol. III. Linear algebra.

\bibitem{Charlety2013}
J.~Ch\'{a}rlety, S.~Voronin, Nolet G., I.~Loris, F.J. Simons, K.~Sigloch, and
  I.C. Daubechies.
\newblock Global seismic tomography with sparsity constraints: Comparison with
  smoothing and damping regularization.
\newblock {\em Journal of Geophysical Research - Solid Earth}, 2013.

\bibitem{CDFWavelets}
A.~Cohen, I.C. Daubechies, and J.-C. Feauveau.
\newblock Biorthogonal bases of compactly supported wavelets.
\newblock {\em Communications on Pure and Applied Mathematics}, 45(5):485--560,
  1992.

\bibitem{DaubechiesWaveletsI}
I.C. Daubechies.
\newblock Orthonormal bases of compactly supported wavelets.
\newblock {\em Communications in Pure and Applied Mathematics}, 41:909--996,
  1988.

\bibitem{ingrid_thresholding1}
I.C. Daubechies, M.~Defrise, and C.~De~Mol.
\newblock An iterative thresholding algorithm for linear inverse problems with
  a sparsity constraint.
\newblock {\em Communications on Pure and Applied Mathematics},
  57(11):1413--1457, 2004.

\bibitem{debayle2004}
E.~Debayle and M.~Sambridge.
\newblock {Inversion of massive surface wave data sets: Model construction and
  resolution assessment}.
\newblock {\em Journal of Geophysical Research}, 109:B02316, 2004.

\bibitem{Halko:2011:FSR:2078879.2078881}
N.~Halko, P.G. Martinsson, and J.A. Tropp.
\newblock Finding structure with randomness: Probabilistic algorithms for
  constructing approximate matrix decompositions.
\newblock {\em SIAM Rev.}, 53(2):217--288, May 2011.

\bibitem{har-etal:wavelets}
W.~H{\"a}rdle, G.~Kerkyacharian, D.~Picard, and A.~Tsybokov.
\newblock {\em Wavelets, Approximation, and Statistical Applications}, volume
  129 of {\em Lecture Notes in Statistics}.
\newblock Springer-Verlag, New York, 1998.

\bibitem{Lampe20122845}
J.~Lampe, L.~Reichel, and H.~Voss.
\newblock Large-scale tikhonov regularization via reduction by orthogonal
  projection.
\newblock {\em Linear Algebra and its Applications}, 436(8):2845 -- 2865, 2012.
\newblock Special Issue dedicated to Danny Sorensen's 65th birthday.

\bibitem{MarkovskyLowRank}
I.~Markovsky.
\newblock {\em Low Rank Approximation: Algorithms, Implementation,
  Applications}.
\newblock Communications and Control Engineering. Springer, 2012.

\bibitem{GJI:GJI426}
H.~Marquering, G.~Nolet, and F.A. Dahlen.
\newblock Three-dimensional waveform sensitivity kernels.
\newblock {\em Geophysical Journal International}, 132(3):521--534, 1998.

\bibitem{YvesMeyerWaveletsandAlgs}
Y.~Meyer.
\newblock {\em Wavelets: Algorithms \& Applications}.
\newblock Society for Industrial and Applied Mathematics, Philadelphia, 1993.
\newblock Translated and revised by Robert D. Ryan.

\bibitem{nolet08}
G.~Nolet.
\newblock {\em A Breviary of Seismic Tomography}.
\newblock Cambridge Univ. Press, Cambridge, U.K., 2008.

\bibitem{Paige:1982:LAS:355984.355989}
C.C. Paige and M.A. Saunders.
\newblock Lsqr: An algorithm for sparse linear equations and sparse least
  squares.
\newblock {\em ACM Trans. Math. Softw.}, 8(1):43--71, March 1982.

\bibitem{Ronchi199693}
C.~Ronchi, R.~Iacono, and P.S. Paolucci.
\newblock The “cubed sphere”: A new method for the solution of partial
  differential equations in spherical geometry.
\newblock {\em Journal of Computational Physics}, 124(1):93 -- 114, 1996.

\bibitem{Simons.Loris.ea2011}
F.J. Simons, I.~Loris, G.~Nolet, I.C. Daubechies, S.~Voronin, J.~S. Judd, P.A.
  Vetter, J.~Ch\'{a}rlety, and C.~Vonesch.
\newblock Solving or resolving global tomographic models with spherical
  wavelets, and the scale and sparsity of seismic heterogeneity.
\newblock {\em Geophysical Journal International}, 187(2):969--988, 2011.

\bibitem{swe:spie95}
W.~Sweldens.
\newblock The lifting scheme: {A} new philosophy in biorthogonal wavelet
  constructions.
\newblock In Andrew~F. Laine, Michael~A. Unser, and Mladen~V. Wickerhauser,
  editors, {\em Wavelet applications in signal and image processing {III}},
  volume 2569 of {\em Proceedings of SPIE}, pages 68--79, 1995.

\bibitem{Tikhonov63}
A.N. Tikhonov.
\newblock Solution of incorrectly formulated problems and the regularization
  method.
\newblock {\em Soviet Math. Dokl}, 1963.

\bibitem{trefethen97}
Lloyd~N. Trefethen and David Bau, III.
\newblock {\em Numerical linear algebra}.
\newblock Society for Industrial and Applied Mathematics (SIAM), Philadelphia,
  PA, 1997.

\bibitem{vanheijst99}
H.-J. van Heijst and J.H. Woodhouse.
\newblock Global high-resolution phase velocity distributions of overtone and
  fundamental mode surface waves determined by mode branch stripping.
\newblock {\em Geophysical Journal International}, 137(3):601--620, 1999.

\bibitem{2015arXiv150205366V}
S.~{Voronin} and P.-G. {Martinsson}.
\newblock {RSVDPACK: Subroutines for computing partial singular value
  decompositions via randomized sampling on single core, multi core, and GPU
  architectures}.
\newblock {\em ArXiv e-prints}, February 2015.

\bibitem{ImprovingCURMatrixDecomps}
S.~Wang and Z.~Zhang.
\newblock Improving cur matrix decomposition and the nystrom approximation via
  adaptive sampling.
\newblock {\em J. Mach. Learn. Res.}, 14(1):2729--2769, January 2013.

\bibitem{Woodbury50}
M.A. Woodbury.
\newblock {\em Inverting modified matrices}.
\newblock Statistical Research Group, Memo. Rep. no. 42. Princeton University,
  Princeton, N. J., 1950.

\end{thebibliography}
\end{document}